\documentclass{amsart}
\usepackage{graphicx,amssymb,xcolor,enumerate}
\usepackage{hyperref}
\usepackage{geometry}
\usepackage{amssymb}
\usepackage{graphicx}
\usepackage{color}
\usepackage{enumerate}
\usepackage{stmaryrd}
\usepackage{hyperref}
\hypersetup{
	bookmarks=true,         
	pdffitwindow=false,     
	pdfstartview={FitH},    
	colorlinks=true,      
	citecolor=red,
}
\usepackage{tikz}

\newtheorem{thm}{Theorem}[section]
\newtheorem{lem}[thm]{Lemma}
\newtheorem{prop}[thm]{Proposition}
\newtheorem{claim}[thm]{Claim}
\newtheorem{cor}[thm]{Corollary}

\theoremstyle{definition}
\newtheorem{defn}[thm]{Definition}

\theoremstyle{remark}
\newtheorem{rem}[thm]{Remark}

\theoremstyle{Fact}

\theoremstyle{Example}

\newcounter{Ass}

\theoremstyle{definition}
\newtheorem{Assumption}[Ass]{Assumption}

\numberwithin{equation}{section}

\usepackage{setspace}
\usepackage{amssymb}

\begin{document}
	
	

	
	\title[Optimal savings and Value of Population]{Optimal Savings and Value of Population in A Stochastic Environment: Transient Behavior}

	\author[Hao Liu,    Suresh P. Sethi,   Tak Kwong Wong, and  Sheung Chi Phillip Yam]
	{
		Hao Liu$^{1}$,
		Suresh P. Sethi$^{2}$, Tak Kwong Wong$^{3}$,   and Sheung Chi Phillip Yam$^{4}$}
	\thanks{${}^1$ School of Mathematical Sciences and Institute of Natural Sciences, Shanghai Jiao Tong University, Shanghai, China (mathhao.liu@sjtu.edu.cn)}
	\thanks{${}^2$ Jindal School of Management, University of Texas at Dallas, Richardson, TX	75080-3021, USA (sethi@utdallas.edu)}
	\thanks{${}^3$ Department of Mathematics, The University of Hong Kong, Pokfulam, Hong Kong (takkwong@maths.hku.hk)}
	\thanks{${}^4$ Department of Statistics, The Chinese University of Hong Kong, Shatin, Hong Kong (scpyam@sta.cuhk.edu.hk)}

	\maketitle
	\textit{We dedicate this paper to the memory of Kenneth J. Arrow;  it extends the research he initiated and pursued further with a coauthor of this paper.}

	\textbf{Abstract}%
	
	\selectfont
	We  extend the work on optimal investment and consumption of a population considered in \cite{arrow2007optimal}
	to a general stochastic setting over a finite time horizon. 
	We  incorporate the Cobb-Douglas production function    in the  capital dynamics while the consumption utility function and the drift rate in the population dynamics can be general, 
	in contrast with   \cite{arrow2007optimal,morimoto2008optimal2,morimoto2008optimal}. 
	The dynamic programming formulation yields an unconventional nonlinear Hamilton-Jacobi-Bellman (HJB) equation, in which the Cobb-Douglas production function as the coefficient of the gradient of the value function induces the mismatching of power rates between capital and population. Moreover, the equation has a very singular term, essentially a very negative power  of the partial derivative of the value function with respect to the capital, coming from the optimization of control, and their resolution turns out to be a complex problem not amenable to classical analysis.
	To show that this singular term, which has not been studied in any physical systems yet, does not actually blow up, we establish new pointwise generalized power laws for the partial derivative of the value function. Our contribution lies in providing a theoretical treatment that combines both the probabilistic approach and theory of partial differential equations to derive the pointwise upper and lower bounds as well as energy estimates in weighted Sobolev spaces. By then, we  accomplish showing the well-posedness of classical solutions to a non-canonical parabolic equation arising from a long-lasting problem in macroeconomics.

Unlike the stationary problem  in \cite{arrow2007optimal}, the intricate transient behavior is only evident over finite time horizons;
particularly, if the production function $F$ is  dominant, the capital will be steadily built and then reach  an equilibrium level; otherwise, with a weak influence of $F$, the capital will gradually shrink to a lower level. Further, due to the nonlinearity of the system, there could be a phase transition at certain critical parameter thresholds, under which an overshooting (undershooting, resp.) of the capital dynamics can be observed when the initial population is larger (smaller, resp.) than the equilibrium level. However, the presence of large enough volatility in the capital investment detours us from determining what the exact  transient mechanism of the  underlying economy is.


%




%
{
	\hypersetup{linkcolor=blue}
	\tableofcontents
}
\section{Introduction }\label{sec:Introduction}


\subsection{Literature review }\label{sec:Literature review}
In 1928, Frank Ramsey \cite{ramsey1928mathematical} introduced the dynamic problem of optimal savings.
He used the calculus of variations to derive the general principle of maximizing society’s enjoyment by equating the product of the savings rate and the marginal utility of consumption with the bliss point (the maximum obtainable enjoyment rate minus the actual total utility rate). 
Ever since, economists and financial mathematicians have extensively explored various versions of the optimal savings problem; for instances, various related works can be founded in \cite{bank2001optimal,  bensoussan2009optimal, cuoco1997optimal,  deshmukh1983optimal, fleming1991optimal, grandits2016optimal,  kraft2011optimal, lehoczky1983optimal, levy1976multi,     liu2004optimal,  mendelson1982optimal, shreve1994optimal,  xu1992duality, yao2004optimal}.

Following Arrow et al. \cite{arrow2007optimal, arrow2003genuine},  we aim to maximize the total utility of consumption  by a randomly growing population over a finite time horizon and of its terminal wealth in a stochastic economic environment. The vital difference between our model and the existing literature is that both the capital stock and the population size, the states in our model, are subject to stochastic dynamics. As will become clear later, such presence of randomness has subtle implications absent in any deterministic framework.



In the economic growth theory, the society’s capital stock represents the total value of assets used to produce goods and services. These assets produce output to be allocated between consumption and investment. Any individual in the society derives utility from his own consumption, and   investment generates more capital to help sustain future economic growth. As Ramsey \cite{ramsey1928mathematical} pointed out in a deterministic framework, sensible economic planning regarding consumption and investment is in striving for a balance between current and future welfare.

How does  a society plan for its economic future? Arrow et al.  \cite{arrow2003genuine} studied this question 
by formulating an optimal control problem to determine whether and to what extent the optimal consumption policy results in the gain in aggregate welfare.
The population plays two roles. One, it  represents
the working labor force that goes into the production function  along with the capital. Two, the individuals in the population consume and derive utility from the capital acquired.
The concept of total utilitarianism considers  the society’s utility as the product 
of per capita consumption and  the size
of the population. This  concept dates back to the work of  Henry Sidgwick and Francis Edgeworth
in the 1870s \cite{sidgwick2019methods,edgeworth1877new} and  has been echoed by several subsequent works including  Meade \cite{meade1962trade} and
Mirrlees \cite{MirrleesOptimumGrowth}. 

More importantly, Arrow et al. \cite{arrow2003genuine}  also jettisoned the standard exponential growth assumption for the population as unrealistic, given the dramatic reduction in birth rates worldwide. 
The assumption has been a hallmark of the classical economic growth literature mainly because  it helps to eliminate the population size from being a state variable.
The underlying optimal control problem is greatly simplified with capital stock as the only state variable.
Arrow et al. \cite{arrow2007optimal} considered arbitrary population growth until it reached a critical level or  exponential growth until it reached a saturation level, after which there was no more growth.
By letting the population variable serve as the time surrogate, they depicted the optimal path and  its convergence to the long-run equilibrium on a two-dimensional phase diagram. The phase diagram consists of a transient curve that reaches the
classical curve associated with a positive exponential growth when the population reaches the critical mass. In the case of an asymptotic population saturation, the transient curve approaches the equilibrium  as the population tends to its saturation level. The authors also characterized the transient approach to the classical curve and to the equilibrium.

Morimoto and Zhou \cite{morimoto2008optimal} used the Cobb-Douglas production function to study a finite time horizon optimal consumption problem. They used a stochastic differential equation (SDE) to model the population dynamics and an ordinary differential equation (ODE) to model the capital dynamics.  
The resulting equation is of second-order parabolic type and  the spatial dimensions of the equation could be reduced to one  due to the special form of the population drift. Under a boundedness condition on the consumption process, they proved that a unique solution exists using viscosity techniques under a sign condition on the model parameters. They also synthesized the optimal consumption policy in the feedback form. The same problem with the constant elasticity of substitution (CES) production function was studied in 
\cite{adachi2009optimal, morimoto2008optimal2}. 
Huang and Khalili \cite{huang2019optimal} removed the boundedness condition on the consumption process and proved the existence and uniqueness of the solution to the corresponding HJB equation under the same sign condition on the parameters as in
\cite{morimoto2008optimal}.

Here, we consider a more realistic economic growth model where capital and population dynamics can interact on a finite time interval. Moreover, we use SDEs to model both the capital and population dynamics. Relating to \cite{morimoto2008optimal}, it is fair to say that modeling capital by an SDE is more realistic and meaningful since, in a short time, capital increment is noisier and fluctuates more frequently compared to labor force. 
Perhaps, this formulation has not been adopted in the literature due to the difficulty of the required analysis. 
We also consider a more general function for the drift rate of the population dynamics, representing a generalization of the Ornstein–Uhlenbeck process. 
As for the objective function, we consider a more general, both  intertemporal and terminal, utility functions over a finite time horizon, and we only impose some natural growth conditions on them. 
We use dynamic programming and see that the corresponding HJB equation (cf. \eqref{eq:second HJB for v(K,N)} below) is parabolic. 
Unlike \cite{morimoto2008optimal} and \cite{huang2019optimal}, our spatial domain remains two-dimensional and it cannot be reduced to one-dimensional because of the general form of the  drift rate of population growth. Moreover, we are able to dispense with the sign condition on the model parameters used in the papers cited above.

The main  ingredient in  solving the HJB equation \eqref{eq:second HJB for v(K,N)}  is establishing enough a priori estimates along with the using of
analytical tools such as weighted Sobolev spaces, the Feynman-Kac formula, and Schauder’s fixed point theorem. Compared to \cite{arrow2007optimal}, the optimal capital and population paths exhibit some oscillating behavior because of the associated random noise terms. Notably, the random noise enables us to cover different optimal paths of capital with essentially different behavior, which we call the transient behavior of the capital. However, this phenomenon also depends on the power of the production function, which plays a vital role in the optimal path of capital.


\subsection{Model formulation and  HJB equation}\label{sec:Model formulation}
This section formulates the problem and defines relevant notations. Our objective is to maximize the value of society’s utility of consumption over a finite time horizon  $[0,T]$ while we treat the consumption per head as the control variable. The two state variables are the capital stock $K(t)$  of the whole community and the population $N (t)$, and we model them by the following stochastic differential equations: for $0\leq t \leq T$,
\begin{equation}\label{eq:SDE for K}
	\left\{ 
	\begin{aligned}
		dK(t) &= \left[F(K(t),N(t)) -N(t)c(t)\right]dt+\epsilon K(t)dW_K(t),\\
		K(0)&=K>0,
	\end{aligned}
	\right.
\end{equation}
\begin{equation}\label{eq:SDE for N}
	\left\{
	\begin{aligned}
		dN(t) &=  f(N(t))dt + {\sigma} N(t) dW_N(t), \\
		N(0) &=N>0,
	\end{aligned}
	\right.
\end{equation}
where $\epsilon$ and ${\sigma}$ are two positive constants, 
$W_K$ and $W_N$ are two independent Brownian motions on a complete probability space ($\mathit{\Omega}$, $\mathcal{F}$, $\mathcal{F}_t$, $\mathbb{P}$),  while the filtration $\mathcal{F}_t$ being generated jointly by $W_K$ and $W_N$.
We  use $F(K,N)$ to denote the production function, which represents the goods produced by the society using capital stock $K$ and labor force $N$, and control $c(t)$ represents per capita consumption.
We aim to maximize  society’s expected utility of consumption
\begin{equation}\label{eq:optimal control problem}
	J(c) = \mathbb{E}\left[\int_{0}^{T}N(t)u_1(c(t))dt+N(T)u_2\left(\frac{K(T)}{N(T)}\right)\right],
\end{equation}
where $u_1(\cdot)$ is the running utility function of each individual’s consumption and $u_2(\cdot)$ is the individual’s terminal utility of per capita capital.
We  require $c(t)\ge 0$ and  progressively measurable with respect to the filtration $\mathcal{F}_t$.
\begin{rem}
	We do not discount the utilities over time for convenience in notation, given that we are considering a finite horizon problem. It does not alter the underlying mathematical structure much, and we can easily extend our analysis to allow for discounting. 
\end{rem}
We give the following definition:
\begin{defn}\label{def:definition for value function}
	The value function $v(t,K,N)$ -- a function of time $t$, capital stock $K$ and population $N$ is defined as
	\begin{equation}\label{eq:definition for value function}
		v (t, K,N) = \sup_{c \in\mathcal{A}_t} J(t,c):=\sup_{c \in\mathcal{A}_t} \mathbb{E}\left[\int_{t}^{T}N(s)u_1(c(s))dt+N(T)u_2\left(\frac{K(T)}{N(T)}\right)\right],
	\end{equation}
	where $\mathcal{A}_t$, the class of all admissible consumption policies is defined as:
	\begin{equation}\label{def: def of dmissible consumption policies}
		\begin{aligned}
			\mathcal{A}_t:&=\left.\Bigg{\{}c: \mathit{\Omega} \times[t, T] \to \mathbb{R}_+ |~c(s) ~ \textrm{is $\mathcal{F}_s$-progressively measurable}\right.\\
			&\qquad  \left. \textrm{and}~ \int_{t}^{r} c(\omega, s) \frac{N(\omega, s)}{K(\omega, s)}ds < \infty, ~\forall r \in [t,T] ~ ~\textrm{for}~ \mathbb{P}-\text{a.s.}~\omega \in \mathit{\Omega}  \right.\Bigg{\}},
		\end{aligned}
	\end{equation}
	with $K(\omega, s)$ and $N(\omega, s)$ ($t\leq s \leq T$) are characterized 
	by the SDEs \eqref{eq:SDE for K} and \eqref{eq:SDE for N} with the corresponding initials  changed to $K(t)=K>0$ and $N(t)=N>0$.
\end{defn}
Using  the standard Dynamic Programming Principle (DPP), we can derive the corresponding Hamilton-Jacobi-Bellman (HJB) equation for the value function $v$: 
\begin{equation}\label{eq:HJB for v(K,N)}
	\left\{
	\begin{aligned}
		&v_t + F(K,N)v_K + f(N)v_N + \frac{1}{2} {\epsilon}^2K^2v_{KK}+\frac{1}{2}{{\sigma}^2 }N^2v_{NN}\\
		&\quad+ \sup_{\substack {c\ge 0} }\left\{-Ncv_K+Nu_1(c)\right\}=0, ~K, N>0, ~\textrm{for}~ 0\leq t<T,\\
		&v(T) = Nu_2\left(\frac{K}{N}\right).
	\end{aligned}
	\right.
\end{equation}
The first-order condition gives $u'_1(\hat{c})$ = $v_K$, where the optimal control $\hat{c}:=\text{arg} \sup_{\substack {c\ge 0} }\left\{-Ncv_K+Nu_1(c)\right\}$.
Consequently, we have $\hat{c}(\omega,t)=(u'_1)^{-1}(v_K(t, K(\omega), N(\omega)))$.
Therefore,   we can rewrite \eqref{eq:HJB for v(K,N)}  as
\begin{equation}\label{eq:second HJB for v(K,N)}
	\left\{
	\begin{aligned}
		&0 = v_t + Nu_1\left((u'_1)^{-1}(v_K)\right) -Nv_K\cdot(u'_1)^{-1}(v_K) + F(K,N)v_K
		+f(N)v_N \\
		&\quad +\frac{1}{2}{\epsilon}^2K^2v_{KK} + \frac{1}{2}{\sigma}^2N^2v_{NN},\\
		&v(T, K, N) = Nu_2\left(\frac{K}{N}\right).\\
	\end{aligned}
	\right.
\end{equation}
It is known that   solving the HJB equation  is not always equivalent to solving the original optimal control problem unless its solution is classical, in which case, the solution gives the value function of the problem. 
Here, the traditional approach does not help, and we must thus develop  a novel method to accomplish the task.
Next, we impose   technical assumptions on the given functions $F$, $f$, $u_1$ and $u_2$ with their economic interpretations.

\begin{Assumption}[Cobb-Douglas  production function]\label{Assum:production function}
	The production function $F$ is of the form
	\begin{equation}\label{eq:production function}
		F(K,N) = AK^{\beta}N^{1-\beta},
	\end{equation}
	where $A$ is a positive constant and $\beta$ is a constant satisfying $0<\beta<1$. 
	
	\begin{rem}
		The Cobb-Douglas production function was first proposed and statistically tested  by Charles Cobb and Paul Douglas during 1927-1947; see \cite{cobb1928theory} for details. This concise and useful production function has been widely adopted;  see  \cite{douglas1976cobb}.
	\end{rem}
	
\end{Assumption}

\begin{Assumption}[Lipschitz  drift in population dynamics]\label{Assum:drift rate in population model}
	Assume  $f \in C^{1}[0,\infty)$, and there is a positive constant $C_f$ such that 
	\begin{equation*}
		f(0)=0,  \quad \textrm{and} \quad |f(x)-f(y)|\le C_{f}|x-y|, ~\forall~ x, y \in [0,\infty).
	\end{equation*}
\end{Assumption}

\begin{rem}\label{rem: remark on the drift rate f}
	It is natural that $f(0)=0$; the growth rate of zero population is zero.
	As we only require $f$ to be Lipschitz continuous, any well-chosen   
	$f$ can  model the general phenomenon that a population  increases rapidly at the beginning when the people in need are scarce, and 
	approaches  its saturation level due to its ultimate limit.
	An example of such an $f$ is
	\begin{equation}\label{eq: O-U like function}
		f(x)=
		\begin{cases}
			\alpha x(\bar{N}-x)\qquad \qquad \textrm{for $x$ $\le$ $M$};\\  \alpha M^2 + \alpha(\bar{N}-2M)x\quad\textrm{for $x> M$},
		\end{cases}
	\end{equation}
	where $\alpha$, $M$, and $\bar{N}$ are positive constants and $M\ge \bar{N}$.
	This function has the same effect as the logistic growth model (see Chapter 6.1 in \cite{edelstein2005mathematical})  but has a slower convergence rate to the equilibrium state.
	In contrast,  \cite{morimoto2008optimal} and \cite{huang2019optimal} assume  $f(N) = \nu N$ for some constant $ \nu$, which must satisfy  the sign condition $\nu -\sigma^2>0$ to make their analysis viable.  Our approach extends their case even when $\nu -\sigma^2\leq0$. 
\end{rem}

\begin{Assumption}[Running utility function $u_1$]\label{Assum: running utility function $u_1$}
	The following conditions are satisfied:
	\begin{enumerate}
		\item
		(Inada condition) 	$u_1\in C[0,\infty) \cap C^2(0,\infty)$, $u'_1(x)>0$, $u''_1(x)<0$, $u'_1(0+)=u_1(\infty)=\infty$, \textrm{and} $u'_1(\infty)=u_1(0)=0$;
		\item
		(Risk aversion condition) $0< u'_1(x)\le {a_1^{\gamma}}x^{-\gamma}$ for some $0<\gamma<1$;
		\item(Minimal level of risk aversion) $0< -x\frac{1}{u_1''((u'_1)^{-1}(x) )}\le \tilde{a}_1(u'_1)^{-1}(x)+ \tilde{\tilde{a}}_1$,
	\end{enumerate}
	for some  positive constants $a_1$, $\tilde{a}_1$, and $\tilde{\tilde{a}}_1$.
\end{Assumption}

\begin{Assumption}[Terminal utility function $u_2$]\label{Assum: terminal utility function $u_2$}	The following conditions holds:
	\begin{enumerate}
		\item (Smoothness of $u_2$)
		$u_2\in C[0,\infty) \cap C^3(0,\infty)$ and $u_2(0)=0$;
		\item (Risk-aversion condition)
		$a_2x^{-\gamma}\le u'_2(x)\le a_3x^{-\gamma}$;
		\item (Boundedness of relative risk aversion) $\beta \wedge \gamma\le -x\frac{u''_2(x)}{u'_2(x)}\le \beta\vee\gamma  $.  Additionally, $ \inf _{x>0}\frac{\tilde{u}_2' }{\tilde{u}_2}(x) >\beta$ when $\beta <\gamma$, where $\tilde{u}_2(x):=u'_2(\frac{1}{x})$;
		\item (Higher order derivative technical assumption) $a_4{u}'_2(x)\le x^2 u'''_2(x) \le a_5u'_2(x)$,
	\end{enumerate}
	for some positive constants $a_2$, $a_3$, $a_4$, and $a_5$.
		\begin{rem}
			Condition (3) in our Assumption \ref{Assum: terminal utility function $u_2$} gets its name since it is equivalent to   the inequality $\min\{\beta, \gamma \}  \le \textrm{RRA}_2(x) \le \max\{\beta, \gamma \} $. We recall that the	RRA of a utility function $u$ is the Arrow-Pratt measure of its relative risk-aversion   defined as 
			RRA: = $-\frac{xu''(x)}{u'(x)}$.
			Here, $\textrm{RRA}_2(x) $ means the RRA of $u_2$.
			We also note  the RRA of a  power function $u = \frac{x^{1-\gamma}}{1-\gamma}+e$ is the constant $\gamma$.
		\end{rem}
		
		\begin{rem}
			Note that the commonly used constant relative risk aversion utilities $u_1(x)$ and $u_2(x)$ of the form $cx^{1-\gamma}$ satisfy Assumption \ref{Assum: running utility function $u_1$}
			and Assumption \ref{Assum: terminal utility function $u_2$}, respectively.	The second condition in \ref{Assum: terminal utility function $u_2$} ensures that
			\begin{equation*}
				u_2(x) = \int_{0}^{x} u'_2(t)dt + u_2(0)\le \int_{0}^{x} a_3t^{-\gamma} dt = \frac{a_3}{1-\gamma}x^{1-\gamma},
			\end{equation*}
			and	hence,
			\begin{equation}\label{eq:production is 0 when population is 0}
				\begin{aligned}
					&\lim_{N\to 0} Nu_2\left(\frac{K}{N}\right)\le 	\lim_{N\to 0}  \frac{a_3}{1-\gamma}N\left(\frac{K}{N}\right)^{1-\gamma} =  	\lim_{N\to 0} \frac{a_3}{1-\gamma}N^{\gamma}K^{1-\gamma}=0,\\
					&\lim_{K\to 0} Nu_2\left(\frac{K}{N}\right) = Nu_2(0) =0.
				\end{aligned}
			\end{equation}
			Equation \eqref{eq:production is 0 when population is 0}  agrees with the intuition that when either population or capital vanishes, the society's terminal utility drops to 0.
		\end{rem}
	\end{Assumption}
	
	\subsection{Matters arising from tackling   HJB equation \eqref{eq:second HJB for v(K,N)}}\label{sec:Main resluts}
	We note the HJB equation \eqref{eq:second HJB for v(K,N)} is a nonlinear degenerate second-order parabolic PDE on the unbounded domain $\mathbb{R}_+^2 := (\mathbb{R}_+)^2 := \{(K,N); K, N >0\}$ with unbounded coefficients. Due to the unboundedness of the spatial domain $\mathbb{R}_+^2$ and the degeneracy of the underlying elliptic operator of \eqref{eq:second HJB for v(K,N)} at the spatial boundaries $K=0$ and $N=0$, 
	we have to grasp the asymptotic behavior (e.g., upper bounds) of solutions and their derivatives as $N$, $K\to 0^+$ or $+\infty$ so that a unique solution with  a sound economic meaning can be obtained. 
	Moreover,  the nonlinearity arising from the running utility $u_1$ in \eqref{eq:second HJB for v(K,N)} behaves like a negative power of $v_K$ after incorporating the first-order condition, indicating the lower bound on $v_K$  is quite critical for the well-posedness of the equation. 
	Thus both a lower and upper bound for the growth rate for $v_K$  must be sought to establish the solution's very existence. 
	
	Furthermore, the difference between the corresponding exponents of $K$ and $N$ as defined in the Cobb-Douglas production function also brings in additional difficulty, especially when we try to obtain the aforementioned lower and upper bounds and when applying the energy estimates method.
	To see the difficulty from another aspect, we can apply the transformation $x=\ln K$ and $y=\ln N$ to \eqref{eq:second HJB for v(K,N)},  then  the Cobb-Douglas function $F$ and the terms involving $(u'_1)^{-1}(v_K)$ in \eqref{eq:second HJB for v(K,N)} have exponential growth rates with respect to $x$ and $y$ on $\mathbb{R}^2 := \{(x,y); x, y \in \mathbb{R}\}$. 
	Notably, classical theories of parabolic equations, which commonly focus on various physical systems, rarely address equations with such rapid exponential growth rates. So an extra caution has to be paid. One major  novelty of our present work is to  apply the probabilistic techniques to derive appropriate pointwise lower and upper bounds,  using which good, appropriately-weighted energy estimates can be obtained.
	
	To this end, we first show the  unique existence of classical solutions to Equation  \eqref{eq: differential equation for original $lambda$} for $\lambda:=\frac{\partial v}{\partial K}$ by obtaining the optimal 
	upper and lower bounds for $\lambda$. On the one hand, we  must obtain the necessary upper and lower bounds for  $\lambda$ as  explained above. On the other, 
	it is almost equivalent to the   unique existence of a classical solution to the HJB Equation  \eqref{eq:second HJB for v(K,N)}. Once we obtain $\lambda$, then heuristically, we can obtain $v$ by  direct integration with respect to $K$. 
	Of course, the resulting function is well-defined only modulo a function of $N$ and $t$, so to select the correct one, we have to use  the HJB equation \eqref{eq:second HJB for v(K,N)} again. Technically and more precisely, we replace all  $v_K$ by this $\lambda$, which is the unique solution to  \eqref{eq: differential equation for original $lambda$} (that we have already obtained)  in \eqref{eq:second HJB for v(K,N)}, and then solve  $v$ in the resulting linear equation (see  \eqref{eq:third HJB for v(K,N)} for instance). The last thing is to check that we  have the compatible relation $\lambda=\frac{\partial v}{\partial K}$, which can finally be shown by using a specific uniqueness theorem for parabolic equations; see Section \ref{Existence  and uniqueness to the original HJB equation} for details.
	
	Further, we aim to solve the nonlinear problem \eqref{eq: differential equation for original $lambda$} with a proper fixed-point map. Thus, we must  solve its linearized counterpart \eqref{eq: differential equation for $lambda$} first. We  note that the coefficient $(u_1')^{-1}(\psi) = j(\psi) \sim \psi^{-\frac{1}{\gamma}}$ (e.g., the lower bound of $\psi$) affects the wellposedness of PDE \eqref{eq: differential equation for $lambda$} critically; this suggests a certain lower bound on $\psi$ is essential for solving  \eqref{eq: differential equation for $lambda$}.
	Indeed, for a sufficient smooth   $\psi$ with a decent lower bound, the rapid growth of   $j(\psi) \sim \psi^{-\frac{1}{\gamma}}$ can be controlled. By then, we are able to establish the upper bound and unique existence  of the classical solution to \eqref{eq: differential equation for $lambda$}  in a local H\"older space; see Theorem \ref{thm:existence when psi is smooth}.  The next challenge  is  to ensure that  the output $\lambda$ is bounded from below by the same  lower bound of the input $\psi$ so that, a fixed-point argument can be utilized to solve the nonlinear problem \eqref{eq: differential equation for original $lambda$}. This standard approach will not work without our optimal pointwise estimates obtained by using the following probabilistic techniques.
	
	 For the crucial lower bound for the solution $\lambda$ to  the linearized problem \eqref{eq: differential equation for $lambda$}, we  shall prove a stronger result, namely  {\itshape a priori} lower bound for $\lambda$  independent of $\psi$ can be obtained. Particularly in Section \ref{section: pointwise estimate}
	we shall  deduce the pointwise lower bounds  via  the terminal condition and estimates of various  SDEs induced from  Feynman-Kac formulas for \eqref{eq: differential equation for $lambda$}; the  upper bounds can be then obtained based on the crucial lower bounds and the same techniques;
	see Theorems \ref{thm:crucial estimate theorem}, \ref{thm:crucial estimate theorem for case beta<gamma} and their proofs for details. 
	We may ultimately verify the pointwise estimates by  the usual  maximum principle, it can never be an effective method here (cf. Remark \ref{rem:comment on maximum principle}) since one has to correctly guess the corresponding pointwise behavior in advance, which is far from trivial and out of the blue. 
	Technically speaking, the existence of a solution in a local H\"older space guarantees that
	the Feynman-Kac  representation is valid, based on which the crucial pointwise estimates can be derived rigorously. 
	
	We then extend this  existence result to  measurable inputs $\psi$ with the same aforementioned lower bound a.e.  (cf. Lemma \ref{lem:extension lemma}) based on the energy estimates in weighted Sobolev spaces to be obtained in Theorem \ref{thm: A priori estimates} and properties of weights. By then, a fixed-point argument on a properly chosen subset $S$ defined in \eqref{def:definition for the growth rate set} below (involving the pointwise bounds)   of the weighted Sobolev space 
	will give the solution to the nonlinear problem \eqref{eq: differential equation for original $lambda$}, which relies on the compactness of the weighted Sobolev spaces crucially. The
	reason we use weighted Sobolev spaces to find a fixed point is that they are much more feasible to deal with  compactness than the weighted H\"older spaces.
	
	The organization  of the remainder of this article  is as follows.
	In Section \ref{sec:Solving the HJB},
	we introduce the most relevant definition of  the  functional spaces in which the first proxy weak solution is constructed.  
	Section \ref{section:Empirical study} provides some numerical experiments to illustrate   some interesting phenomena that are ubiquitous over finite time horizon, which makes the economic prediction so hard in a short time.
	In Section \ref{sec:paraboliceq}, we   prove the unique existence   of a classical solution to  \eqref{eq: differential equation for $lambda$} in a H\"older space when the input  $\psi$  is smooth enough and possesses a specific  lower bound. The  existence and uniqueness are both based on a specialized maximum principle for parabolic equations with possible exponential growth coefficients on the whole space. The solution obtained in this section is sufficiently regular  to use the Feynman-Kac  representation for deriving the lower and upper growth rate estimates in Section \ref{section: pointwise estimate}.
	In Section  \ref{section: pointwise estimate}, we prove the almost optimal pointwise estimates of the solution to the linearized equation \eqref{eq: differential equation for $lambda$} and  verify that   the solution's lower bound  can be  the same  as that of  the input $\psi$.
	We present  the weighted energy estimates for the weak solution of the linearized equation \eqref{eq: differential equation for $lambda$} in Section \ref{sec:energy estimate}. Finally, we return  to the nonlinear problem 
	in Section \ref{sec:Weak solution and fixed point for nonlinear equation}, where we  solve the nonlinear problem  \eqref{eq: differential equation for original $lambda$} by a fixed point argument and the original HJB equation \eqref{eq:second HJB for v(K,N)}.
	To conclude this section, we  state the main theorem  motivating  subsequent sections.
	\begin{thm}\label{main theorem} 
		There exists a  classical solution $v$ to  \eqref{eq:second HJB for v(K,N)}. Moreover,  it is
		the unique one that satisfies the following pointwise estimates: for some positive  constants  $B_1$, ... , $B_9$ and some real (not necessarily positive)  constants $\beta_1$ and  $\gamma_1$ such that
		\begin{equation*}
			B_1\left(\frac{N}{K}\right)^{\gamma} \le \frac{\partial v}{\partial K}(t,K,N) \le B_2 \left(\frac{N}{K}\right)^{\gamma} + B_3  \left(\frac{N}{K}\right)^{\beta} ~\textrm{when} ~ \beta \ge \gamma; ~\textrm{or}~ \qquad \qquad(\ast 1)
		\end{equation*}
		\begin{equation*}
			B_4  \left(\frac{N}{K}\right)^{\beta}\left(  B_5 \left(\frac{N}{K}\right)^{\beta-1} + B_6\right)^{ \frac{\gamma-\beta}{\beta-1 } }\le \frac{\partial v}{\partial K} (t,K,N) \le B_7\left(\frac{N}{K}\right)^{\gamma} ~\textrm{when} ~ \beta < \gamma,  (\ast2)
		\end{equation*}
		and
		\begin{equation*}
			|v(t,K,N)| \le B_8 K^{\beta_1}+ B_8N^{\gamma_1}+B_9. \qquad \qquad\qquad \qquad \qquad \qquad \qquad \qquad \quad \quad \, 	 (\ast 3)
		\end{equation*}
	\end{thm}
	\begin{rem}
		Without using the stochastic techniques, the generalized power laws $(\ast 1)-(\ast 2)$  are hard to be observed, so proving $(\ast 1)-(\ast 2)$  solely by  the classical techniques, such as the maximum principle, is highly non-trivial. Indeed, the non-effectiveness of classical techniques for showing  $(\ast 1)-(\ast 2)$ is anticipated, since the generalized power laws like $(\ast 1)-(\ast 2)$ are rare in physical systems.
	\end{rem}
	Since the solution $v$ we obtained is classical, while $(\ast 1-\ast3)$ are all economically soundly, we can conclude that it indeed solves the original optimal control problem.
	\begin{cor}[Verification Theorem]\label{verification theorem}
		{\itshape The solution  $v$ to equation \eqref{eq:second HJB for v(K,N)} is   the value function of Definition \ref{def:definition for value function}.}
	\end{cor}
	%
	\section{Classical solvability of   the HJB equation}
	\label{sec:Solving the HJB}
	\subsection{The linearized equation for the partial derivative of the  value function}\label{Linearized equation}
	As illustrated in the ``generalized" power laws $(\ast 1)-(\ast 2)$, one may heuristically realize that $\frac{\partial v}{\partial K}$ is a better
	unknown, because of the scaling. More essentially, due
	to the  nonlinearity of equation \eqref{eq:second HJB for v(K,N)} and the necessity to obtain lower and upper bounds of the derivative $\frac{\partial v}{\partial K}$, we shall  deal with $\lambda:=\frac{\partial v}{\partial K}$ first, which is more effective; and  it satisfies the following  equation:
	\begin{equation}\label{eq: differential equation for original $lambda $ unsimplified} 
		\left\{
		\begin{aligned}
			&{\lambda}_t +Nu'_1((u'_1)^{-1}(v_K))\cdot((u'_1)^{-1})'(v_K)\cdot v_{KK}-N(u'_1)^{-1}(v_K)\cdot v_{KK}\\
			&\quad	-Nv_K\cdot((u'_1)^{-1})'(v_K) \cdot v_{KK} 
			+  A N^{1-\beta}K^{\beta}\frac{\partial \lambda}{\partial K} +{\epsilon}^2K\frac{\partial \lambda}{\partial K} + \frac{1}{2}{\epsilon}^2K^2\frac{\partial^2 \lambda}{\partial K^2} \\
			&\quad + f(N)\frac{\partial \lambda}{\partial N} +  \frac{1}{2}{{\sigma}}^2N^2 \frac{\partial^2 \lambda}{\partial N^2} + A\beta N^{1-\beta}K^{\beta-1}\lambda= 0,\\
			& \lambda(T,K,N) = u'_2\left(\frac{K}{N}\right).
		\end{aligned}
		\right.
	\end{equation}
	The Inada condition on $u_1$ implies that $u'_1(x)$ is strictly decreasing in $x$, hence its inverse $j(\cdot):=(u'_1)^{-1}(\cdot)$ is well-defined and strictly decreasing, which leads to the following equality:
	\begin{equation*}
		\begin{aligned}
			&\quad Nu'_1((u'_1)^{-1}(v_K))((u'_1)^{-1})'(v_K)v_{KK}-N(u'_1)^{-1}(v_K)v_{KK}-Nv_K((u'_1)^{-1})'(v_K)v_{KK}=-Nj(v_K)v_{KK}.
		\end{aligned}
	\end{equation*}
	With this result, we can simplify \eqref{eq: differential equation for original $lambda $ unsimplified}  further  as follows: 
	\begin{equation}\label{eq: differential equation for original $lambda$} 
		\left\{
		\begin{aligned}
			&{\lambda}_t - Nj(\lambda)\frac{\partial \lambda}{\partial K} + A N^{1-\beta}K^{\beta}\frac{\partial \lambda}{\partial K} +{\epsilon}^2K\frac{\partial \lambda}{\partial K} + \frac{1}{2}{\epsilon}^2K^2\frac{\partial^2 \lambda}{\partial K^2} \\
			&\quad + f(N)\frac{\partial \lambda}{\partial N} +  \frac{1}{2}{\sigma}^2N^2 \frac{\partial^2 \lambda}{\partial N^2} + A\beta N^{1-\beta}K^{\beta-1}\lambda= 0,\\
			&\lambda(T,K,N) = u'_2\left(\frac{K}{N}\right).
		\end{aligned}
		\right.
	\end{equation}
	We note the inverse function $j$, using the first condition (on $u_1$) in Assumption \ref{Assum: running utility function $u_1$},  satisfies the inequality:
	\begin{equation}\label{eq:condition on j(x)}
		0< j(x)\le {a_1}x^{-\frac{1}{\gamma}}, \quad 0<-xj'(x)<\tilde{a}_1j(x) + \tilde{\tilde{a}}_1,~ \textrm{for}~ x>0.
	\end{equation}
	From \eqref{eq:condition on j(x)}, we know that $0<-j'(x)<a_1\tilde{a}_1 x^{-1-\frac{1}{\gamma}} + \frac{\tilde{\tilde{a}}_1}{x}\lesssim x^{-1-\frac{1}{\gamma}} +1$, 
	then for the sake of convenience and to avoid introducing new parameters, we also assume that
	\begin{equation}\label{eq:another condition on j(x)}
		0<-j'(x)\leq \tilde{a}_1 x^{-1-\frac{1}{\gamma}} + \tilde{\tilde{a}}_1, ~ \textrm{for}~ x>0.
	\end{equation}
	Note that $j(\lambda)$ is generally nonlinear in $\lambda$.  To rectify this, one may replace the argument $\lambda$ in $j(\lambda)$ by a given function $\psi$. We therefore consider the  linearized problem
	\begin{equation}\label{eq: differential equation for $lambda$} 
		\left\{
		\begin{aligned}
			&{\lambda}_t - Nj(\psi)\frac{\partial \lambda}{\partial K} +  A N^{1-\beta}K^{\beta}\frac{\partial \lambda}{\partial K} +{\epsilon}^2K\frac{\partial \lambda}{\partial K} + \frac{1}{2}{\epsilon}^2K^2\frac{\partial^2 \lambda}{\partial K^2} \\
			&\quad + f(N)\frac{\partial \lambda}{\partial N} +  \frac{1}{2}{{\sigma}^2}N^2 \frac{\partial^2 \lambda}{\partial N^2} + A\beta N^{1-\beta}K^{\beta-1}\lambda= 0,\\
			&\lambda(T,K,N) =u'_2\left(\frac{K}{N}\right).
		\end{aligned}
		\right.
	\end{equation}
	Based on it, we define a map $\Gamma$ from  an input $\psi$ to   $\lambda$.
	A suitable choice of weighted Sobolev space will be chosen later, on which  a fixed point theorem  for   $\Gamma$ can be applied to construct a weak solution. 
	To this end, we first discuss  the weak solution to  Equation \eqref{eq: differential equation for $lambda$}  in Subsection \ref{sec:Weak solution of the linearized PDE}. Its unique existence and  regularity will be discussed in Sections \ref{sec:paraboliceq}, \ref{sec:energy estimate} and \ref{sec:Weak solution and fixed point for nonlinear equation}.
	\subsection{Weak solution of the linearized equation \eqref{eq: differential equation for $lambda$} }\label{sec:Weak solution of the linearized PDE}
	
	The unbounded domain of  PDEs \eqref{eq: differential equation for original $lambda$} and \eqref{eq: differential equation for $lambda$} necessitates us to adopt a weighted Sobolev space;  it   essentially compactifies the domain.
	To this end, we define the following inner products   associated with different power weights of $K$ and $N$ and a common weight $\varphi$:
	\begin{equation}\label{eq:inner product}
		\begin{aligned}
			(\phi,\psi)_{k,K,n,N,\varphi} :=\int_{\mathbb{R}_+^2 }\phi \psi K^{2k}N^{2n}{\varphi}  dKdN, ~k,n \in \mathbb{N},
		\end{aligned}
	\end{equation}
	where $\varphi$ is a smooth positive bounded function on $\mathbb{R}^2_+$ to be specified later in Equation \eqref{eq: the choice of the weight}. 
	We  denote the  corresponding spaces  by  
	${L^2_{k,K, n,N,\varphi}}$ , with  norms  
	denoted by $||\cdot||_{L^2_{k,K, n,N,\varphi}}$. 
	Moreover, we denote $(\phi,\psi)_{\varphi}=(\phi,\psi)_{0,K,0,N,\varphi}$, 
	$(\phi,\psi)_{k,K,\varphi}=(\phi,\psi)_{k,K,0,N,\varphi}$ and $(\phi,\psi)_{n,N,\varphi}=(\phi,\psi)_{0,K,n,N,\varphi}$.
	The spaces $L^2_{\varphi}$, ${L^2_{k,K,\varphi}}$ and ${L^2_{n,N,\varphi}}$, and their norms can be understood accordingly.
	
	Next, we define  the weighted Sobolev space 
	\begin{equation}\label{functional space H^1_varphi}
		H^1_\varphi:=\left\{ \psi \in  L^2_\varphi\left|  \left|\left|\frac{\partial \psi}{\partial K}\right|\right|^2_{L^2_{1,K,\varphi}}<\infty, \left|\left|\frac{\partial \psi}{\partial N}\right|\right|^2_{L^2_{1,N,\varphi}} <\infty \right.\right\},
	\end{equation}
	with the norm
	$||\psi||_{H^1_\varphi}^2=||\psi||^2_{L^2_\varphi} + \left|\left|\frac{\partial \psi}{\partial K}\right|\right|^2_{L^2_{1,K,\varphi}} + \left|\left|\frac{\partial \psi}{\partial N}\right|\right|^2_{L^2_{1,N,\varphi}}$.
	To handle the mismatching of the exponents in the  term $AN^{1-\beta}K^{\beta}$, we further define
{\color{black}  $Y:= L^2_\varphi \cap L^2_{\tilde{\varphi}} = L^2_{\varphi+\tilde{\varphi}}$ where
	\begin{equation}\label{eq:def of tildevarphi}
			\tilde{\varphi}= \left\{
		\begin{aligned}
		(N^{4-4\beta}K^{4\beta-4}  + N^{2+2\beta}K^{-2-2\beta})\varphi~ \textrm{ when $\beta \ge \gamma$};\\
			  (N^{4-4\beta}K^{4\beta-4}  + N^{2+2\gamma}K^{-2-2\gamma})\varphi ~\textrm{ when $\beta < \gamma$}.\\
		\end{aligned}
		\right.
	\end{equation} 
	We  define $ Y_1:=H^1_\varphi \cap L^2_{\tilde{\varphi}}$, 
	and $Y_0$ is the dual space of $Y_1$. 
	Note that 
	$Y_1$ is a Hilbert space with the inner product $	(\phi,\psi)_{Y_1}:=\int_{\mathbb{R}_+^2 }\phi \psi ({\varphi}+\tilde{\varphi}) +\frac{\partial \phi}{\partial K}\frac{\partial \psi}{\partial K}K^2\varphi +\frac{\partial \phi}{\partial N}\frac{\partial \psi}{\partial N}N^2\varphi  dKdN$. The  dual space $Y_0$ is endowed with the operator norm.}
We shall always identify the space $L^2_{\varphi}$ with its dual space, but not $Y$ and its dual space, as we cannot identity them with their dual spaces simultaneously.
\begin{rem}
	Any $\psi \in L^2_{\varphi}$ can be identified as an element in $Y_0$ through the  linear functional: for any $\phi \in Y_1$, $
	\phi  \mapsto \int_{\mathbb{R}_+^2 } \psi \phi \varphi  dKdN = (\psi, \phi)_{ \varphi}$.
	With  this  identification,   we  have a  Hilbert triple $Y_1 \subset  L^2_{\varphi} \subset Y_0$.
\end{rem}
If $\psi$ is a square integrable function of time $t$, capital $K$ and population $N$, we can treat $\psi$ as a map from $[0,T]$ to $L^2_{\varphi}$. Hence, we can naturally define the following space 
\begin{equation}\label{functional space L^2([0,T];L^2_varphi)}
	L^2([0,T];L^2_\varphi): = \left\{ \psi : [0, T] \to L^2_\varphi\left|  \int_{0}^{T}\|\psi \|^2_{L^2_\varphi}dt<\infty \right.\right\},
\end{equation}
with the norm
$||\psi||_{L^2([0,T];L^2_\varphi)}^2=\int_{0}^{T} \|\psi \|^2_{L^2_\varphi} dt.$
Analogous to \eqref{functional space L^2([0,T];L^2_varphi)}, we further define other functional spaces by using $Y_1$ and  $Y_0$  as follows: The space 
\begin{equation}\label{functional space L^2([0,T];Y_1)}
	L^2([0,T];Y_1): = \left\{ \psi : [0, T] \to Y_1\left|   \int_{0}^{T}\left|\left|\psi \right|\right|^2_{Y_1}dt<\infty\right. \right\} 
\end{equation}
is equipped with norm
$||\psi||_{L^2([0,T];Y_1)}^2=\int_{0}^{T} \left|\left|\psi \right|\right|^2_{Y_1} dt$.
The space
\begin{equation}\label{functional space L^2([0,T];Y_0)}
	L^2([0,T];Y_0)
	: = \left\{ \psi : [0, T] \to Y_0\left|   \int_{0}^{T}\left|\left|\psi\right|\right|^2_{Y_0}dt<\infty\right.\right\}
\end{equation}
is equipped with norm
$||\psi||_{L^2([0,T];Y_0)}^2= \int_{0}^{T}\left|\left|\psi\right|\right|^2_{Y_0}dt$.
The following space  is the space in which we shall seek for a weak solution:
\begin{equation}\label{functional space mathcal{H}^1}
	\mathcal{H}^1_\varphi := \left\{  \psi\in L^2([0,T];Y_1),\frac{\partial \psi}{\partial t}\in L^2([0,T];Y_0)\right\}.
\end{equation}
\begin{defn}\label{def: definiton of the weak solution of the linearized problem}
	We call $\lambda$ $\in$ $\mathcal{H}^1_\varphi$ a weak solution to Problem \eqref{eq: differential equation for $lambda$} if the following equations are satisfied: for every $\phi \in Y_1$, 
	\begin{equation}\label{eq:weak solution of equation of lambda}
		\left\{
		\begin{aligned}
			&\langle\lambda_t, \phi\rangle_{Y_0,Y_1} +\int_{\mathbb{R}_+^2 }\left(- N j(\psi)  \frac{\partial \lambda}{\partial K}+  A N^{1-\beta}K^{\beta}\frac{\partial \lambda}{\partial K} +{\epsilon}^2K\frac{\partial \lambda}{\partial K} \right) \phi \varphi dKdN \\
			& \quad+ \int_{\mathbb{R}_+^2 }\left(f(N)\frac{\partial \lambda}{\partial N} +  A\beta N^{1-\beta}K^{\beta-1}\lambda\right)\phi \varphi dKdN \\
			&\quad -\int_{\mathbb{R}_+^2 }\frac{1}{2}{\epsilon}^2\frac{\partial \lambda}{\partial K}\left(\frac{\partial \phi}{\partial K}K^2\varphi + \phi\frac{\partial (K^2\varphi)}{\partial K}\right) +  \frac{1}{2}{{\sigma}^2}\frac{\partial \lambda}{\partial N}\left(\frac{\partial \phi}{\partial N}N^2\varphi + \phi\frac{\partial (N^2\varphi)}{\partial N}\right)dKdN \\
			&\qquad = 0, \quad \textrm{a.e. $t \in [0,T],$}
			\\
			&\lambda(T,K,N)  = u'_2\left(\frac{K}{N}\right).
		\end{aligned}
		\right.
	\end{equation}
	Here, $\langle\lambda_t, \phi\rangle_{Y_0,Y_1}$ represents the pairing between the dual space $Y_0$ and $Y_1$ and $\langle\lambda_t, \phi\rangle_{Y_0,Y_1} = \left(\lambda_t, \phi\right)_{\varphi}$ if $\lambda_{t} \in L^2_\varphi$.  
	Similarly, we call $\lambda$ $\in$ $\mathcal{H}^1_\varphi$ a weak solution to Problem \eqref{eq: differential equation for original $lambda$} if  for all $\phi \in Y_1$,
	\begin{equation}\label{eq:weak solution of nonlinear equation of lambda}
		\left\{
		\begin{aligned}
			&\langle\lambda_t, \phi\rangle_{Y_0,Y_1} +\int_{\mathbb{R}_+^2 }\left(- N j(\lambda)  \frac{\partial \lambda}{\partial K}+  A N^{1-\beta}K^{\beta}\frac{\partial \lambda}{\partial K} +{\epsilon}^2K\frac{\partial \lambda}{\partial K} \right) \phi \varphi dKdN \\
			& \quad+ \int_{\mathbb{R}_+^2 }\left(f(N)\frac{\partial \lambda}{\partial N} +  A\beta N^{1-\beta}K^{\beta-1}\lambda\right)\phi \varphi dKdN \\
			&\quad -\int_{\mathbb{R}_+^2 }\frac{1}{2}{\epsilon}^2\frac{\partial \lambda}{\partial K}\left(\frac{\partial \phi}{\partial K}K^2\varphi + \phi\frac{\partial (K^2\varphi)}{\partial K}\right) +  \frac{1}{2}{{\sigma}^2}\frac{\partial \lambda}{\partial N}\left(\frac{\partial \phi}{\partial N}N^2\varphi + \phi\frac{\partial (N^2\varphi)}{\partial N}\right)dKdN \\
			&\qquad = 0, ~ \textrm{a.e. for  $t \in [0,T],$}
			\\
			&\lambda(T,K,N)  = u'_2\left(\frac{K}{N}\right).
		\end{aligned}
		\right.
	\end{equation}
\end{defn}
We note that 
by Lions–Magenes lemma \cite{Lions72}, $\mathcal{H}^1_\varphi$ is embedded into $C([0,T];L^2_{\varphi})$, hence,   the terminal condition $\lambda(T,K,N)  = K^{-\gamma}N^{\gamma}$ a.e. in $K,N$ in Equations \eqref{eq:weak solution of equation of lambda} and \eqref{eq:weak solution of nonlinear equation of lambda} makes sense. 
We define a bilinear map at time $t$:
\begin{equation}\label{eq:bilinear form}
	\begin{aligned}
		B[\lambda(t),\phi;t]  :&=
		\int_{\mathbb{R}_+^2 }\left(- N\frac{\partial \lambda}{\partial K}j(\psi) +  A N^{1-\beta}K^{\beta}\frac{\partial \lambda}{\partial K} +{\epsilon}^2K\frac{\partial \lambda}{\partial K} \right) \phi \varphi dKdN \\
		& + \int_{\mathbb{R}_+^2 }\left(f(N)\frac{\partial \lambda}{\partial N} +  A\beta N^{1-\beta}K^{\beta-1}\lambda\right) \phi \varphi dKdN\\
		& - \int_{\mathbb{R}_+^2 }\frac{1}{2}{\epsilon}^2\frac{\partial \lambda}{\partial K}\left(\frac{\partial \phi}{\partial K}K^2\varphi + \phi\frac{\partial (K^2\varphi)}{\partial K}\right) +  \frac{1}{2}{{\sigma}^2}\frac{\partial \lambda}{\partial N}\left(\frac{\partial \phi}{\partial N}N^2\varphi + \phi\frac{\partial (N^2\varphi)}{\partial N}\right)dKdN.\\
	\end{aligned}
\end{equation}
Then, we can rewrite   Equation \eqref{eq:weak solution of equation of lambda} as 
\begin{equation}\label{eq:another writing for weak solution}
	\left\{
	\begin{aligned}
		&\langle\lambda_t, \phi\rangle_{Y_0, Y_1}   + B[\lambda(t),\phi;t]  =0 \textrm{ for all $\phi$ in $Y_1$ and a.e. $t$ $\in$ $[0,T]$,}\\
		&\lambda(T,K,N)  = u'_2\left(\frac{K}{N}\right).
	\end{aligned}
	\right.
\end{equation}
We note that if $\lambda$ is a classical solution to Equation \eqref{eq: differential equation for $lambda$}, then Equation \eqref{eq:weak solution of equation of lambda} or \eqref{eq:another writing for weak solution} holds for any $\phi\in C^{\infty}_c({\mathbb{R}_+^2 })$. If we know that $C^{\infty}_c({\mathbb{R}_+^2 })$ is dense in  $Y_1$ and $B[\lambda(t),\phi;t]$ is a bilinear form on $Y_1$ for each  $t\in[0,T]$, then by a density argument, we can see that Equation \eqref{eq:weak solution of equation of lambda} or \eqref{eq:another writing for weak solution}  still holds for all $\phi \in Y_1$ under the requirement that $\lambda(t)\in Y_1$ with $\frac{\partial \lambda}{\partial t} \in Y_0$. The density of  $C^{\infty}_c({\mathbb{R}_+^2 })$ in $Y_1$ follows from the standard argument with the help of weights $K^2\varphi$ and $N^2\varphi$. We deduce $|B[\lambda(t),\phi;t]| \le C||\lambda(t)||_{Y_1}||\phi||_{Y_1}$ for some constant $C$ independent of $\lambda$ and $\phi$  from the  H\"older's inequality and the definition of $Y_1$;  we refer to the proof of Theorem \ref{thm:Y_0 estimate} for more details.

We introduce some relevant notations before we proceed.
The parabolic distance between two points $P=(t,x)$ and $Q=(s,y)$ in $Q_T:= (0,T)\times \Omega$ for some domain $\Omega$ in $\mathbb{R}^n$ is defined by 
$d(P,Q) := (|t-s|+|x-y|^2 )^{\frac{1}{2}}$,
where $|x-y|$ is the Euclidean distance between $x$ and $y$. 
We define the ($t-$anisotropic) H\"older space $C^{\frac{\alpha}{2},\alpha}(Q_T)$  for any $0<\alpha<1$  as the collection of functions $u$ that
\begin{equation*}
	\|u\|_{\frac{\alpha}{2},\alpha}:=\sup_{P\in Q_T} |u(P)| + \sup_{P,Q \in Q_T, P\ne Q}\frac{|u(P)-u(Q)|}{|d(P,Q)|^\alpha} <\infty,
\end{equation*}
and the norm of $u$ in $C^{\frac{\alpha}{2},\alpha}(Q_T)$ is defined by the quantity $\|u\|_{\frac{\alpha}{2},\alpha}$ above.
Similarly, for any  integer $k$, 
we define $C^{k+\frac{\alpha}{2},2k+\alpha}(Q_T)$  as  the collection of functions $u$ that
\begin{equation*}
	\left\|u\right\|_{k+\frac{\alpha}{2},2k+\alpha} :=\sum_{2|k_1|+|k_2|\le 2k}\left\|D^{k_1}_tD^{k_2}_x u \right\|_{\frac{\alpha}{2},\alpha} < \infty,
\end{equation*}
and the norm of $u$ in $C^{k+\frac{\alpha}{2},2k+\alpha}(Q_T)$ is defined by $\|u\|_{k+\frac{\alpha}{2},2k+\alpha}$. 
Moreover, we say that $u$ is  in $C^{k+\frac{\alpha}{2},2k+\alpha}_{loc}(Q_T)$ if $u\in $ $C^{k+\frac{\alpha}{2},2k+\alpha}(V)$ for any compact set $V$ of $Q_T$. 
Similarly, we also consider the usual (spatial) H\"older space $C^{\alpha}(\Omega)$ consists of functions $u$ such that 
\begin{equation*}
	\|u\|_{\alpha}:=\sup_{x\in \Omega} |u_1(x)| + \sup_{x,y \in \Omega, x\ne y}\frac{|u_1(x)-u(y)|}{|x-y|^\alpha} <\infty,
\end{equation*}
with the corresponding norm $\|u\|_{\alpha}$. The spaces $C^{k+\alpha}(\Omega)$ and $C^{k+\alpha}_{loc}(\Omega)$ are defined accordingly.

\subsection{On the solution of the linearized equation \eqref{eq: differential equation for $lambda$} and the fixed point argument}
We first discuss the feasible choices of $\psi$ that ensure the existence of a solution $\lambda$ to \eqref{eq: differential equation for $lambda$}.
As pointed out in Subsection \ref{sec:Main resluts}, the problematic coefficient $j(\psi)\sim \psi^{-\frac{1}{\gamma}}$ demands us to bound the growth of   $\psi$ from below so that a pointwise upper bound of $\lambda$ can be obtained. Then we are able to choose suitable weights according to this   upper bound, and henceforth  to derive a proper weighted energy estimate;  see Section \ref{sec:energy estimate} for details. 
Finally, to ensure  $\Gamma$ is a self-map, we have to obtain the same lower bound for $\lambda$.
Fortunately, we can obtain this lower bound {\itshape a priori} (independent of the input $\psi$), mainly due to the terminal data already has a certain lower bound  that can propagate over time. The proof will heavily rely  on the Feymann-Kac theorem; also see Section \ref{section: pointwise estimate}.


To facilitate the subsequent fixed point argument in Section \ref{sec:Weak solution and fixed point for nonlinear equation}, we define the set in which we seek a fixed point as follows:
\begin{equation}\label{def:definition for the growth rate set}
	S:=\left\{
	\begin{aligned}
		&\left\{\psi \left|c_0^{\gamma}\left(\frac{N}{K}\right)^{\gamma}\le\psi(K,N,t)\le C_1\left(\frac{N}{K}\right)^{\gamma} + C_2\left(\frac{N}{K}\right)^{\beta} \right. \textrm{a.e. } \right\}, 
		\qquad \qquad \qquad\textrm{if $\beta \ge \gamma$};\\
		&\left\{\psi\left|C_3  \left(\frac{N}{K}\right)^{\beta}\left(  C_4 \left(\frac{N}{K}\right)^{\beta-1} + C_5\right)^{ \frac{\gamma-\beta}{\beta-1 } }\le \psi (t,K,N) \le C_6\left(\frac{N}{K}\right)^{\gamma} \right. \textrm{a.e. } \right\},  \textrm{if $\beta < \gamma$}.
	\end{aligned}
	\right.
\end{equation}
The constant $c_0$ will be given in \eqref{eq:value for c_0},	 $C_1$ and $C_2$ will be given in \eqref{eq: values for C_1 and C_2 case1} when $\gamma \le 2\beta -1$ and  $\gamma<\beta$ while it  will be given in \eqref{eq: values for C_1 and C_2} when  $\beta>\gamma > 2\beta -1$;  finally the values of $C_1$ and $C_2$ will be given in \eqref{eq: values for C_1 and C_2 when bera=gamma} when $\beta =\gamma$. The constants  
$C_3, C_4$ and $C_5$ will be given in \eqref{eq:value for C3, C4, C5},  and $C_6$ will be given in 
\eqref{eq:value for C6}.
In Theorem \ref{sec:The extension lemma and fixed point argument}, we shall demonstrate that $\Gamma:S\to S$ is well-defined and possesses a fixed point.

\section{Numerical experiments and  economic interpretations}\label{section:Empirical study}

We present numerical simulations to explore the dynamics of state variables $K$ and $N$ and compare their behaviors with their deterministic counterparts when both $\sigma$ and $\epsilon$ vanish. We observe overshooting and undershooting phenomena in the capital process $K$, as well as  how the relative magnitude of   the production function and the volatilities of $K$ and $N$ affect the path of  the capital, resulting in transient behavior. 
To conduct the simulations, we employ the Euler-Maruyama method (cf. \cite{Kloeden92}) with a step size of $\Delta t=0.002$ for time $t$. Additionally, we set $\alpha$ = 0.5 and $ M $= 3 in $f(x)$ and $\beta = \gamma = 0.5$ in $F(K,N)$.

We first provide numerical examples for the population state. The quantity $\bar{N}$ in $f(x)$ represents the equilibrium level of the population, to which the population should converge  as $t$ approaches infinity. Figure \ref{simulation of population with different initials} depicts the movement of $N$ against $t$ for different initial values $N_0=$  1.4 and 3 with $\bar{N}=2$, along with varying values of $\sigma$, representing different levels of volatility and fluctuation.
Notably, all scenarios demonstrate quick convergence, and the red curves in both subfigures cover all the other curves due to their largest magnitudes of fluctuations. This observation highlights the significant impact of $\sigma$ on the population dynamics.

\begin{figure}[!t]
	\centering
	\includegraphics[width=0.4\textwidth]{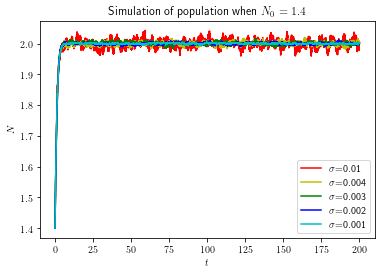}
	\includegraphics[width=0.4\textwidth]{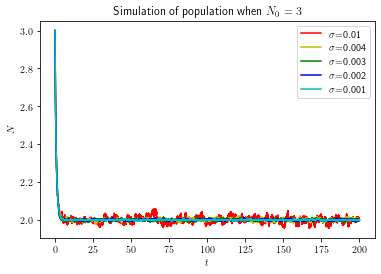}
	\caption{Simulation of population with  $\alpha$= 0.5, $M=3$ and $\bar{N}=2.0$.}
	\label{simulation of population with different initials}
\end{figure}


As for the capital dynamics, we first consider a scenario where the production function is dominant due to a very large coefficient $A$, indicating high production efficiency. In our simulations, given an initial value for capital (which is 2 in our simulation), the production function drives the behavior of capital, leading to its increase regardless of population movement. Capital also converges to an equilibrium in a manner similar to population, as shown in Figure \ref{simulation of capital with large production factor (N_0, A)=(1.4, 10)}.
If we ignore production (i.e., $A=0$), the capital decreases significantly with the sole contribution from consumption. Hence, a phase transition behavior is expected for intermediate values of $A$.

Meanwhile, it is crucial in our study whether the population is increasing or decreasing to the equilibrium level.
The total consumption remains approximately proportional to capital due to the control effect $c=(u'_1)^{-1}(v_k)\simeq \frac{K}{N}$.
If the population is relatively small at the beginning and increases towards the equilibrium level,  when $A$ falls within a critical region, a small initial value of $N$ leads to  a small value for the production function initially, after subtracting the total consumption it will lead to the decay of capital. 
Subsequently, we may expect the capital  to increase due to the population growth and reduced total consumption, resulting in an undershoot behavior.
Conversely, if the initial population is relatively large and decreases towards equilibrium, we may expect an overshoot behavior in capital.
This model motivates an explanation of a sharp economic growth  after a body boom. 
These two types of capital behavior illustrate the potential pitfalls in anticipating economic trends. For instance, a decrease in capital might be perceived negatively, but it could actually be the turning point for the economy. In a certain sense, in order to understand the economy, one has to consider both the population and capital dynamics.
By carefully considering the interactions between production function, population, and consumption, we gain insights into the dynamic behavior of the economy, shedding light on potential economic trends and the need for cautious interpretation of observed phenomena.
To verify our observations, we present various numerical illustrations on the movement of the capital. We have previously observed the rapid convergence of population in Figure \ref{simulation of population with different initials} regardless of the value of $\sigma$. Therefore, in the following illustrations, we set $\sigma=0$ to demonstrate the capital dynamics more clearly. Now, let us describe the other parameter settings as follows:
the population deterministically increase or decrease from initial value $N_0$ to the equilibrium level $\bar{N}=2$; $\epsilon$ takes values from 0 to 0.05 at an 0.01 interval. 
We present the simulations  of $K(t)$ with initial value $K_0=2$ in   Figures \ref{simulation of capital with large production factor (N_0, A)=(1.4, 10)}, \ref{simulation of capitial with (N_0, A)=(1.4, 1),}, \ref{simulation of capitial with (N_0, A)=(3,10)} and     \ref{simulation of capitial with (N_0, A)=(3,1)}    when the pair of parameters $(N_0, A) = (1.4, 10), (1.4, 1), (3, 10), (3, 1)$, respectively. We observe  that when $A=10$, $K(t)$ exhibits an increasing trend regardless of the values of $N_0$ and $\epsilon$. 
On the other hand, when $A=1$, which actually falls into a critical region, we observe clear undershoot (resp. overshoot) behavior as long as $\epsilon$ below 0.03 when there is an increasing (resp. decreasing) trend towards equilibrium for population. However, such behavior becomes unidentifiable when $\epsilon$ is above 0.03 due to the large fluctuations, regardless of the population trend.

\begin{figure}[!t]
	\centering
	\includegraphics[width=0.4\textwidth]{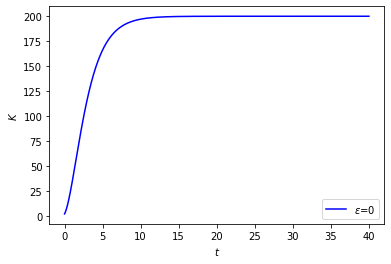}
	\includegraphics[width=0.4\textwidth]{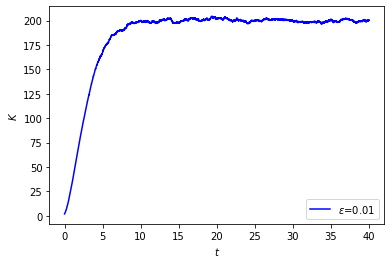}
	\includegraphics[width=0.4\textwidth]{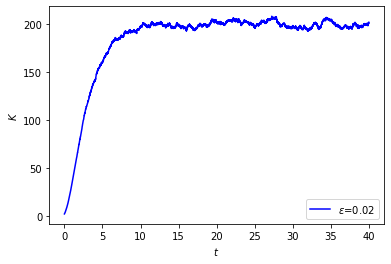}
	\includegraphics[width=0.4\textwidth]{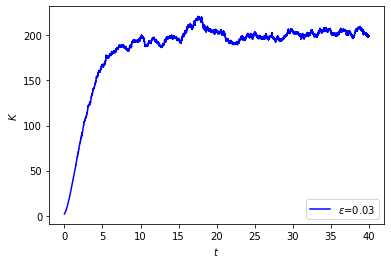}
	\includegraphics[width=0.4\textwidth]{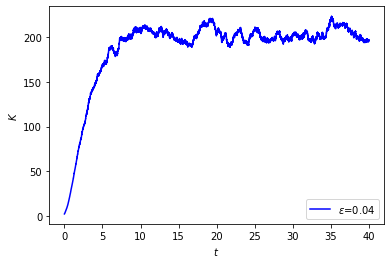}
	\includegraphics[width=0.4\textwidth]{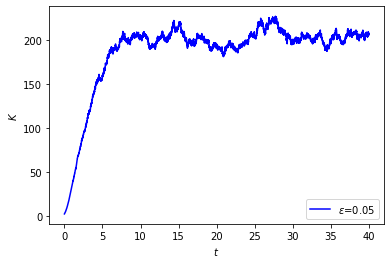}
	\caption{Simulation of capital when $(N_0, A)=(1.4, 10)$, $K_0=2$ and $\epsilon= 0,\ 0.01, \ 0.02, \ 0.03, \ 0.04, \ 0.05$.}
	\label{simulation of capital with large production factor (N_0, A)=(1.4, 10)}
\end{figure}


\begin{figure}[!t]
	\centering
	\includegraphics[width=0.4\textwidth]{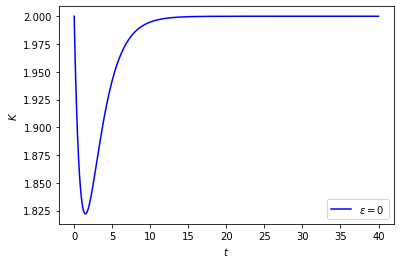}
	\includegraphics[width=0.4\textwidth]{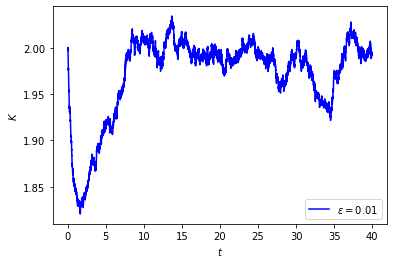}
	\includegraphics[width=0.4\textwidth]{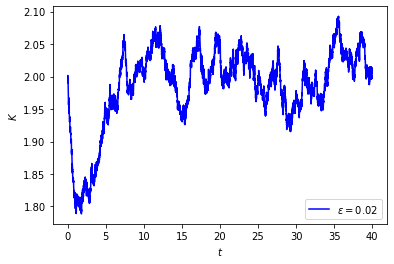}
	\includegraphics[width=0.4\textwidth]{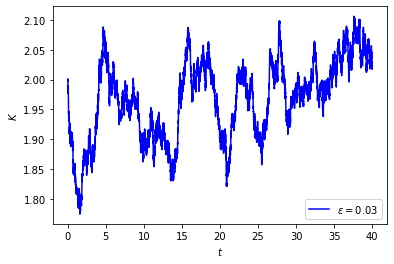}
	\includegraphics[width=0.4\textwidth]{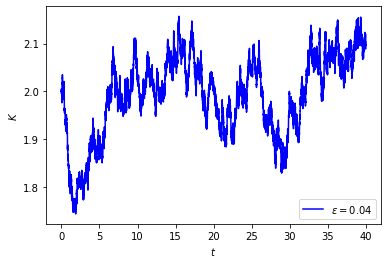}
	\includegraphics[width=0.4\textwidth]{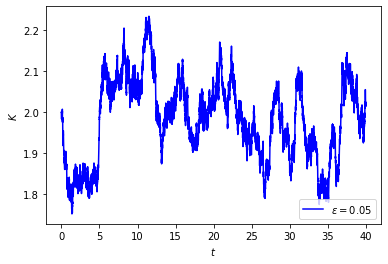}
	\caption{Simulation of capital when $(N_0, A)=(1.4, 1)$, $K_0=2$ and $\epsilon=0, \ 0.01, \  0.02, \ 0.03, \ 0.04, \ 0.05$}
	\label{simulation of capitial with (N_0, A)=(1.4, 1),}
\end{figure}


%
%
\begin{figure}[!t]
	\centering
	\includegraphics[width=0.4\textwidth]{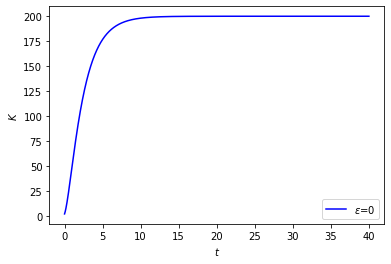}
	\includegraphics[width=0.4\textwidth]{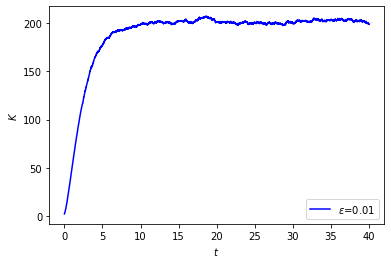}
	\includegraphics[width=0.4\textwidth]{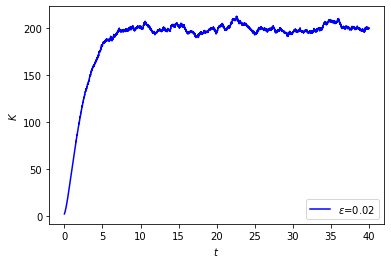}
	\includegraphics[width=0.4\textwidth]{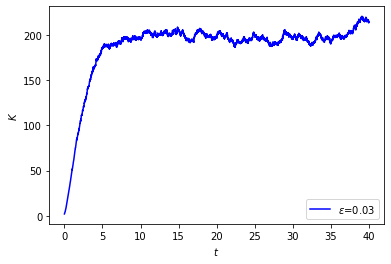}
	\includegraphics[width=0.4\textwidth]{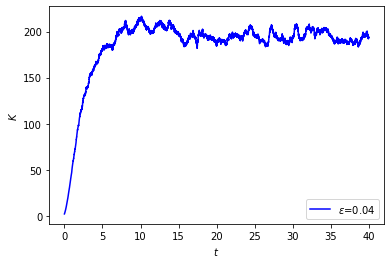}
	\includegraphics[width=0.4\textwidth]{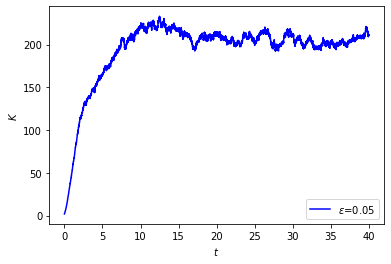}
	\caption{Simulation of capital when $(N_0, A)=(3,10)$, $K_0=2$ and $\epsilon=0, \ 0.01, \ 0.02, \ 0.03, \ 0.04, \ 0.05$.}
	\label{simulation of capitial with (N_0, A)=(3,10)}
\end{figure}


\begin{figure}[!t]
	\centering
	\includegraphics[width=0.4\textwidth]{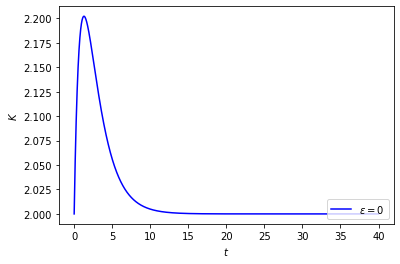}
	\includegraphics[width=0.4\textwidth]{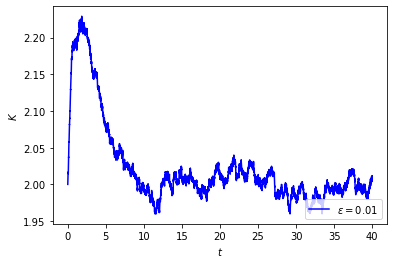}
	\includegraphics[width=0.4\textwidth]{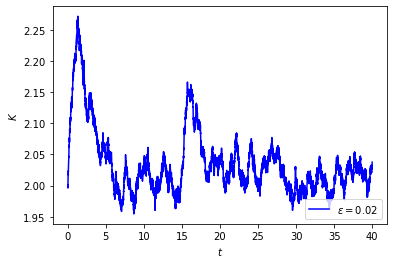}
	\includegraphics[width=0.4\textwidth]{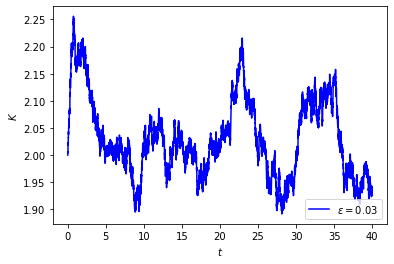}
	\includegraphics[width=0.4\textwidth]{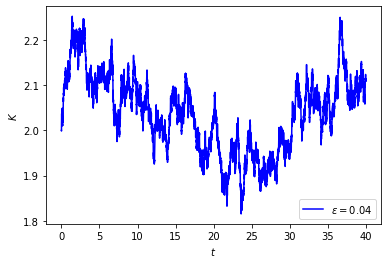}
	\includegraphics[width=0.4\textwidth]{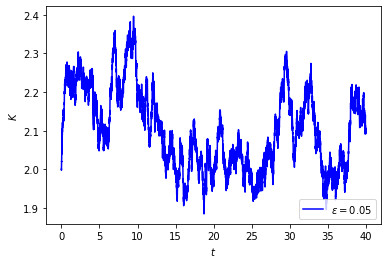}
	\caption{Simulation of capital when $(N_0, A)=(3,1)$, $K_0=2$ and $\epsilon= 0, \ 0.01, \ 0.02, \ 0.03, \ 0.04, \ 0.05$.}
	\label{simulation of capitial with (N_0, A)=(3,1)}
\end{figure}



To conclude this section, we further investigate the effect of the random noise; namely its existence hinders our ability to tell whether the overshoot or undershoot behavior is present for $K(t)$. 
To do so, we conduct two sets of simulations and plot their respective results in Figures \ref{Simulation of population with equilibrium level1.6}  to \ref{simulation of capital under different population dynamics for small initial N}. 
In these figures, we first consider two populations with the same initial value but different equilibrium, with corresponding parameters specified in each plot.
We then  plot the paths of $K$ in Figure \ref{simulation of capital under different population dynamics for large initial N}  and  Figure \ref{simulation of capital under different population dynamics for small initial N} under the different dynamics of populations. We can see that if we are only given the development of the capital,  it is challenging for us to determine whether the red line or the violet line of population traces represent the true underlying population, especially when the random noise is significant, hence increasing the difficulty to identify the overshoot (resp. undershoot) behavior of capital correctly. 
That is, the presence of random noise complicates the interpretation of the asymptotic behavior of the capital, leading to potential misunderstandings. Therefore, the effect of random fluctuations should be carefully considered when we analyze the behavior of the capital process and draw conclusions based solely on observed data.


\begin{figure}[!t]
	\centering
	\includegraphics[width=0.4\textwidth]{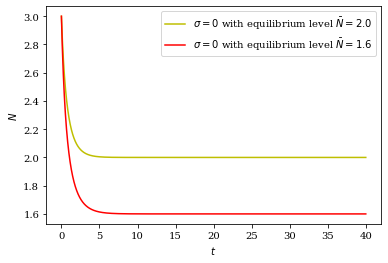}
	\caption{Simulation of population with initial value $N_0=$3 and equilibrium levels $\bar{N}=$1.6 and $\bar{N}=$2.}
	\label{Simulation of population with equilibrium level1.6}
\end{figure}

\begin{figure}[!t]
	\centering
	\includegraphics[width=0.4\textwidth]{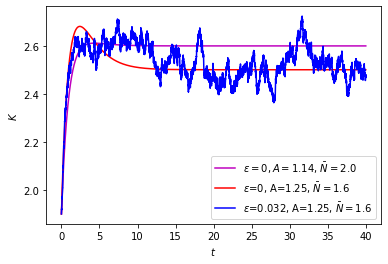}
	\includegraphics[width=0.4\textwidth]{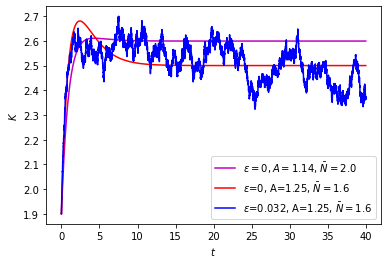}
	\caption{Simulation of capital under different population dynamics when $N_0=3, K_0=1.9$.}
	\label{simulation of capital under different population dynamics for large initial N}
\end{figure}


\begin{figure}[!t]
	\centering
	\includegraphics[width=0.4\textwidth]{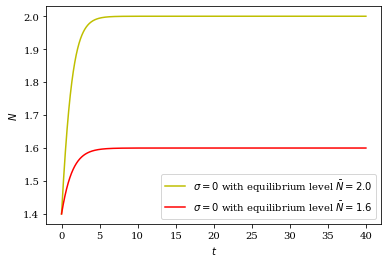}
	\caption{Simulation of population with initial value $N_0=$1.4 and  equilibrium levels $\bar{N}=$1.6 and $\bar{N}=$2.}
	\label{Simulation of population with equilibrium level 1.6 and 2 and initial 1.4,}
\end{figure}

\begin{figure}[!t]
	\centering
	\includegraphics[width=0.4\textwidth]{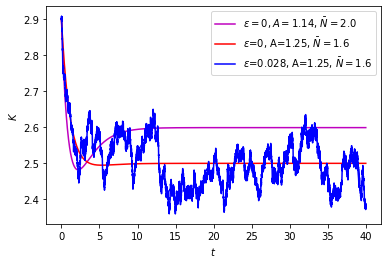}
	\includegraphics[width=0.4\textwidth]{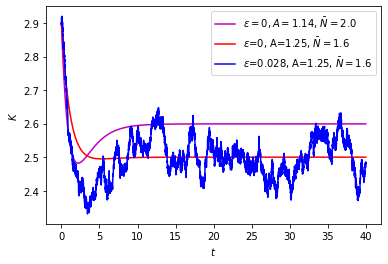}
	\caption{ Simulation of capital evolution under different population dynamics.}
	\label{simulation of capital under different population dynamics for small initial N}
\end{figure}


\section{Unique existence of the classical solution of  \eqref{eq: differential equation for $lambda$} in the local H\"older space}\label{sec:paraboliceq}

We  show the existence in Subsection \ref{sec: L^infty a prior estimates} and uniqueness in Subsection \ref{sec:Maximum Principle and Uniqueness} for Problem \eqref{eq: differential equation for $lambda$} with smooth input $\psi$.  The main difficulty encountered here is the   exponential growth of the coefficients of  Equation \eqref{eq: equation for tilde lamba},  this prevents us from obtaining a global Schauder estimates to establish a solution, and it is also not possible to use the classical maximum principle to conclude uniqueness. 
Since the solution is expected to be unbounded on the whole space, our main idea is to multiply the solution by an appropriate spatial weight so that the weighted solution is bounded.  Then, we can apply the maximum principle to the weighted solution to derive $L^{\infty}$ estimates and uniqueness.

%
One should note  that the $L^{\infty}$ bounds obtained in this section do not provide the crucial non-trivial lower bounds of the solutions to \eqref{eq: differential equation for $lambda$}, which is essential for the well-posedness of the nonlinear problem, and this will be the main objective of   Section \ref{section: pointwise estimate}.

We now state the existence and uniqueness results given in this section:
\begin{thm}\label{thm:existence when psi is smooth}
	Let   $\psi \in {C^{\infty}(  [0,T] \times \mathbb{R}^2_+) }$  such that
	
(i) when $\beta\geq \gamma$, $~\psi(t,K,N) \ge c_0^{\gamma}\left(\frac{N}{K}\right)^{\gamma}~$ for a positive constant $c_0$;
	
(ii) 
when $\beta< \gamma$, $~\psi(t,K,N) \ge C_3  \left(\frac{N}{K}\right)^{\beta}\left(  C_4 \left(\frac{N}{K}\right)^{\beta-1} + C_5\right)^{ \frac{\gamma-\beta}{\beta-1 } }~$ for some positive constants $C_3$, $C_4$ 

 ~ \quad and $C_5$.
\\
	Then there is a unique classical solution $\lambda$ to \eqref{eq: differential equation for $lambda$} satisfying  
	\begin{equation}
		|\lambda(t, K, N)| \le B \left(\left(\frac{N}{K}\right)^{\beta_1} + \left(\frac{N}{K}\right)^{\gamma_1}\right), ~ t \in [0,T],
	\end{equation}
	for some  $B>0$ and  $0<\beta_1, \gamma_1<1$.
	Furthermore, 
	$\lambda \in C^{1+\frac{\alpha}{2},2+\alpha}_{loc}(  [0,T] \times \mathbb{R}^2_+)$   for any $0<\alpha<1$.
\end{thm}
For this theorem, any constants $c_0$, $C_3$, $C_4$ and $C_5$  suffice for the claim of the unique existence of the solution.

\subsection{$L^{\infty}$ {\itshape a priori} estimates and  existence of a solution of the linearized equation \eqref{eq: differential equation for $lambda$}}\label{sec: L^infty a prior estimates}

For the linearized equation \eqref{eq: differential equation for $lambda$}, we try to use some classical tools of parabolic equations directly to obtain a regular solution. To do so, we want to first  transform \eqref{eq: differential equation for $lambda$} into a uniformly parabolic equation. Precisely,
we apply  the transformation $K=e^x$, $N=e^y$, $x,y\in \mathbb{R}$, and denote  $\tilde{\lambda}(t,x,y) = {\lambda}(t, e^x, e^y)$,  we then obtain the following equation for $\tilde{\lambda}$ from Equation  \eqref{eq: differential equation for $lambda$}: 
\begin{equation}\label{eq: equation for tilde lamba}
	\left\{
	\begin{aligned}
		\tilde{\lambda}_t &+\left(-e^{y-x}j(\psi(t, e^x, e^y))+ Ae^{(y-x)(1-\beta)} + \frac{1}{2}\epsilon^2\right)\tilde{\lambda}_x + \left(\frac{f(e^y)}{e^y}-\frac{1}{2}\sigma^2\right)\tilde{\lambda}_y \\
		&+ \frac{1}{2}\epsilon^2\tilde{\lambda}_{xx} + \frac{1}{2}\sigma^2\tilde{\lambda}_{yy}+ A\beta e^{(y-x)(1-\beta)}\tilde{\lambda} = 0, ~x, y \in \mathbb{R}, ~t \in [0,T),\\
		&\tilde{\lambda}(T,x,y)= \tilde{u}_2(e^{y-x}) \quad(\textrm{its rate is asymptotically equal to } e^{\gamma(y-x)}),
	\end{aligned}\right.
\end{equation}
where $\tilde{u}_2(z) = u'_2(\frac{1}{z})$.	
We  use $w: =\frac{\tilde{\lambda}}{\rho}$ to transform Equation \eqref{eq: equation for tilde lamba}, where the  auxiliary function
$\rho=e^{(\beta+\theta)(y-x)}+e^{\gamma(y-x)}>0$ and $\theta$ is a positive number.  We then obtain:
\begin{equation}\label{eq: the equation of w}
	\left\{
	\begin{aligned}
		&w_t+\left({-e^{y-x}j(\psi)} + Ae^{(y-x)(1-\beta)} + \frac{1}{2}\epsilon^2 +\epsilon^2\frac{\rho_x}{\rho}\right)w_x + \left(\frac{f(e^y)}{e^y}-\frac{1}{2}\sigma^2+\sigma^2\frac{\rho_y}{\rho}\right)w_y
		\\
		&\quad+ \frac{1}{2}\epsilon^2 w_{xx}
		+ \frac{1}{2}\sigma^2w_{yy}+ \left[\left( {\color{red}-e^{y-x}j(\psi)} + {\color{red}Ae^{(y-x)(1-\beta)} }+ \frac{1}{2}\epsilon^2\right)\frac{\rho_x}{\rho}\right.\\
		&\left.\quad+\left(\frac{f(e^y)}{e^y}-\frac{1}{2}\sigma^2\right)\frac{\rho_y}{\rho}+\frac{1}{2}\epsilon^2\frac{\rho_{xx}}{\rho}+\frac{1}{2}\sigma^2\frac{\rho_{yy}}{\rho}+ {\color{red}A\beta e^{(y-x)(1-\beta)}}\right]w = 0, ~\textrm{for}~ (x,y)\in \mathbb{R}^2, t\in [0,T],\\
		&w(T,x,y)= \frac{\tilde{u}_2(e^{y-x}) }{\rho}.
	\end{aligned}\right.
\end{equation}
The auxiliary function $\rho$ serves as a weight function and it captures the growth rate of $\tilde{\lambda}$ (the solution of Equation \eqref{eq: equation for tilde lamba}) in order to make $w=\frac{\tilde{\lambda}}{\rho}$ as a bounded function.
To derive the Schauder estimate for $w$, we shall  apply an  approximation scheme that truncates the unbounded coefficients.
The key is to obtain the  $L^{\infty}$ estimate uniformly in these  approximate solutions so that we can extract
a convergent subsequence from the approximate solutions.

To apply 
the  maximum principle, we need to ensure an upper bound for the coefficient of zeroth-order term of $w$. To this end, it is easy to see that we only need to provide an upper bound for  the sum of  red terms in \eqref{eq: the equation of w}, i.e., $-e^{y-x} j(\psi)\frac{\rho_x}{\rho} + Ae^{(y-x)(1-\beta)}\frac{\rho_x+\beta \rho}{\rho}$ since all the other terms are already bounded.
It could be done by  a detailed calculation with this particular choice of  $\rho$ using the bounds on $\psi$ stated in Theorem \ref{thm:existence when psi is smooth} and the bound on $j(x)$ in \eqref{eq:condition on j(x)} by considering the cases $\beta\geq \gamma$ and $\beta< \gamma$ separately, one can refer to \cite[Section 3.2]{HKUHUB_Hao} for more detailed calculations, and we omit the details here.
Hence, we obtain that  the coefficient of zeroth-order term of $w$ is bounded from above by some constant $M$ depending on  given parameters as follows:
\begin{equation}\label{eq:upper bound for zero-th oredr term}
	\begin{aligned}
		&\left( -e^{y-x}j(\psi)+ Ae^{(y-x)(1-\beta)} + \frac{1}{2}\epsilon^2\right)\frac{\rho_x}{\rho}
		+\left(\frac{f(e^y)}{e^y}-\frac{1}{2}\sigma^2\right)\frac{\rho_y}{\rho}\\
		&\quad+\frac{1}{2}\epsilon^2\frac{\rho_{xx}}{\rho}+\frac{1}{2}\sigma^2\frac{\rho_{yy}}{\rho}+A\beta e^{(y-x)(1-\beta)}\le M(\theta,\beta,\gamma,C_f,A,\epsilon, \sigma, a_1, c_0).
	\end{aligned}
\end{equation}
We can then  consider the approximate solution $w^m$ to the following approximate equation for Equation \eqref{eq: the equation of w} and derive its uniform $L^{\infty}$ bound being independent of $m$ with the aid of the classical maximum principle:
\begin{equation}\label{eq:Approximation wm}
	\left\{
	\begin{aligned}
		w^m_t&+\left(-e^{y-x}j(\psi)+ Ae^{(y-x)(1-\beta)} + \frac{1}{2}\epsilon^2 +\epsilon^2\frac{\rho_x}{\rho}\right)\xi_mw^m_x \\
		& + \left(\frac{f(e^y)}{e^y}-\frac{1}{2}\sigma^2+\sigma^2\frac{\rho_y}{\rho}\right)\xi_mw^m_y
		+ \frac{1}{2}\epsilon^2 w^m_{xx}
		+ \frac{1}{2}\sigma^2w^m_{yy}\\
		& + \left(( -e^{y-x}j(\psi)+ Ae^{(y-x)(1-\beta)} + \frac{1}{2}\epsilon^2)\frac{\rho_x}{\rho}\right.\\
		&\left. \quad +\left(\frac{f(e^y)}{e^y}-\frac{1}{2}\sigma^2\right)\frac{\rho_y}{\rho}+\frac{1}{2}\epsilon^2\frac{\rho_{xx}}{\rho}+\frac{1}{2}\sigma^2\frac{\rho_{yy}}{\rho}+A\beta e^{(y-x)(1-\beta)}\right)\xi_mw^m = 0,\\
		&w^m(T,x,y)= \frac{\tilde{u}_2(e^{y-x}) }{\rho},
	\end{aligned}\right.
\end{equation} 
where $\xi_m$ for $m\ge$1 is a sequence of  smooth cutoff functions such that $\xi_m=1$ on the ball $\mathcal{B}_{\frac{m}{2}}(0)$ and vanishes outside the ball $\mathcal{B}_{m}(0)$ with bounded derivatives whose bounds
independent of $m$. 
The unique existence of solution $w_m$  in $C^{1+\frac{\alpha}{2},2+\alpha}([0,T] \times \mathbb{R}^2)$  for Equation \eqref{eq:Approximation wm}  is standard as  the coefficients are  all bounded, being at least Lipschitz continuous, and the terminal data is $C^3$.  It is also clear that the coefficient of the zeroth-order term of Equation \eqref{eq:Approximation wm} is  bounded from above by the same bound $M=M(\theta,\beta,\gamma,C_f,A,\epsilon, \sigma, a_1, c_0)$ in \eqref{eq:upper bound for zero-th oredr term} which is independent of $m$. 
Hence, it follows from the comparison principle that, by first  comparing with 0, we see that $w^m>0$.  
Clearly, $w$ is bounded as  $0<w(T,x,y)\leq a_3$  due to Assumption \ref{Assum: terminal utility function $u_2$} on $u_2$ (see also \eqref{Conditons on tilde u_2}). 
Then comparing with $e^{M(T-t)}M_1$ such that  $M_1:=\sup_{\mathbb{R}^2} \frac{\tilde{u}_2(e^{y-x}) }{\rho}$, we also have 
\begin{equation}\label{eq:uniform upper bound for wm}
	0<w^m\le e^{M(T-t)}M_1, ~M=M(\theta,\beta,\gamma,C_f,A,\epsilon, \sigma, a_1, c_0).
\end{equation} 
Using the classical  interior Schauder estimates (cf. \cite{MR0181836}), 
and the uniform $L^{\infty}$ bound of $w^m$ we obtained in \eqref{eq:uniform upper bound for wm},   we can use a diagonal argument to find a subsequence  $\{w^{m_k}\}_{k=0}^{\infty}$ such that $w^{m_k}$ converge to some $w$ in $C^{1+\frac{\alpha}{2},2+\alpha}_{loc}([0,T] \times\mathbb{R}^2)$ for any $0<\alpha< 1$. We can now conclude that Equation \eqref{eq: the equation of w}, and so, Equation \eqref{eq: equation for tilde lamba} has a   classical solution in $C^{1+\frac{\alpha}{2},2+\alpha}_{loc}([0,T] \times\mathbb{R}^2)$ for any $\alpha< 1$. Besides, it is smooth as long as $j$, $\psi$, and $f$ are smooth
by the standard regularity theory; see \cite{MR0181836} for instance.
\begin{rem}\label{rem:comment on maximum principle}
	
	The conventional maximum principle helps  in establishing an upper bound but not so for getting a  more precise lower bound other than the trivial bound 0, 
	without which it is impossible to ensure the well-posedness of
	the  nonlinear problem \eqref{eq: differential equation for original $lambda$}. We shall obtain more accurate lower and upper bounds for the solution $\lambda$ below in next Section \ref{section: pointwise estimate} through Theorem \ref{thm:crucial estimate theorem} and Theorem \ref{thm:crucial estimate theorem for case beta<gamma}.
\end{rem}
\subsection{A modified maximum principle for the linearized equation  \eqref{eq: equation for tilde lamba} and   uniqueness of its solution}\label{sec:Maximum Principle and Uniqueness}

In this section, we aim to establish the uniqueness of the classical solution for Equation \eqref{eq: equation for tilde lamba} within a certain function class. For this purpose, we need a customized maximum principle to handle the growth rate. We begin with a lemma  asserting that if we already have the control,  uniformly in time for the solution at the far field, then the maximum principle holds.

\begin{lem}\label{lem: lemma of maximum principle}
	Consider the parabolic operator 
	\begin{equation}\label{eq:parabolic operator}
		Lu : = \sum_{i,j =1}^{n} a_{ij} (t,x)\frac{\partial ^2 u}{\partial x_i \partial x_j} + \sum_{i=1}^{n}b_i(t,x) \frac{\partial u}{\partial x_i} + c(t,x) u -\frac{\partial u}{\partial t},
	\end{equation}
	for $(t,x)\in (0,T] \times \mathbb{R}^n$. Assume that $a_{ij}$, $b_i$ and $c$ are continuous for $i,j=1,\cdots, n$. Moreover,   $c$ is assumed to be bounded from above.
	If $Lu \le 0$ on $(0,T] \times \mathbb{R}^n$, $u(0, x) \ge 0$ on $\mathbb{R}^n$ and 
	\begin{equation}\label{eq:one side control in maximum prin}
		\liminf_{|x| \to \infty} u(t,x) \ge 0
	\end{equation}
	uniformly in $t \in [0,T]$, then $u(t,x) \ge 0$ on $[0,T] \times \mathbb{R}^n$.
\end{lem}
This lemma is standard, we refer to  \cite[Section 2.4, Lemma 5]{MR0181836}.
\begin{thm}[Maximum principle]\label{thm:Maximum principle}
	Let $\psi$ satisfy the condition specified in Theorem \ref{thm:existence when psi is smooth}. Assume that $\tilde{\lambda}$ is a classical solution of Equation ${\eqref{eq: equation for tilde lamba}}_1$\footnote{Here, ${\eqref{eq: equation for tilde lamba}}_1$ means the first equation in ${\eqref{eq: equation for tilde lamba}}$. In general, we use $\textrm{(i.j)}_{\textrm{k}}$ to represent the k-th equation in System (i.j).}, and  $|\tilde{\lambda} (t,x, y)| \le B (e^{\beta_1(y-x)} + e^{\gamma_1(y-x)})$ for all $(x,y)\in \mathbb{R}^2$, $t \in [0,T]$ for some constants $0< \beta_1, \gamma_1<1$ and  $B>0$. If $\tilde{\lambda}(T, x,y) \ge 0$ on $\mathbb{R}^2$, then $\tilde{\lambda}(t, x,y) \ge 0$ on $[0,T] \times \mathbb{R}^2$.
\end{thm}

\begin{proof}
	The  validity of our claim relies on  a suitable choice of a weight function $\rho$, which should have faster growth at infinity than those assumed for $\tilde{\lambda}$.  Let $\rho = e^{ -x + 2y}  + e^{x-y} + e^{-x+y} +e^{-x-y}$. Consider $w = \frac{\tilde{\lambda}}{\rho}$ which satisfies ${\eqref{eq: the equation of w}}_1$ for this new choice $\rho$.
	After some elaborate calculations (cf. \cite[Pooof of Theorem 3.3.1]{HKUHUB_Hao}), one can show that  the coefficient of the zeroth-order term in  ${\eqref{eq: the equation of w}}_1$ is  still bounded from above no matter $\beta \ge \gamma$ or $\beta<\gamma$; moreover
	\begin{equation}\label{eq:uniform to 0 at far field}
		\lim_{R \to \infty} \sup_{|x| + |y| \ge R} |w| \to 0 ~ \textrm{uniformly in $t \in [0,T]$}, 
	\end{equation}
	due to the choice of this new $\rho$ and the  bound on $\tilde{\lambda}$ assumed in the statement, i.e., $|\tilde{\lambda} (t,x, y)| \le B (e^{\beta_1(y-x)} + e^{\gamma_1(y-x)})$.
	Applying Lemma \ref{lem: lemma of maximum principle}  (after a reversal of time), we have $w \ge 0$, so is $\tilde{\lambda}$.
\end{proof}

\begin{rem}
	In fact  Theorem \ref{thm:Maximum principle} remains valid under a weaker condition $\tilde{\lambda} (t,x, y) \ge -B (e^{\beta_1(y-x)} + e^{\gamma_1(y-x)})$ upon Lemma \ref{lem: lemma of maximum principle}.
\end{rem}
\begin{rem}
	The maximum principle provided in Theorem \ref{thm:Maximum principle} is designed for our parabolic operator with exponential growth coefficients,
	for which we need a more stringent growth rate on the solution compared with the classical result.
	That is,  if the growth rates of the coefficients are not so rapid, then the maximum principle holds for a broader class of functions. In particular,  if the coefficients $a_{ij}$'s of the parabolic operator in \eqref{eq:parabolic operator} are bounded, $b_i$'s   are of at most linear growth
	rate, and $c$ is of at most quadratic growth,
	then the maximum principle holds for  the class of solutions with at most $e^{k|x|^2}$ (for any $k>0$) growth rate in $x$; see   \cite[Section 2.4, Theorem 9 and Theorem 10]{MR0181836}.
\end{rem}
\begin{cor}[Uniqueness of the classical solution]\label{cor:Uniqueness of the classical solution}
	There exists at most one classical solution to the Cauchy problem \eqref{eq: equation for tilde lamba} satisfying 
	\begin{equation*}
		|\tilde{\lambda}(t,x,y)| \le B (e^{\beta_1(y-x)} + e^{\gamma_1(y-x)}),~ t \in [0,T],
	\end{equation*}
	for some  constants $B>0$ and  $0< \beta_1, \gamma_1<1$.
	\begin{proof}
		To show the uniqueness, we only need to show $\tilde{\lambda} =0$ when  the terminal data $\tilde{u}_2$ in the second equation of Problem $\eqref{eq: equation for tilde lamba}$ is replaced by 0.
		We  then apply Theorem \ref{thm:Maximum principle} for $\tilde{\lambda} $ and $-\tilde{\lambda} $ to deduce that.
	\end{proof}
\end{cor}

\section{Refined pointwise estimates for the solution to the linearized problem \eqref{eq: differential equation for $lambda$}}\label{section: pointwise estimate}


We derive the refined pointwise lower and upper estimates, which seem  almost optimal for the solution $\lambda$ to the linearized equation \eqref{eq: differential equation for $lambda$}. 
In particular, we shall establish a crucial lower bound for solutions to \eqref{eq: differential equation for $lambda$}  that enables the application of a fixed point argument to solve \eqref{eq: differential equation for original $lambda$}. To achieve this, we employ a combination of probabilistic and analytic methodologies, combining the strengths of both approaches.

The main tools of use are the celebrated Feynman-Kac formula with a delicate stochastic analysis for a pair of partially coupled diffusion processes with non-Lipschitz driving coefficients associated with the infinitesimal generator of Equation \eqref{eq: differential equation for $lambda$} and a sophisticated use of 
H\"older's and reverse H\"older's inequalities.
%
The main results in this section are the following.

\begin{thm}[Pointwise estimates when $\beta\ge\gamma$]\label{thm:crucial estimate theorem}
Let $\beta$, $\gamma$ $\in (0,1)$ such that  $\beta\ge\gamma$.
Let the input function  $\psi$  be a continuous function satisfying $\psi(t,K,N) \ge c_0^{\gamma}\left(\frac{N}{K}\right)^{\gamma} $ for some positive constant $c_0$ which is specified in \eqref{eq:value for c_0}. Assume that $\lambda \in  C^{1,2}( [0,T]\times \mathbb{R}_+^2 )$ is  a classical solution of the linearized problem \eqref{eq: differential equation for $lambda$}.  Under Assumptions \ref{Assum: running utility function $u_1$} and \ref{Assum: terminal utility function $u_2$} in Section \ref{sec:Model formulation},  the  solution $\lambda$  satisfies
\begin{equation*}
	c_0^{\gamma}\left(\frac{N}{K}\right)^{\gamma} \le\lambda(t,K,N) \le C_1\left(\frac{N}{K}\right)^{\gamma} + C_2\left(\frac{N}{K}\right)^{\beta},
\end{equation*}
where the positive constants $C_1$ and $C_2$ are given in \eqref{eq: values for C_1 and C_2 case1} and \eqref{eq: values for C_1 and C_2 when bera=gamma} when $\gamma \le 2\beta -1$ or in \eqref{eq: values for C_1 and C_2}  when  $\gamma > 2\beta -1$. 
\end{thm}

\begin{thm}[Pointwise estimates for $\beta<\gamma$]\label{thm:crucial estimate theorem for case beta<gamma}
Let $\beta$, $\gamma$ $\in (0,1)$ such that   $\beta<\gamma$.
Let the input function $\psi$  be a continuous function satisfying $\psi(t,K,N) \ge  C_3  \left(\frac{N}{K}\right)^{\beta}\left(  C_4 \left(\frac{N}{K}\right)^{\beta-1} + C_5\right)^{ \frac{\gamma-\beta}{\beta-1 } }$. Assume that $\lambda \in  C^{1,2}( [0,T]\times \mathbb{R}_+^2 )$ is a classical solution of the linearized problem \eqref{eq: differential equation for $lambda$}.  Under Assumptions \ref{Assum: running utility function $u_1$} and \ref{Assum: terminal utility function $u_2$} in Section \ref{sec:Model formulation}, then the  solution $\lambda$  satisfies
\begin{equation*}
	C_3  \left(\frac{N}{K}\right)^{\beta}\left(  C_4 \left(\frac{N}{K}\right)^{\beta-1} + C_5\right)^{ \frac{\gamma-\beta}{\beta-1 } }\le\lambda(t,K,N) \le C_6\left(\frac{N}{K}\right)^{\gamma}  ,
\end{equation*}
where the positive constants $C_3, C_4$ and $C_5$ are given in \eqref{eq:value for C3, C4, C5},  and $C_6$ is  given in 
\eqref{eq:value for C6}.
\end{thm}


The proof of Theorem \ref{thm:crucial estimate theorem} relies on Lemma \ref{lem:lower bound} (providing the lower bound when $\beta \ge \gamma$), Lemma \ref{lem:upper bound case1}, and Lemma \ref{lem:upper bound} (together providing the upper bound when $\beta \ge \gamma$). On the other hand, the proof of Theorem \ref{thm:crucial estimate theorem for case beta<gamma} involves Lemma \ref{lem: lower bound for lambda for beta<gamma} (giving the lower bound when $\beta < \gamma$) and Lemma \ref{lem: upper bound for lambda for beta<gamma} (giving the upper bound when $\beta < \gamma$).

To summarize, in Subsection \ref{Probabilistic representation for solution} we   derive a Feynman-Kac representation for the solution to Problem \eqref{eq: differential equation for $lambda$}. In Subsections \ref{Lower bound for the solution} and \ref{Upper bound for the solution}, we establish  the lower and upper bounds when $\beta \ge \gamma$, respectively; while in Subsection \ref{Upper bound for the solution when $beta < gamma$}, we derive the lower and upper bounds  in the case  $\beta < \gamma$.  All of them rely on the Feynman-Kac representation in Subsection  \ref{Probabilistic representation for solution}.
\subsection{Feynman-Kac  representation of the solution to the linearized problem \eqref{eq: differential equation for $lambda$}}\label{Probabilistic representation for solution}


Given the existence of the solution $\lambda$ of Equation \eqref{eq: differential equation for $lambda$}, we give its solution formula.
To this end, we adopt the celebrated Feynman-Kac formula to express it through some properly chosen diffusion processes. It is more feasible for us to derive pointwise estimates, particularly the more critical lower bounds. 


Specifically, we first define two processes $X_1$ and $X_2$, which are associated with the   infinitesimal generators of Equation \eqref{eq: differential equation for $lambda$}, with which the corresponding  Feynman-Kac formulas can be derived. We consider
\begin{equation}\label{eq:process of X_1}
\left\{
\begin{aligned}
	dX_1(s) =& \left[AX_2^{1-\beta}(s)X_1^{\beta}(s) -X_2(s)j\left(\psi\left(s,X_1(s),X_2(s)\right)\right) + {\epsilon}^2X_1(s)\right]ds 
	+ {\epsilon}X_1(s)dB_1(s),\\
	X_1(t) =& K,
\end{aligned}\right.
\end{equation} 
\begin{equation}\label{eq:process of X_2}
\left\{
\begin{aligned}
	dX_2(s) &= f(X_2(s))ds + {{\sigma}}X_2(s)dB_2(s),\\
	X_2(t) &= N,
\end{aligned}\right.
\end{equation}
where $B_1(s)$ and $B_2(s)$ are two independent Brownian motions on some suitable probability space.
There is a unique process $X_2$ satisfying \eqref{eq:process of X_2} since $f(x)$ is Lipschitz continuous. 
Moreover, since 
\begin{equation*}
\begin{aligned}
	dX_2(s) &\ge -C_fX_2ds + {{\sigma}}X_2(s)dB_2(s),\quad
	X_2(t) = N>0,
\end{aligned}
\end{equation*}
we deduce from the comparison principle that $X_2(s)> 0$ for all $s\in[t,T]$.
For the process $X_1$, it is not obvious that it has a  solution, so we need  to first show the solvability of $X_1$. 
\begin{lem}[Existence of process $X_1$]\label{Solvability of $X_1$}
There exists  a  solution $X_1$ of \eqref{eq:process of X_1}.
\end{lem}

To prove this, we adopt the idea of considering the stochastic differential equation (SDE) for $X_1^{1-\beta}$, which is motivated by the Bernoulli differential equation. Importantly, the corresponding SDE is indeed ``almost"  linear. The proof of Lemma \ref{Solvability of $X_1$}  will be provided in Appendix \ref{appen:Solvability of X1 and Probabilistic representation of lambda}.
Applying the Feynman-Kac formula, we can express $\lambda$ in an expectation form in terms of the processes $X_1$ and $X_2$ as follows.
\begin{lem}[Feynman-Kac formula for $\lambda$]\label{lem:Probabilistic representation for lambda}
For $\lambda$,  the classical solution to Problem \eqref{eq: differential equation for $lambda$}, we can write 
\begin{equation}\label{eq:Probabilistic representation for lambda}
	\lambda(t,K,N)	= \mathbb{E}\left[e^{\int_{t}^{T}A\beta(X_2(\tau))^{1-\beta}(X_1(\tau))^{\beta-1}d\tau}u'_2\left(\frac{X_1(T)}{X_2(T)}\right)\right].
\end{equation}
\end{lem}

Indeed, by considering $Y(s) = e^{\int_{t}^{s}A\beta(X_2(\tau))^{1-\beta}(X_1(\tau))^{\beta-1}d\tau} \lambda(s,X_1(s),X_2(s))$, we can establish this Lemma \ref{lem:Probabilistic representation for lambda}, and the details are  in Appendix  \ref{appen:Solvability of X1 and Probabilistic representation of lambda}.
We mainly aim to utilize \eqref{eq:Probabilistic representation for lambda} to derive both lower and upper pointwise bounds for $\lambda(t,K,N)$. Motivated by the expression \eqref{eq:Probabilistic representation for lambda}, to estimate
$\lambda(t, K,N)$,
we consider the process 
\begin{align}\label{eq:defofG}
G (s)&:=e^{\int_{t}^{s}A\beta(X_2(\tau))^{1-\beta}(X_1(\tau))^{\beta-1}d\tau} u'_2\left(\frac{X_1(s)}{X_2(s)}\right)=e^{\int_{t}^{s}A\beta(X_2(\tau))^{1-\beta}(X_1(\tau))^{\beta-1}d\tau} \tilde{u}_2\left(\frac{X_2(s)}{X_1(s)}\right),
\end{align}
where $\tilde{u}_2(x) = u'_2(\frac{1}{x})$, since we have $\lambda(t,K,N) = \mathbb{E}[G(T)]$.
Our main goal is then to estimate $\mathbb{E}[G(T)]$ by studying the process $G(s)$ and other related processes.
A direct computation using \eqref{eq:process of X_1} and \eqref{eq:process of X_2} shows that 
\begin{equation}\label{eq: process of X_2/X_1}
	\begin{aligned}
		d\left(\frac{X_2}{X_1}\right) &= \frac{1}{X_1}dX_2 + X_2d\left(\frac{1}{X_1}\right) + d(X_2)d\left(\frac{1}{X_1}\right)\\
		& = \frac{1}{X_1}\left(f(X_2)ds + {{\sigma}}X_2dB_2\right)  - \frac{X_2}{X_1^2}(AX_2^{1-\beta}X_1^{\beta} -X_2j(\psi) +\epsilon^2X_1)ds + \epsilon^2 X_2 X_1^{-1}ds -\epsilon\frac{X_2}{X_1}dB_1\\
		& = \frac{1}{X_1}f(X_2)ds - \left(A\left(\frac{X_2}{X_1}\right)^{2-\beta} -\left(\frac{X_2}{X_1}\right)^{2}j(\psi) \right)ds  + {{\sigma}}\frac{X_2}{X_1}dB_2(s) - \epsilon\frac{X_2}{X_1}dB_1(s).
	\end{aligned}
\end{equation}
From above,
the dynamics of $G$ can be easily obtained, and it will be used to  derive various  useful estimates. We have
\begin{prop}
For $0<s<T$,
\begin{equation}\label{eq:Process of G(s)}
	\begin{aligned}
		dG(s) &= A\left(\beta-\frac{\tilde{u}_2'\left(\frac{X_2}{X_1}\right) }{\tilde{u}_2\left(\frac{X_2}{X_1}\right) } \cdot \frac{X_2}{X_1}\right) \cdot \left(\frac{X_2}{X_1}\right)^{1-\beta}Gds \\
		&\quad+ e^{\int_{t}^{s}A\beta(X_2(\tau))^{1-\beta}(X_1(\tau))^{\beta-1}d\tau}\tilde{u}_2'\left(\frac{X_2}{X_1}\right) \left[\frac{1}{X_1}f(X_2)ds  +\left(\frac{X_2}{X_1}\right)^{2}j(\psi) ds\right] \\
		&\quad +   e^{\int_{t}^{s}A\beta(X_2(\tau))^{1-\beta}(X_1(\tau))^{\beta-1}d\tau}\left[ \frac{1}{2}\tilde{u}_2''\left(\frac{X_2}{X_1}\right)\left(\frac{X_2}{X_1}\right)^{2} ({{\sigma}^2}+ {\epsilon}^2)ds \right]\\ 
		&\quad  + e^{\int_{t}^{s}A\beta(X_2(\tau))^{1-\beta}(X_1(\tau))^{\beta-1}d\tau}\tilde{u}_2'\left(\frac{X_2}{X_1}\right)\left(\frac{X_2}{X_1}\right)\left({{\sigma}}dB_2 - \epsilon dB_1 \right).
	\end{aligned}
\end{equation}
\end{prop}
We  note that it follows from Assumption \ref{Assum: terminal utility function $u_2$} (i.e., conditions on $u_2$) that $\tilde{u}_2$ satisfies the following conditions:
\begin{equation}\label{Conditons on tilde u_2}
a_2x^{\gamma}\le\tilde{u}_2(x)\le a_3x^{\gamma}, ~
\min\{\beta, \gamma\} \tilde{u}_2(x)\leq x{\tilde{u}_2'}(x)\le \max\{\beta, \gamma\} \tilde{u}_2(x), ~
\tilde{a}_4\tilde{u}_2(x)\leq x^2\tilde{u}_2''(x)\le \tilde{a}_5\tilde{u}_2(x),
\end{equation}
where 
\begin{equation}\label{eq:constansa4a5}
\tilde{a}_4=a_4-2\max\{\beta, \gamma\} \textrm{ and } \tilde{a}_5=a_5-2\min\{\beta, \gamma\}
\end{equation}
are both constants.

\subsection{Lower bound for the solution in the case $\beta \ge \gamma$}\label{Lower bound for the solution} 
This section  derives a lower bound  for the  solution $\lambda$ to the linearized problem \eqref{eq: differential equation for $lambda$} when $\beta \ge \gamma$. Our main result is:
\begin{lem}[Lower bound for $\lambda$]\label{lem:lower bound}
Under the same assumptions  stated in Theorem \ref{thm:crucial estimate theorem}, we have 
\begin{equation*}
	\lambda(t,K,N)=\mathbb{E}[G(T)]  \ge a_2e^{(- C_f \beta + \frac{\tilde{a}_4}{2}({\sigma}^2+\epsilon^2))(T-t)}\left(\frac{N}{K}\right)^{\gamma}\ge
	c_0^{\gamma}\left(\frac{N}{K}\right)^{\gamma},
\end{equation*}
where the constant 
\begin{equation}\label{eq:value for c_0}
	c_0=\min{\left\{a_2^{\frac{1}{\gamma}},a_2^{\frac{1}{\gamma}}e^{{(- C_f\beta  +\frac{\tilde{a}_4}{2}({\sigma}^2+\epsilon^2))T}/{\gamma}}\right\}}.
\end{equation}
\end{lem}

\begin{proof}
It follows from \eqref{eq:Process of G(s)}, conditions \eqref{Conditons on tilde u_2}, Assumption \ref{Assum:drift rate in population model}, $j(\psi)>0$  and $\beta \ge \gamma$ that we have
\begin{equation*}
	dG(s)\ge \tilde{a}G(s)ds + e^{\int_{t}^{s}A\beta(X_2(\tau))^{1-\beta}(X_1(\tau))^{\beta-1}d\tau}  \tilde{u}_2'\left(\frac{X_2}{X_1}\right) \frac{X_2}{X_1} ({{\sigma}}dB_2 - \epsilon dB_1 ),
\end{equation*}
for the constant $\tilde{a} : =- C_f \beta + \frac{\tilde{a}_4}{2}({\sigma}^2+\epsilon^2)$.
Integrating from $t$ to $s$ yields 
\begin{equation*}
	\begin{aligned}
		G(s)-G(t)\ge& \int_t^s \tilde{a} G(\tau)d\tau + \int_{t}^{s} e^{\int_{t}^{\tau}A\beta(X_2(\tau))^{1-\beta}(X_1(\tau))^{\beta-1}d\tau}  \tilde{u}_2'\left(\frac{X_2(\tau)}{X_1(\tau)}\right)\frac{X_2}{X_1}(\tau) ({{\sigma}}dB_2 - \epsilon dB_1 ).
	\end{aligned}
\end{equation*}
Taking expectations on both sides and using the initial conditions $\eqref{eq:process of X_1}_2$ and $\eqref{eq:process of X_2}_2$, we have 
\begin{equation*}
	\begin{aligned}
		&\mathbb{E}[G(s)]-a_2\left(\frac{N}{K}\right)^{\gamma} 
		\ge\mathbb{E}[G(s)-G(t)]\\
		\ge& \mathbb{E}\left[\int_t^s \tilde{a}G(\tau)d\tau + \int_{t}^{s}e^{\int_{t}^{\tau}A\beta\left(\frac{X_2}{X_1}(\tau)\right)^{1-\beta}d\tau}  \tilde{u}_2'\left(\frac{X_2(\tau)}{X_1(\tau)}\right)\left(\frac{X_2	(\tau)}{X_1	(\tau)}
	\right) ({{\sigma}}dB_2 - \epsilon dB_1 )\right] = \mathbb{E}\left[\int_t^s \tilde{a}G(\tau)d\tau \right].
	\end{aligned}
\end{equation*}
It follows from  the  integral form of the Gr\"onwall  inequality to deduce that $
\mathbb{E}[G(s)] \ge a_2e^{\tilde{a}(s-t)}\left(\frac{N}{K}\right)^{\gamma}$.
In particular, evaluating at $s=T$, and noting the definition of $c_0$, we conclude the proof.
\end{proof}	

\begin{rem}\label{rem: remark for lower bound beta>gamma}
In the proof of Lemma \ref{lem:lower bound}, we do not use the explicit lower bound or upper bound of the input $\psi$ but only require that $\beta \ge \gamma$ and $\psi > 0$.  In this sense, the lower bound of  $\lambda$ mainly comes from the terminal data and $\beta \ge \gamma$.
\end{rem}

\subsection{Upper bound for the solution in the case $\beta \ge \gamma$}\label{Upper bound for the solution}

We next derive an upper bound for $\mathbb{E}[G(s)]$ when $\beta \ge \gamma$. 
To do so, we first  consider the process
\begin{equation}\label{eq:process h}
h(s)=e^{\int_{t}^{s}A(\gamma-\beta)(X_2(\tau))^{1-\beta}(X_1(\tau))^{\beta-1}d\tau}G(s)=e^{\int_{t}^{s}A\gamma(X_2(\tau))^{1-\beta}(X_1(\tau))^{\beta-1}d\tau}\tilde{u}_2\left(\frac{X_2(s)}{X_1(s)}\right),
\end{equation}
and its variants, then use $\mathbb{E}[h(s)]$ and  its variants to bound $\mathbb{E}[G(s)]$. Let us begin with the following estimate for $\mathbb{E}[h(s)]$. The estimates for its variants  can be obtained similarly.
\begin{lem}[Estimate for process $h(s)$]\label{lem:Estimate for process h(s)}	
Under the same assumptions stated in Theorem \ref{thm:crucial estimate theorem}, we have 
\begin{equation}\label{eq:estimatesforh}
	\mathbb{E}[h(s)] \le {a}_3e^{{a}(s-t)}\left(\frac{N}{K}\right)^{\gamma},
\end{equation}
where the constant $a:=C_f\beta  +  \frac{\tilde{a}_5}{2}({\sigma}^2+\epsilon^2)+\frac{a_1\beta}{c_0}$ with $c_0$ given in \eqref{eq:value for c_0}.
\end{lem}

\begin{proof}
It follows from a direct computation using \eqref{eq: process of X_2/X_1} that
\begin{equation}\label{eq: process of h(s)}
	\begin{aligned}
		dh(s)
		&= A\left(\gamma-\frac{\tilde{u}'_2\left(\frac{X_2}{X_1}\right)}{\tilde{u}_2\left(\frac{X_2}{X_1}\right)}\cdot\frac{X_2}{X_1}\right)X_2^{1-\beta}X_1^{\beta-1}hds\\
		&	\quad+  e^{\int_{t}^{s}A\gamma(X_2(\tau))^{1-\beta}(X_1(\tau))^{\beta-1}d\tau}\tilde{u}_2'\left(\frac{X_2}{X_1}\right) \left[\frac{1}{X_1}f(X_2)ds  +\left(\frac{X_2}{X_1}\right)^{2}j(\psi) ds\right] \\
		&\quad +   e^{\int_{t}^{s}A\gamma(X_2(\tau))^{1-\beta}(X_1(\tau))^{\beta-1}d\tau}\left[ \frac{1}{2}\tilde{u}_2''\left(\frac{X_2}{X_1}\right)\left(\frac{X_2}{X_1}\right)^{2} ({{\sigma}^2}+ {\epsilon}^2)ds  \right]\\ 
		&\quad +   e^{\int_{t}^{s}A\gamma(X_2(\tau))^{1-\beta}(X_1(\tau))^{\beta-1}d\tau}\left[ \tilde{u}_2'\left(\frac{X_2}{X_1}\right)\left(\frac{X_2}{X_1}\right)\left({{\sigma}}dB_2 - \epsilon dB_1 \right)\right]\\ 
		&\le  e^{\int_{t}^{s}A\gamma(X_2(\tau))^{1-\beta}(X_1(\tau))^{\beta-1}d\tau}\tilde{u}_2'\left(\frac{X_2}{X_1}\right) \left[\frac{1}{X_1}f(X_2)ds  +\left(\frac{X_2}{X_1}\right)^{2}j(\psi) ds\right] \\
		&\quad +   e^{\int_{t}^{s}A\gamma(X_2(\tau))^{1-\beta}(X_1(\tau))^{\beta-1}d\tau}\left[ \frac{1}{2}\tilde{u}_2''\left(\frac{X_2}{X_1}\right)\left(\frac{X_2}{X_1}\right)^{2} ({{\sigma}^2}+ {\epsilon}^2)ds  \right]\\ 
		&\quad +   e^{\int_{t}^{s}A\gamma(X_2(\tau))^{1-\beta}(X_1(\tau))^{\beta-1}d\tau}\left[ \tilde{u}_2'\left(\frac{X_2}{X_1}\right)\left(\frac{X_2}{X_1}\right)\left({{\sigma}}dB_2 - \epsilon dB_1 \right)\right],\\ 
	\end{aligned}
\end{equation}
where  last inequality holds due to  $\frac{\tilde{u}'_2\left(\frac{X_2}{X_1}\right)}{\tilde{u}_2\left(\frac{X_2}{X_1}\right)}\cdot \frac{X_2}{X_1}\ge \gamma$ which follows from  \eqref{Conditons on tilde u_2}. 
Using  \eqref{Conditons on tilde u_2} and \eqref{eq:condition on j(x)}, we obtain
\begin{equation}\label{eq:estimate for a term in h(s)}
	\begin{aligned}
		&\quad\tilde{u}_2'\left(\frac{X_2}{X_1}\right)\cdot\left(\frac{X_2}{X_1}\right)^{2}j(\psi)=\tilde{u}_2'\left(\frac{X_2}{X_1}\right)\frac{X_2}{X_1}\frac{X_2}{X_1}j(\psi)
		\le\beta \tilde{u}_2\left(\frac{X_2}{X_1}\right)\frac{X_2}{X_1}j(\psi)\\
		&\le \beta \tilde{u}_2\left(\frac{X_2}{X_1}\right)\frac{X_2}{X_1}a_1\psi^{-\frac{1}{\gamma}}
		\le a_1\beta \tilde{u}_2\left(\frac{X_2}{X_1}\right)\frac{X_2}{X_1}\left(c_0^{\gamma}\left(\frac{X_2}{X_1}\right)^{\gamma}\right)^{-\frac{1}{\gamma}}
		\le \frac{a_1\beta }{c_0}\tilde{u}_2\left(\frac{X_2}{X_1}\right).
	\end{aligned}
\end{equation}
Now, using  the above estimate \eqref{eq:estimate for a term in h(s)}, the conditions \eqref{Conditons on tilde u_2} and Assumption \ref{Assum:drift rate in population model} into  Equation \eqref{eq: process of h(s)}, we  conclude  that
\begin{equation*}
	dh(s) \le ah(s)ds + e^{\int_{t}^{s}A\beta(X_2(\tau))^{1-\beta}(X_1(\tau))^{\beta-1}d\tau}  \tilde{u}_2'\left(\frac{X_2}{X_1}\right)\left(\frac{X_2}{X_1}\right) ({{\sigma}}dB_2 - \epsilon dB_1 ),\\
\end{equation*}
where the constant 
$a:=C_f\beta  +  \frac{\tilde{a}_5}{2}({\sigma}^2+\epsilon^2)+\frac{{a_1\beta}}{c_0}$.
Using a similar argument as  that in the proof of Lemma \ref{lem:lower bound} and noting that $h(t)\le {a}_3\left(\frac{N}{K}\right)^{\gamma}$ due to \eqref{Conditons on tilde u_2}, we finally obtain \eqref{eq:estimatesforh}.
\end{proof}
Now we derive the upper bound for $\mathbb{E}[G(s)]$ when $\beta\geq \gamma$ in the two distinct subcases: $\gamma \leq 2\beta - 1$ and $\gamma > 2\beta - 1$. 
In the case when $\gamma \leq 2\beta - 1$, we can control the nonlinear term $\left(\frac{X_2}{X_1}\right)^{1-\beta+\gamma}$ by using a linear combination of $\left(\frac{X_2}{X_1}\right)^{\beta}$ and $\left(\frac{X_2}{X_1}\right)^{\gamma}$, such as \eqref {eq:linear combination to controll the process} below, we can then obtain the unconventional upper bound; see  Lemma \ref{lem:upper bound case1}.
However, when $\gamma > 2\beta - 1$, \eqref {eq:linear combination to controll the process} does not hold, the proof for Lemma \ref{lem:upper bound case1} will not be applicable anymore; we need to study the relationship between processes $G$ and $h$ carefully, for this purpose we shall analyse the process $\left(\frac{X_2}{X_1}\right)^{\beta-1}$, which indeed satisfies an ``almost'' linear SDE, see Equation \eqref{process of (X_2/X_1)^(beta-1)} below. For technical details
for obtaining the upper bound in this case, see  Lemma \ref{lem:Estimate for frac{X_2}{X_1}(s)^{beta -1}} and Lemma \ref{lem:upper bound}.

\begin{lem}[Upper bound for $\lambda$ when $\gamma \le 2\beta -1$]\label{lem:upper bound case1}
Under the same assumptions stated in Theorem \ref{thm:crucial estimate theorem} and  $\gamma \le 2\beta -1$, we have
\begin{equation*}
	\lambda(t,K,N) =\mathbb{E}[G(T)]\le C_1\left(\frac{N}{K}\right)^{\gamma} + C_2\left(\frac{N}{K}\right)^{\beta}.
\end{equation*}
When $\beta>\gamma$, the constants $C_1$ and $C_2$ above are given by
\begin{equation}\label{eq: values for C_1 and C_2 case1}
	\begin{aligned}
		C_1 :&= \max \left\{ a_3,  a_3e^{b_1T}\right\}, \quad
		C_2: =  \max \left\{  b_2 e^{b_1T}\frac{e^{b_3T}  -1 }{b_3},  b_2 \frac{e^{b_3T}  -1 }{b_3}\right\},~
		\textrm{where} \\
		b_1: &=  (\beta -\gamma) A  \frac{a_3}{a_2}+C_f\beta  + \frac{{a_1\beta}}{c_0} + \frac{\tilde{a}_5}{2}({\sigma}^2+\epsilon^2),\\
		b_2: &=  (\beta -\gamma) A a_3C(\beta, \gamma), \quad
		b_3: =  C_f\beta+ \frac{a_1}{c_0}\beta + \frac{\beta(\beta-1)}{2} (\sigma^2+\epsilon^2),
	\end{aligned}
\end{equation}
and $C(\beta ,\gamma)$ is any positive constant so that $1\le C(\beta, \gamma) x^{2\beta-1-\gamma} + x^{\beta-1}$ holds for  all $x>0$.
And when $\beta=\gamma$, $C_1$ and $C_2$ are given by 
\begin{equation}\label{eq: values for C_1 and C_2 when bera=gamma}
	\begin{aligned}
		C_1 =a_3 e^{{a}T}, ~C_2=0, \textrm{and}~ a=C_f\beta  +  \frac{\tilde{a}_5}{2}({\sigma}^2+\epsilon^2)+\frac{a_1\beta}{c_0}.
	\end{aligned}
\end{equation}

\end{lem}
\begin{proof}
	   Assume that $\beta> \gamma$, otherwise if $\beta=\gamma$,  we have  $\lambda(t,K,N)=\mathbb{E}[G(T)] = \mathbb{E}[h(T)] \le a_3 e^{{a}(T-t)}\left(\frac{N}{K}\right)^{\gamma}$, by then the lemma follows. 		Since $\beta \ge 2\beta -1 \ge \gamma$, we know $\beta \ge 1- \beta +\gamma \ge \gamma$ and by the definition of $C(\beta, \gamma)$, we have
\begin{equation}\label{eq:linear combination to controll the process}
	\left(\frac{X_2}{X_1}\right)^{1-\beta+\gamma} \le C(\beta, \gamma) \left(\frac{X_2}{X_1}\right)^{\beta} + \left(\frac{X_2}{X_1}\right)^{\gamma}.
\end{equation}
Let us now define 
\begin{equation}\label{eq:eqforh_1}
	h_1(s): = e^{\int_{t}^{s}A\beta (X_2(\tau))^{1-\beta}(X_1(\tau))^{\beta-1}d\tau} \left(\frac{X_2(s)}{X_1(s)}\right)^{\beta}.
\end{equation}
Then using the definition of $G(s)$ and $h_1(s)$, i.e., Equations \eqref{eq:defofG} and  \eqref{eq:eqforh_1} and  combining conditions \eqref{Conditons on tilde u_2}  on $\tilde{u}_2$ and  the inequality \eqref{eq:linear combination to controll the process}, we have
\begin{equation}\label{eq:estimatesX^{1-beta}G}
	\begin{aligned}
		\left(\frac{X_2}{X_1}\right)^{1-\beta}(s)G(s) &=  e^{\int_{t}^{s}A\beta(X_2(\tau))^{1-\beta}(X_1(\tau))^{\beta-1}d\tau} \tilde{u}_2\left(\frac{X_2(s)}{X_1(s)}\right)\left(\frac{X_2}{X_1}\right)^{1-\beta}\\
		&\leq a_3e^{\int_{t}^{s}A\beta(X_2(\tau))^{1-\beta}(X_1(\tau))^{\beta-1}d\tau} \left(\frac{X_2}{X_1}\right)^{1-\beta+\gamma}\\
		&\leq a_3e^{\int_{t}^{s}A\beta(X_2(\tau))^{1-\beta}(X_1(\tau))^{\beta-1}d\tau} \left(C(\beta, \gamma) \left(\frac{X_2}{X_1}\right)^{\beta} + \left(\frac{X_2}{X_1}\right)^{\gamma}\right)\\
		&\leq a_3e^{\int_{t}^{s}A\beta(X_2(\tau))^{1-\beta}(X_1(\tau))^{\beta-1}d\tau} \left( C(\beta, \gamma) \left(\frac{X_2}{X_1}\right)^{\beta} + \frac{1}{a_2} \tilde{u}_2\left(\frac{X_2(s)}{X_1(s)}\right)\right) \\
		&= a_3 \left( \frac{1}{a_2}G(s) + C(\beta, \gamma) h_1(s) \right). 
	\end{aligned}
\end{equation}
Hence, 
it follows from Equations   \eqref{eq:Process of G(s)}, \eqref{eq:estimatesX^{1-beta}G} and  conditions \eqref{Conditons on tilde u_2}  on $\tilde{u}_2$  that 
\begin{equation}\label{eq:diffeqforG}
	\begin{aligned}
		dG(s) &\le A\left(\beta-\gamma\right)\left(\frac{X_2}{X_1}\right)^{1-\beta}(s)G(s)ds \\
		&\quad+ e^{\int_{t}^{s}A\beta(X_2(\tau))^{1-\beta}(X_1(\tau))^{\beta-1}d\tau}\tilde{u}_2'\left(\frac{X_2(s)}{X_1(s)}\right) \left[\frac{1}{X_1}f(X_2)ds  +\left(\frac{X_2}{X_1}\right)^{2}j(\psi) ds\right] \\
		&\quad +   e^{\int_{t}^{s}A\beta(X_2(\tau))^{1-\beta}(X_1(\tau))^{\beta-1}d\tau}\left[ \frac{1}{2}\tilde{u}_2''\left(\frac{X_2(s)}{X_1(s)}\right)\left(\frac{X_2}{X_1}\right)^{2} ({{\sigma}^2}+ {\epsilon}^2)ds \right]\\ 
		&\quad  + e^{\int_{t}^{s}A\beta(X_2(\tau))^{1-\beta}(X_1(\tau))^{\beta-1}d\tau}\tilde{u}_2'\left(\frac{X_2(s)}{X_1(s)}\right)\left(\frac{X_2}{X_1}\right)\left({{\sigma}}dB_2 - \epsilon dB_1 \right)\\
		&\le Aa_3 \left(\beta-\gamma\right) \left( \frac{1}{a_2}G(s) + C(\beta, \gamma) h_1(s) \right) ds \\
		&\quad+ e^{\int_{t}^{s}A\beta(X_2(\tau))^{1-\beta}(X_1(\tau))^{\beta-1}d\tau}\tilde{u}_2'\left(\frac{X_2(s)}{X_1(s)}\right) \left[\frac{1}{X_1}f(X_2)ds  +\left(\frac{X_2}{X_1}\right)^{2}j(\psi) ds\right] \\
		&\quad +   e^{\int_{t}^{s}A\beta(X_2(\tau))^{1-\beta}(X_1(\tau))^{\beta-1}d\tau}\left[ \frac{1}{2}\tilde{u}_2''\left(\frac{X_2(s)}{X_1(s)}\right)\left(\frac{X_2}{X_1}\right)^{2} ({{\sigma}^2}+ {\epsilon}^2)ds \right]\\ 
		&\quad  + e^{\int_{t}^{s}A\beta(X_2(\tau))^{1-\beta}(X_1(\tau))^{\beta-1}d\tau}\tilde{u}_2'\left(\frac{X_2(s)}{X_1(s)}\right)\left(\frac{X_2}{X_1}\right)\left({{\sigma}}dB_2 - \epsilon dB_1 \right).\\
	\end{aligned}
\end{equation}

We now integrate \eqref{eq:diffeqforG} from $t$ to $s$, and then take expectations on both sides, by using the conditions \eqref{Conditons on tilde u_2}, Assumption \ref{Assum:drift rate in population model}  and  the initial conditions  $\eqref{eq:process of X_1}_2$ and $\eqref{eq:process of X_2}_2$, we have
\begin{equation}\label{eq:upperdiffeqofG}
	\begin{aligned}
		&\mathbb{E}[G(s)]- a_3\left(\frac{N}{K}\right)^{\gamma} 
		\le \mathbb{E}[G(s)-G(t)] \\
		\le& \int_{t}^{s}  \left(  Aa_3 (\beta -\gamma) \frac{1}{a_2}+C_f\beta  + \frac{{a_1\beta}}{c_0} + \frac{\tilde{a}_5}{2}({\sigma}^2+\epsilon^2)\right) \mathbb{E}[G( \tau )] 
		+ Aa_3(\beta -\gamma)C(\beta, \gamma)\mathbb{E}[h_1(\tau )] d\tau\\
		=: & \int_{t}^{s} b_1 \mathbb{E}[G( \tau )] + b_2 \mathbb{E}[h_1(\tau )] d\tau.
	\end{aligned}
\end{equation}
Following   the same argument as that for Lemma \ref{lem:Estimate for process h(s)}, we can first obtain 
\begin{equation}\label{eq:estiforh1}
\mathbb{E}[h_1(s)] \le e^{b_3(s-t)}\left(\frac{N}{K}\right)^{\beta}, ~ s \in [t, T],
\end{equation}
where $ b_3:= C_f\beta+ \frac{a_1}{c_0}\beta + \frac{\beta(\beta-1)}{2} (\sigma^2+\epsilon^2) $.
Finally, applying Gr\"onwall inequality to  \eqref{eq:upperdiffeqofG} and using \eqref{eq:estiforh1}, we have
\begin{equation}
	\begin{aligned}
		\mathbb{E}[G(s)] \le  e^{b_1(s-t)}
		\left( b_2\frac{e^{b_3(s-t)}  -1 }{b_3}\left(\frac{N}{K}\right)^{\beta} + a_3\left(\frac{N}{K}\right)^{\gamma} \right),
	\end{aligned}
\end{equation}
then  choosing $s= T$,
we obtain
\begin{equation}
	\begin{aligned}
		\lambda(t, K,N) &= \mathbb{E}[G(T)] \le e^{b_1(T-t)}
		\left( b_2\frac{e^{b_3(T-t)}  -1 }{b_3}\left(\frac{N}{K}\right)^{\beta} + a_3\left(\frac{N}{K}\right)^{\gamma} \right)
		\le  C_1 \left(\frac{N}{K}\right)^{\gamma} + C_2 \left(\frac{N}{K}\right)^{\beta}.
	\end{aligned}
\end{equation}
This finishes the proof.
\end{proof}

Next, we consider the subcase when $\gamma>2\beta-1$, and as mentioned above, we analyze    $\left(\frac{X_2}{X_1}\right)^{\beta-1}$ first.
It leads to the following result.
\begin{lem}[$\text{Estimate for} \left(\frac{X_2(s)}{X_1(s)}\right)^{\beta -1}$ ] \label{lem:Estimate for frac{X_2}{X_1}(s)^{beta -1}}
We have
\begin{equation*}
	\mathbb{E}\left[\left(\frac{X_2(s)}{X_1(s)}\right)^{\beta -1}\right] \le e^{b(s-t)}\left(\frac{N}{K}\right)^{\beta-1} + A(1-\beta)\frac{ e^{b(s-t)}-1}{b},
\end{equation*}
where $b: = (\beta-1) \left(-C_f +\frac{\beta -2}{2}({{\sigma}^2} + {\epsilon}^2)\right)>0$.
\end{lem}

\begin{proof}
 From a direct computation using \eqref{eq: process of X_2/X_1}, we see that 
\begin{equation}\label{process of (X_2/X_1)^(beta-1)}
	\begin{aligned}
		d\left(\frac{X_2}{X_1}\right)^{\beta-1} 
		& = (\beta-1 )\left(\frac{X_2}{X_1}\right)^{\beta-2}\left\{ \frac{f(X_2)}{X_1} - \left(A\left(\frac{X_2}{X_1}\right)^{2-\beta} - \left(\frac{X_2}{X_1}\right)^{2}j(\psi)\right)  \right\} ds\\
		&\quad +  (\beta-1 )\left(\frac{X_2}{X_1}\right)^{\beta-1}\left({{\sigma}}dB_2 - \epsilon dB_1\right)   +\frac{1}{2}(\beta-1)(\beta-2)\left(\frac{X_2}{X_1}\right)^{\beta-1} \left( {{\sigma}^2} + {\epsilon}^2\right)ds.\\
	\end{aligned}
\end{equation}
Using Assumption \ref{Assum:drift rate in population model}, we have $f(X_2)\ge -C_fX_2$. Following \eqref{eq:condition on j(x)}, we also have the fact     that $j(\psi)>0$. Hence,
\begin{equation*}
	\begin{aligned}
		d\left(\frac{X_2}{X_1}\right)^{\beta-1} &\le (\beta-1) \left(\frac{X_2}{X_1}\right)^{\beta-1}\left(-C_f+0\right)ds + A(1-\beta)ds\\
		&\quad +  \frac{1}{2}(\beta-1)(\beta-2)\left(\frac{X_2}{X_1}\right)^{\beta-1} \left( {{\sigma}^2} + {\epsilon}^2\right)ds
		+(\beta-1) \left(\frac{X_2}{X_1}\right)^{\beta-1}({{\sigma}}dB_2-{\epsilon}dB_1).
	\end{aligned}
\end{equation*}
Let $y(s)=\mathbb{E}\left[\left(\frac{X_2}{X_1}(s)\right)^{\beta -1}\right]$. Then, it follows from the initial conditions  $\eqref{eq:process of X_1}_2$ and $\eqref{eq:process of X_2}_2$ that   $y(t)=\left(\frac{N}{K}\right)^{\beta-1}$. So  directly integrating and taking  expectations on both sides, it yields 
\begin{equation*}
	\begin{aligned}
		y(s) -\left(\frac{N}{K}\right)^{\beta-1}
		&\le \int_{t}^{s} b\cdot y(\tau)d\tau + A(1-\beta)(s-t),
	\end{aligned}
\end{equation*}
where $b= (\beta-1)(-C_f +\frac{\beta -2}{2}({{\sigma}^2} + {\epsilon}^2))>0,$ since $\beta<1$. We now use the integral form of  Gr\"onwall's inequality to conclude that
\begin{equation*}
	\begin{aligned}
		y(s)&\le A(1-\beta)(s-t) + \left(\frac{N}{K}\right)^{\beta-1} + \int_{t}^{s}\left(A(1-\beta)(\tau-t) + \left(\frac{N}{K}\right)^{\beta-1} \right)\cdot b e^{\int_{\tau}^{s}bdr} d\tau\\
		& = e^{b(s-t)}\left(\frac{N}{K}\right)^{\beta-1} + A(1-\beta)\frac{ e^{b(s-t)} -1 }{b}. 
	\end{aligned}
\end{equation*}
\end{proof}
\begin{lem}[Upper bound for $\lambda$ when $\gamma > 2\beta -1$]\label{lem:upper bound}
Under the same assumptions in Theorem \ref{thm:crucial estimate theorem}, and that $\gamma >2\beta -1$, 
\begin{equation*}
	\lambda(t,K,N) =\mathbb{E}[G(T)]\le C_1\left(\frac{N}{K}\right)^{\gamma} + C_2\left(\frac{N}{K}\right)^{\beta},
\end{equation*}
where the constants $C_1$ and $C_2$ are given by
\begin{equation}\label{eq: values for C_1 and C_2}
	\begin{aligned}
		&C_1: = a_3 \max\left\{1,   \exp\left({\frac{a_{p(\beta-\gamma)}T}{p}}+\frac{a_{p\beta}T}{p} + {b(\gamma-\beta)T}/{(\beta-1)}\right) \right\} ,\\
		&C_2 :=  a_3\max\left\{\left(A(1-\beta)\frac{ e^{bT}-1}{b} \right)^{\frac{\gamma-\beta}{\beta-1}},~  \exp\left({\frac{a_{p(\beta-\gamma)}T}{p}}+{\frac{a_{p\beta}T}{p} }\right)\left(A(1-\beta)\frac{e^{bT}-1}{b} \right)^{\frac{\gamma-\beta}{\beta-1}} \right\}.
			\end{aligned}
	\end{equation}
	Here, we denote
			\begin{equation}\label{eq: values for a_{p(beta-gamma)} and a_{pbeta}}
					\begin{aligned}
				&a_{p(\beta-\gamma)}:=C_f{p(\beta-\gamma)}  + \frac{1}{2}{{\sigma}^2}{p(\beta-\gamma)}({p(\beta-\gamma)}-1)+\frac{1}{2}{\epsilon ^2}{p(\beta-\gamma)}({p(\beta-\gamma)}-1)+\frac{a_1{p(\beta-\gamma)} }{c_0},\\
				&a_{p\beta}:=C_f{p\beta}  + \frac{1}{2}({\sigma}^2+\epsilon^2)(p(p-1)\beta^2+p\tilde{a}_5)+\frac{ {a_1}{p\beta} }{c_0},~ \textrm{and finally p is chosen as}~p:=2\frac{1-\beta}{1+\gamma-2\beta}.
			\end{aligned}
		\end{equation}
	\end{lem}
	\begin{proof}
		If we look at the process $h(s)$ defined in \eqref{eq:process h}, it has a special structure; we rewrite $\mathbb{E}[u_2(s)]$ as follows to use this special structure and then   use the H\"older's inequality to derive some relevant bounds. We have
		\begin{equation}\label{eq:split of expectation of G(s)}
			\begin{aligned}
				\mathbb{E}[G(s)]&=\mathbb{E}\left[e^{\int_{t}^{s}A\beta(X_2(\tau))^{1-\beta}(X_1(\tau))^{\beta-1}d\tau}\tilde{u}_2\left(\frac{X_2(s)}{X_1(s)}\right)\right]\\
				& = \mathbb{E}\left[e^{\int_{t}^{s}A(\beta-\gamma)(X_2(\tau))^{1-\beta}(X_1(\tau))^{\beta-1}d\tau}\left(\frac{X_2(s)}{X_1(s)}\right)^{\beta-\gamma}\right.\\ 
				&\qquad\cdot \left.e^{\int_{t}^{s}A\gamma(X_2(\tau))^{1-\beta}(X_1(\tau))^{\beta-1}d\tau}\tilde{u}_2\left(\frac{X_2(s)}{X_1(s)}\right) \cdot\left(\frac{X_2(s)}{X_1(s)}\right)^{\gamma-\beta}\right]\\
				&\le \left( \mathbb{E}\left[e^{\int_{t}^{s}Ap(\beta-\gamma)(X_2(\tau))^{1-\beta}(X_1(\tau))^{\beta-1}d\tau}\left(\frac{X_2(s)}{X_1(s)}\right)^{p(\beta-\gamma)}\right]\right)^{\frac{1}{p}}\\
				&\quad\cdot
				\left(\mathbb{E}\left[e^{\int_{t}^{s}Ap\gamma(X_2(\tau))^{1-\beta}(X_1(\tau))^{\beta-1}d\tau}\tilde{u}_2^{p}\left(\frac{X_2(s)}{X_1(s)}\right)\right]\right)^{\frac{1}{p}}\cdot 	\left(\mathbb{E}\left[\left(\frac{X_2(s)}{X_1(s)}\right)^{q(\gamma-\beta)}\right]\right)^{\frac{1}{q}},
			\end{aligned}
		\end{equation}
		where $p,q>1$ and $\frac{2}{p}+\frac{1}{q}=1$.
		In the last inequality, we have used H\"older's inequality and $p,q$ will be determined later. Now, following the same argument  in the proof of  Lemma \ref{lem:Estimate for process h(s)}, we can easily obtain upper bounds for $\mathbb{E}\left[e^{\int_{t}^{s}Ap(\beta-\gamma)\left(\frac{X_2}{X_1}(\tau)\right)^{1-\beta}d\tau}\left(\frac{X_2(s)}{X_1(s)}\right)^{p(\beta-\gamma)}\right]$ and $\mathbb{E}\left[e^{\int_{t}^{s}Ap\gamma\left(\frac{X_2}{X_1}(\tau)\right)^{1-\beta}d\tau}\tilde{u}_2^{p}\left(\frac{X_2(s)}{X_1(s)}\right)\right]$. 
		In particular, one may directly obtain
		\begin{equation}\label{eq:estimate for first term in upper bound of lambda}
			\begin{aligned}
				&\left( \mathbb{E}\left[e^{\int_{t}^{s}Ap(\beta-\gamma)(X_2(\tau))^{1-\beta}(X_1(\tau))^{\beta-1}d\tau}\left(\frac{X_2(s)}{X_1(s)}\right)^{p(\beta-\gamma)}\right]\right)^{\frac{1}{p}}\le e^{\frac{a_{p(\beta-\gamma)}(s-t)}{p}}\left(\frac{N}{K}\right)^{\beta-\gamma},\\
			\end{aligned}
		\end{equation}
		where  $a_{p(\beta-\gamma)}:=C_f{p(\beta-\gamma)}  + \frac{ {\sigma}^2 +\epsilon^2}{2}{p(\beta-\gamma)}({p(\beta-\gamma)}-1) + \frac{a_1{p(\beta-\gamma)}}{c_0}$.
		Furthermore, conditions \eqref{Conditons on tilde u_2}  on $\tilde{u}_2$ imply the following conditions on $\tilde{u}_2^p$:
		\begin{equation*}
			p\gamma\tilde{u}_2^p \le (\tilde{u}_2^p)'\cdot x = p\tilde{u}_2^{p-1}\tilde{u}_2'\cdot x\le p\beta \tilde{u}_2^p,
		\end{equation*}
		  and
		  \begin{equation*}
		\left(p(p-1)\gamma^2+p\tilde{a}_4\right)\tilde{u}_2^p \le(\tilde{u}_2^p)''\cdot x^2=p(p-1)\tilde{u}_2^{p-2}(\tilde{u}_2')^2+p\tilde{u}_2^{p-1}\tilde{u}_2''
		\le \left(p(p-1)\beta^2+p\tilde{a}_5\right)\tilde{u}_2^p.
	\end{equation*}
		Hence,    the calculations similar to that in the proof of Lemma \ref{lem:Estimate for process h(s)} yields
		\begin{equation}\label{eq:estimate for second term in upper bound of lambda}
			\begin{aligned}
				&\left(\mathbb{E}\left[e^{\int_{t}^{s}Ap\gamma(X_2(\tau))^{1-\beta}(X_1(\tau))^{\beta-1}d\tau}\tilde{u}_2^{p}\left(\frac{X_2(s)}{X_1(s)}\right)\right]\right)^{\frac{1}{p}}\le a_3 e^{\frac{a_{p\beta}(s-t)}{p}}\left(\frac{N}{K}\right)^{\gamma},\\
			\end{aligned}
		\end{equation}
		where $a_{p\beta}:=C_f{p\beta}  + \frac{1}{2}({\sigma}^2+\epsilon^2)(p(p-1)\beta^2+p\tilde{a}_5)+\frac{{a_1}{p\beta}}{c_0}$.
		For the last term on the right-hand side of \eqref{eq:split of expectation of G(s)}, we apply  Jensen's inequality and Lemma \ref{lem:Estimate for frac{X_2}{X_1}(s)^{beta -1}} to deduce that 
		\begin{equation}
			\begin{aligned}
				\mathbb{E}\left[\left(\frac{X_2(s)}{X_1(s)}\right)^{q(\gamma-\beta)}\right]&= \mathbb{E}\left[\left(\left(\frac{X_2(s)}{X_1(s)}\right)^{\beta -1}\right)^
				{ \frac{q(\gamma-\beta)}{\beta-1}  }\right] \le \left(\mathbb{E}\left[\left(\frac{X_2(s)}{X_1(s)}\right)^{\beta -1}\right]\right)^{\frac{q(\gamma-\beta)}{\beta-1}}\\
				\qquad &\le \left(e^{b(s-t)}\left(\frac{N}{K}\right)^{\beta-1} + A(1-\beta)\frac{ e^{b(s-t)}-1}{b} \right)^{\frac{q(\gamma-\beta)}{\beta-1}},\\
			\end{aligned}
		\end{equation}
		provided that ${\frac{q(\gamma-\beta)}{\beta-1}}\le1$.
		Since $\beta<1$ and $\gamma>2\beta-1$, we know that  $0<\frac{\gamma-\beta}{\beta-1}<1$. Then we can choose $q: =\frac{\beta-1}{\gamma-\beta}$ and $p:= \frac{2q}{q-1} = 2\frac{1-\beta}{1+\gamma-2\beta}$.		
		Choosing $s=T$, we obtain 
		\begin{equation}\label{eq:estimate for last term in upper bound of lambda}
			\begin{aligned}
				&\quad\left(\mathbb{E}\left[\left(\frac{X_2(T)}{X_1(T)}\right)^{q(\gamma-\beta)}\right]\right)^{\frac{1}{q}}
				\le \left(e^{b(T-t)}\left(\frac{N}{K}\right)^{\beta-1} + A(1-\beta)\frac{ e^{b(T-t)}-1}{b} \right)^{\frac{\gamma-\beta}{\beta-1}}\\
				& \le \left(e^{b(T-t)}\left(\frac{N}{K}\right)^{\beta-1}\right)^{{\frac{\gamma-\beta}{\beta-1}}} + \left(A(1-\beta)\frac{ e^{b(T-t)}-1}{b} \right)^{\frac{\gamma-\beta}{\beta-1}}\\
				& = e^{b(\gamma-\beta)(T-t)/(\beta-1)}\left(\frac{N}{K}\right)^{\gamma-\beta} + \left(A(1-\beta)\frac{e^{b(T-t)}-1}{b} \right)^{\frac{\gamma-\beta}{\beta-1}},\\
			\end{aligned}
		\end{equation}
		where for the second inequality we have used the inequality $(x_1+ x_2)^{\theta} \le x_1^{\theta} + x_2^{\theta}$  for any two positive numbers $x_1$ and $x_2$, provided that $0 \le {\theta}\le 1$.
		Applying the equations  \eqref{eq:split of expectation of G(s)} to \eqref{eq:estimate for last term in upper bound of lambda}   and evaluating at $s=T$, we finally obtain
		\begin{equation}
			\begin{aligned}
				&\lambda(t,K,N) = \mathbb{E}[G(T)]
				\le  a_3 e^{\frac{a_{p(\beta-\gamma)}(T-t)}{p}}\left(\frac{N}{K}\right)^{\beta-\gamma}\cdot e^{\frac{a_{p\beta}(T-t)}{p}}\left(\frac{N}{K}\right)^{\gamma}\\
				&\qquad \qquad \qquad \qquad \qquad  \cdot\left[e^{b(\gamma-\beta)(T-t)/(\beta-1)}\left(\frac{N}{K}\right)^{\gamma-\beta} + \left(A(1-\beta)\frac{-1+ e^{b(T-t)}}{b} \right)^{\frac{\gamma-\beta}{\beta-1}}\right]\\
				&   \qquad \qquad \qquad \qquad \quad \le C_1\left(\frac{N}{K}\right)^{\gamma} + C_2\left(\frac{N}{K}\right)^{\beta}.
			\end{aligned}
		\end{equation}
		This finishes the proof.
	\end{proof}
	
	Combining Lemma \ref{lem:lower bound} and Lemma \ref{lem:upper bound}, we complete the proof of Theorem \ref{thm:crucial estimate theorem}.
	
	\begin{rem}
		The lower bound of the  input  function $\psi$ is used in an essential way to bound  $\mathbb{E}[h(s)]$ from above (in the proof of Lemma \ref{lem:Estimate for process h(s)}), $\mathbb{E}[h_1(s)]$ (in the proof of Lemma \ref{lem:upper bound case1}) and other similar stochastic processes (in the proof of Lemma \ref{lem:upper bound}). It is because the upper bound for the solution $\lambda$ relies on the singular behavior of $j(\psi)$ (in Equation \eqref{eq: differential equation for $lambda$}) in a crucial way.
	\end{rem}
	
	\subsection{Lower and upper bounds for the solution in the case $\beta<\gamma$}\label{Upper bound for the solution when $beta < gamma$}
	In this case, the crucial lower bound is more difficult to obtain compared with the case $\beta\geq \gamma$, as the first term in $dG(s)$ in \eqref{eq:Process of G(s)}, which is nonlinear and most challenging  to handle, will be  negative now. Our main ingredient is to apply a reverse H\"older's inequality to overcome this issue.
	We  give the lower bound and upper bound of $\lambda$ in Lemma \ref{lem: lower bound for lambda for beta<gamma} and Lemma \ref{lem: upper bound for lambda for beta<gamma}, respectively. 
	\begin{lem}[Lower bound]\label{lem: lower bound for lambda for beta<gamma}
		Under the same assumption of Theorem \ref{thm:crucial estimate theorem for case beta<gamma}, we have $\lambda(t,K,N) \ge C_3 \left(\frac{N}{K}\right)^{\beta}\left(  C_4 \left(\frac{N}{K}\right)^{\beta-1} + C_5 \right)^{ \frac{\gamma-\beta}{\beta-1 } }$, where 
		\begin{equation}\label{eq:value for C3, C4, C5}
			\begin{aligned}
				C_3 &:= a_2 e^{\left(-C_{f}\beta + \frac{1}{2}\beta \left(\frac{\beta -\beta^2}{\gamma+1-2\beta} -1\right) \left(\sigma^2+\epsilon^2 \right) \right)T}, \quad
				C_4 :=   e^{(\beta-1)(-C_f +\frac{\beta -2}{2}({{\sigma}^2} + {\epsilon}^2))T}\\
				C_5 &:= A(1-\beta)\frac{ e^{bT}-1}{b} = A(1-\beta) \frac{e^{(\beta-1)(-C_f +\frac{\beta -2}{2}({{\sigma}^2} + {\epsilon}^2))T} -1}{(\beta-1)(-C_f +\frac{\beta -2}{2}({{\sigma}^2} + {\epsilon}^2))}.
			\end{aligned}
		\end{equation}
	\end{lem}
	\begin{proof}
		Noting that $\beta<\gamma$,		we can use the reverse H\"older's inequality and conditions \eqref{Conditons on tilde u_2} on $\tilde{u}_2$ to obtain
		\begin{equation}\label{eq:split of expectation of G(s) beta<gamma}
			\begin{aligned}
				\mathbb{E}[G(s)]&=\mathbb{E}\left[e^{\int_{t}^{s}A\beta(X_2(\tau))^{1-\beta}(X_1(\tau))^{\beta-1}d\tau}\tilde{u}_2\left(\frac{X_2(s)}{X_1(s)}\right)\right]\\
				& = \mathbb{E}\left[e^{\int_{t}^{s}A\beta (X_2(\tau))^{1-\beta}(X_1(\tau))^{\beta-1}d\tau}\left(\frac{X_2(s)}{X_1(s)}\right)^{\beta}\right. \cdot \left. \left(\frac{X_2(s)}{X_1(s)}\right)^{-\beta} \tilde{u}_2\left(\frac{X_2(s)}{X_1(s)}\right) \right]\\
				&\ge a_2 \left( \mathbb{E}\left[e^{\int_{t}^{s}A \frac{\beta}{p} (X_2(\tau))^{1-\beta}(X_1(\tau))^{\beta-1}d\tau} \left(\frac{X_2(s)}{X_1(s)}\right)^{\frac{\beta}{p}}\right]\right)^{p}
				\cdot
				\left(\mathbb{E}\left[\left(\frac{X_2(s)}{X_1(s)}\right)^{\frac{\gamma-\beta}{1-p}} \right]\right)^{-(p-1)},\\
			\end{aligned}
		\end{equation}
		where $p>1$ will be determined later.
		Let 
		\begin{equation*}
			h_2(s) = e^{\int_{t}^{s}A\frac{\beta}{p}(X_2(\tau))^{1-\beta}(X_1(\tau))^{\beta-1}d\tau} \left(\frac{X_2(s)}{X_1(s)}\right)^{\frac{\beta}{p}}.
		\end{equation*}
		Following exactly the same argument as Lemma \ref{lem:Estimate for process h(s)} (but now for ``$\ge$'' direction), we can obtain
		\begin{equation}\label{eq:estimate for h_2}
			\mathbb{E}[h_2(s)] \ge e^{b_5(s-t)}\left(\frac{N}{K}\right)^{\frac{\beta}{p}}\quad \textrm{for $s \in [t, T]$} ,
		\end{equation}
		where $ b_5:= -C_f\frac{\beta}{p} + \frac{1}{2}{\frac{\beta}{p} \left(\frac{\beta}{p}-1 \right)} (\sigma^2+\epsilon^2) <0$.
		Moreover, it follows from Jensen's inequality that 
		\begin{equation}\label{Jensen inequality for lower bound}
			\begin{aligned}
				\mathbb{E}\left[\left(\frac{X_2(s)}{X_1(s)}\right)^{\frac{\gamma-\beta}{1-p}} \right] &= \mathbb{E}\left[\left(\left(\frac{X_2(s)}{X_1(s)}\right)^{\beta-1} \right)^{\frac{\gamma-\beta}{(1-p)(\beta-1) }} \right] 
				\le \left(\mathbb{E}\left[\left(\frac{X_2(s)}{X_1(s)}\right)^{\beta-1} \right]  \right)^{ \frac{\gamma-\beta}{(1-p)(\beta-1) } },
			\end{aligned}
		\end{equation}
		as long as $\frac{\gamma-\beta}{(1-p)(\beta-1) } \le 1 $ or equivalently 
		$p \ge 	\frac{\gamma +1 -2\beta}{1-\beta}>1 $.
		Combining Equations \eqref{eq:split of expectation of G(s) beta<gamma}, \eqref{eq:estimate for h_2} and \eqref{Jensen inequality for lower bound} and using Lemma \ref{lem:Estimate for frac{X_2}{X_1}(s)^{beta -1}} leads to
		\begin{equation}
			\begin{aligned}
				\mathbb{E}[G(s)] \ge a_2 e^{b_5 p(s-t)} \left(\frac{N}{K}\right)^{\beta} \left(  e^{b(s-t)}\left(\frac{N}{K}\right)^{\beta-1} + A(1-\beta)\frac{ e^{b(s-t)}-1}{b} \right)^{ \frac{\gamma-\beta}{\beta-1 } }.
			\end{aligned}
		\end{equation}
		Since $b_5p = -C_{f}\beta + \frac{1}{2}\beta \left(\frac{\beta}{p} -1 \right) \left(\sigma^2+\epsilon^2 \right)<0 $, it is clear that we can maximize the lower bound of $\mathbb{E}[G(s)] $ above by choosing $p =\frac{\gamma +1 -2\beta}{1-\beta}$.
		Now, let $p =\frac{\gamma +1 -2\beta}{1-\beta}$ and  $s= T$, using $b>0$ and $b_5 < 0$, we obtain 
		\begin{equation*}
			\begin{aligned}
				&\lambda(t, K, N) = \mathbb{E}[G(T)]\\
				\ge& a_2 e^{b_5 p(T-t)} \left(\frac{N}{K}\right)^{\beta} \left(  e^{b(T-t)}\left(\frac{N}{K}\right)^{\beta-1} + A(1-\beta)\frac{ e^{b(T-t)}-1}{b} \right)^{ \frac{\gamma-\beta}{\beta-1 } }\\
				\ge& a_2 e^{b_5pT}  \left(\frac{N}{K}\right)^{\beta}\left(  e^{bT}\left(\frac{N}{K}\right)^{\beta-1} + A(1-\beta)\frac{e^{bT}-1}{b} \right)^{ \frac{\gamma-\beta}{\beta-1 } }
				= C_3  \left(\frac{N}{K}\right)^{\beta}\left(  C_4 \left(\frac{N}{K}\right)^{\beta-1} + C_5\right)^{ \frac{\gamma-\beta}{\beta-1 } }.
			\end{aligned}
		\end{equation*}
	\end{proof}
	
	\begin{rem}
		Similar to Remark \ref{rem: remark for lower bound beta>gamma}, in the proof of Lemma  \ref{lem: lower bound for lambda for beta<gamma} we only require that $\beta < \gamma$ and $\psi>0$. In this sense, the lower bound of $\lambda$ mainly comes from the terminal data and we can weaken the assumption on $\psi$ when considering the lower bound of $
		\lambda$.	\end{rem}
	\begin{lem}[Upper bound]\label{lem: upper bound for lambda for beta<gamma}
		Under the  assumptions of Theorem \ref{thm:crucial estimate theorem for case beta<gamma} and additional assumption that $ \inf_{x>0} \frac{\tilde{u}_2' }{\tilde{u}_2} (x) >\beta$, we have $\lambda(t,K,N) \le C_6\left(\frac{N}{K}\right)^{\gamma} $,
		where 
		\begin{equation}\label{eq:value for C6}
			C_6 = a_3 e^{\left( b_6 +  C_f \gamma +  \frac{\tilde{a}_5}{2}({\sigma}^2+\epsilon^2)\right)T},
		\end{equation}
		and $b_6=b_6 \left(C_f, a_1, a_2, \sigma, \epsilon, \gamma, \beta, A, T, \inf_{x>0} \frac{\tilde{u}_2' (x)}{\tilde{u}_2(x)} \right) $ is implicitly determined in \eqref{eq: estimate for term j(psi)}.
	\end{lem}
	
	\begin{proof}
		We start from Equation \eqref{eq:Process of G(s)} and rewrite it as
		\begin{equation}\label{eq:rewrite for dG(s)}
			\begin{aligned}
				dG(s) &=
				G(s) \cdot \left( A \left(\beta- \frac{\tilde{u}_2' }{\tilde{u}_2} \right)\left(\frac{X_2(s)}{X_1(s)}\right)^{1-\beta}  + \frac{\tilde{u}_2' }{\tilde{u}_2} \cdot \left(\frac{X_2}{X_1}\right)^{2}j(\psi) \right) ds\\
				& \quad +e^{\int_{t}^{s}A\beta(X_2(\tau))^{1-\beta}(X_1(\tau))^{\beta-1}d\tau}\tilde{u}_2'\left(\frac{X_2(s)}{X_1(s)}\right) \left[\frac{1}{X_1}f(X_2)ds  \right] \\
				&\quad +   e^{\int_{t}^{s}A\beta(X_2(\tau))^{1-\beta}(X_1(\tau))^{\beta-1}d\tau}\left[ \frac{1}{2}\tilde{u}_2''\left(\frac{X_2(s)}{X_1(s)}\right)\left(\frac{X_2}{X_1}\right)^{2} ({{\sigma}^2}+ {\epsilon}^2)ds \right]\\ 
				&\quad  + e^{\int_{t}^{s}A\beta(X_2(\tau))^{1-\beta}(X_1(\tau))^{\beta-1}d\tau}\tilde{u}_2'\left(\frac{X_2(s)}{X_1(s)}\right)\left(\frac{X_2}{X_1}\right)\left({{\sigma}}dB_2 - \epsilon dB_1 \right).\\
			\end{aligned}
		\end{equation}
		
		We observe that using the lower bound on $\psi$ and the definition of $C_3, C_4$ and $C_5$ in \eqref{eq:value for C3, C4, C5}, we have 
		\begin{equation}\label{eq: estimate for term j(psi)}
			\begin{aligned}
				& \quad \  \frac{\tilde{u}_2' }{\tilde{u}_2}\cdot \left(\frac{X_2}{X_1}\right)^{2}j(\psi) \le a_1\gamma  \frac{X_2}{X_1} \psi^{-\frac{1}{\gamma}}\\
				& \le   a_1\gamma \frac{X_2}{X_1}\left(a_2 e^{b_5pT}  \left(\frac{X_1}{X_2}\right)^{\beta}\left(  e^{bT}\left(\frac{X_1}{X_2}\right)^{\beta-1} + A(1-\beta)\frac{ e^{bT}-1}{b} \right)^{ \frac{\gamma-\beta}{\beta-1 } } \right)^{-\frac{1}{\gamma}}
				\\
			 &  	\le a_1\gamma \frac{X_2}{X_1} a_2^{-\frac{1}{\gamma}}e^{-\frac{1}{\gamma}b_5pT} \left(\frac{X_1}{X_2}\right)^{-\frac{\beta}{\gamma}} \left( e^{ \frac{\gamma-\beta}{\gamma(1-\beta)} bT}   \left(\frac{X_1}{X_2}\right)^{\frac{\beta}{\gamma} -1}  +  \left(A(1-\beta)\frac{ e^{bT}-1}{b} \right)^{ \frac{\gamma-\beta}{\gamma(1-\beta) } } \right) \\
				 & =  a_1\gamma a_2^{-\frac{1}{\gamma}}e^{-\frac{1}{\gamma}b_5pT + \frac{\gamma-\beta}{\gamma(1-\beta)} bT } + 
				a_1\gamma a_2^{-\frac{1}{\gamma}} e^{-\frac{1}{\gamma}b_5pT}
				\left(A(1-\beta)\frac{ e^{bT}-1}{b} \right)^{ \frac{\gamma-\beta}{\gamma(1-\beta) } } \left(\frac{X_1}{X_2}\right)^{1- \frac{\beta}{\gamma} }  \\
			 &  	\le b_6 \left(C_f, a_1, a_2, \sigma, \epsilon, \gamma, \beta, A, T, \inf_{x>0} \frac{\tilde{u}_2' (x)}{\tilde{u}_2(x)} \right) + A \left( \frac{\tilde{u}_2' }{\tilde{u}_2} -\beta \right) \left(\frac{X_2}{X_1}\right)^{1-\beta},
			\end{aligned}
		\end{equation}
		where $p =\frac{\gamma +1 -2\beta}{1-\beta}$, and  the third inequality above follows by  using
		\begin{equation*}
			\begin{aligned}
				&\left(  e^{bT}\left(\frac{X_1}{X_2}\right)^{\beta-1} + A(1-\beta)\frac{ e^{bT}-1}{b} \right)^{ \frac{\gamma-\beta}{\gamma(1-\beta) } } 
				\le e^{-\frac{bT}{\gamma}} \left(\frac{X_1}{X_2}\right)^{ \frac{\beta}{\gamma}-1 }  +  \left(A(1-\beta)\frac{ e^{bT}-1}{b} \right)^{ \frac{\gamma-\beta}{\gamma(1-\beta) } },
			\end{aligned}
		\end{equation*}
		due to $\frac{\gamma-\beta}{\gamma(1-\beta) } \le 1 $ and inequality $(x_1+ x_2)^{\theta} \le x_1^{\theta} + x_2^{\theta}$  for any two positive numbers $x_1$ and $x_2$, provided that $0\le {\theta} \le 1$. We also use the fact  $0 < 1-\frac{\beta}{\gamma} < 1-\beta$, and hence, there exists a  positive constant $b_6$ such that the last inequality in Equation \eqref{eq: estimate for term j(psi)} holds.
		It then follows from Equations  \eqref{eq:rewrite for dG(s)}, \eqref{eq: estimate for term j(psi)} and  conditions  \eqref{Conditons on tilde u_2} on $\tilde{u}_2$ that
		\begin{equation*}
			dG(s)\le \tilde{a}G(s)ds + e^{\int_{t}^{s}A\beta(X_2(\tau))^{1-\beta}(X_1(\tau))^{\beta-1}d\tau}  \tilde{u}_2'\left(\frac{X_2(s)}{X_1(s)}\right)\left(\frac{X_2}{X_1}\right) ({{\sigma}}dB_2 - \epsilon dB_1 ),
		\end{equation*}
		where the constant $\tilde{a}= b_6 +  C_f \gamma +  \frac{\tilde{a}_5}{2}({\sigma}^2+\epsilon^2) $.
		Using     $\mathbb{E}[G(t)]  \le {a}_3 \left(\frac{N}{K} \right) ^{\gamma}$,
		we deduce that 
		\begin{equation*}
			\begin{aligned}
				&\mathbb{E}[G(s)]-a_3\left(\frac{N}{K}\right)^{\gamma} 
				\le \mathbb{E}[G(s)-G(t)]\\
				\le & \mathbb{E}\left[\int_t^s \tilde{a}G(\tau)d\tau + \int_{t}^{s}e^{\int_{t}^{\tau}A\beta\left(\frac{X_2}{X_1}(\tau)\right)^{1-\beta}d\tau}  \tilde{u}_2'\left(\frac{X_2(\tau)}{X_1(\tau)}\right)\left(\frac{X_2}{X_1}\right) ({{\sigma}}dB_2 - \epsilon dB_1 )\right]
				=\mathbb{E}\left[\int_t^s \tilde{a}G(\tau)d\tau \right].
			\end{aligned}
		\end{equation*}
		Using the Gr\"onwall  inequality, we find that
		$
		\mathbb{E}[G(s)] \le a_3e^{\tilde{a}(s-t)}\left(\frac{N}{K}\right)^{\gamma}.
		$
		In particular, 
		\begin{align*}
			&\lambda(t, K, N) = \mathbb{E}[G(T)]\le a_3e^{\tilde{a}(T-t)}\left(\frac{N}{K}\right)^{\gamma} \le  	C_6\left(\frac{N}{K}\right)^{\gamma}.
		\end{align*}	
	\end{proof}
	
	\section{Energy estimates for the solution to  linearized problem \eqref{eq: differential equation for $lambda$} in weighted Sobolev spaces}\label{sec:energy estimate}
	
	In  this and the next section, we shall show that  the solution map $\Gamma:\psi \mapsto \lambda$ generated by solving the linearized problem  \eqref{eq: differential equation for $lambda$} has a fixed point in  the set $S$ defined in \eqref{def:definition for the growth rate set} based on the pointwise estimates in Section \ref{section: pointwise estimate}.  
	This section will focus on the energy estimates for the linearized problem in a suitably chosen weighted Sobolev spaces, which will be applied in the next section to extend the existence result in Section \ref{sec:paraboliceq} to the case with a merely measurable input $\psi$.
	Moreover, we shall utilize a compactness result based on the energy estimates to obtain a fixed point and then show the uniqueness and regularity of the solution to the nonlinear problem \eqref{eq: differential equation for original $lambda$} in the next section. In this and the next section, we use $C$ to denote some universal constant that may differ in different lines.
	
	Inspired by the upper bounds we obtained for the linearized problem \eqref{eq: differential equation for $lambda$}, we choose the weight 
	\begin{equation}\label{eq: the choice of the weight}
		\varphi =\frac{K^{6}}{K^{10} +1}\frac{N^{2}}{N^{10} +1},
	\end{equation}
	so  that all the following energy estimates remain finite.
	First of all, let us define another function space over the spatial domain for later use,
	\begin{equation}\label{functional space {H}^2_varphi}
		Y_2:=\left\{{\psi \in Y_1  \left| \left|\left|\frac{\partial^2 \psi}{\partial K^2}\right|\right|^2_{L^2_{2,K,\varphi}},\right.  ~ \left|\left|\frac{\partial^2 \psi}{\partial K^2}\right|\right|^2_{L^2_{2,N,\varphi}}, ~ \left|\left|\frac{\partial^2 \psi}{\partial K\partial N}\right|\right|^2_{L^2_{1,K,1,N,\varphi}}<\infty}\right\},
	\end{equation}
	equipped with the norm
	\begin{equation*}
		\begin{aligned}
			||\psi||_{Y_2}^2:=||\psi||^2_{Y_1}
			+ \left|\left|\frac{\partial^2 \psi}{\partial K^2}\right|\right|^2_{L^2_{2,K,\varphi}} +\left|\left|\frac{\partial^2 \psi}{\partial K^2}\right|\right|^2_{L^2_{2,N,\varphi}}+\left|\left|\frac{\partial^2 \psi}{\partial K\partial N}\right|\right|^2_{L^2_{1,K,1,N,\varphi}}.
		\end{aligned}
	\end{equation*}
	We also define its spatial-time variants
	{\color{black}
		\begin{equation}\label{functional space mathcal{H}^2_varphi}
			\mathcal{H}^2_\varphi :=  \left\{ \psi\in L^2([0,T];Y_2)\left| \frac{\partial \psi}{\partial t}\in L^2([0,T];Y)\right.\right\}.
		\end{equation}
	}
	\begin{rem}
		It follows from the  Sobolev embedding theorem that if $\psi(t)$ belongs to $Y_2$ for some $t$, then $\psi(t)$ $\in$ $\cap_{\delta\in(0,1)}C^{\delta}_{loc}(\mathbb{R}_+^2)$, where $C^{\delta}_{loc}(\mathbb{R}_+^2)$ is the local H\"older space with the exponent $\delta$.
	\end{rem}
	The  main result in this section is:
	\begin{thm}\label{thm: A priori estimates}
		Let $\psi$ be a smooth input function satisfying the pointwise bound in Theorem \ref{thm:existence when psi is smooth},
		 then there exists a classical solution $\lambda$ to Equation \eqref{eq: differential equation for $lambda$} such that  this $\lambda$ belongs to $\mathcal{H}^1_\varphi $, and it is also a weak solution to Equation \eqref{eq: differential equation for $lambda$}
		in the sense of Definition \ref{def: definiton of the weak solution of the linearized problem}. Furthermore, $\lambda \in \mathcal{H}^2_\varphi $, precisely, we have 
		\begin{equation}\label{eq: energy estimates}
			\sup_{0\le t\le T} ||\lambda(t)||_{Y_1}^2  + ||\lambda||_{L^2([0,T];Y_2)} + ||\lambda_t||_{L^2([0,T];Y)}  \le C\left|\left|u'_2\left(\frac{K}{N}\right)\right|\right|_{Y_1}^2 + \bar{C},
		\end{equation}
		for some constant $C$ and $\bar{C}$ that depend on $a_1$, $a_2$, $a_3$, $ {a}_4$, $a_5$ $A, C_f, \beta, \gamma, \epsilon, {\sigma}$ and $T$. Moreover,  $\lambda$ satisfies the pointwise estimates in Theorems
		\ref{thm:crucial estimate theorem} and \ref{thm:crucial estimate theorem for case beta<gamma}, that is $\lambda \in S$ for $S$ defined in  \eqref{def:definition for the growth rate set}.  
	\end{thm}
	\begin{proof}
		When $\psi \in  {C^{\infty}([0,T]\times\mathbb{R}^2_+)}$ and satisfies the pointwise bound in Theorem \ref{thm:existence when psi is smooth}, we already have a classical solution to 
		\eqref{eq: differential equation for $lambda$} by Theorem \ref{thm:existence when psi is smooth}, and   due to Theorems
		\ref{thm:crucial estimate theorem} and \ref{thm:crucial estimate theorem for case beta<gamma}, we have $\lambda \in S$. If we can  establish 
		\eqref{eq: energy estimates}, then it follows that $\lambda$ is a weak solution in the sense of Definition \ref{def: definiton of the weak solution of the linearized problem} by a density argument, as mentioned in Section \ref{sec:Solving the HJB}, the estimates \eqref{eq: energy estimates} will be shown in the next two sections.
	\end{proof}
	\begin{rem}\label{rem: only need lower bound of psi}
		 We only require  $\psi$ in Theorem \ref{thm: A priori estimates} to satisfy the lower bound condition appeared in the definition of the set $S$, as   the singular behavior of $j(\psi)$, and hence, the upper bounds of the solution $\lambda$ in Theorems 
		\ref{thm:crucial estimate theorem} and 
		\ref{thm:crucial estimate theorem for case beta<gamma} are only determined by the lower bounds of $\psi$.
	\end{rem}
	We shall give the weighted energy estimates in $\mathcal{H}^1_\varphi$ and their proofs; for precise statements, see Theorems \ref{thm:L^2([0,T];Y_1) estimate} and \ref{thm:Y_0 estimate}.   The  weighted energy estimates in $	\mathcal{H}^2_\varphi$ are similar and will be omitted. Let us begin with some preliminary results first. 
	
	\begin{lem}[${ L^2_\varphi}$ and ${ L^2_{\tilde{\varphi}}}$-estimates for $\lambda$]\label{lem: { L^2_varphi} estimate} 
		Under the assumption of Theorem \ref{thm: A priori estimates}, there exists a constant $C_7$ depending on $C_1,C_2, C_6, \beta$ and $\gamma$,  such that 
		$\sup_{0\le t\le T} ||\lambda(t)||_{L^2_{\varphi}}$, $\sup_{0\le t\le T} ||\lambda(t)||_{L^2_{\tilde{\varphi}}}$ $\le$ $C_7$ and $\int_{0}^{T}\int_{\mathbb{R}_+^2 }{\lambda}^2\varphi dKdNdt$,  $\int_{0}^{T}\int_{\mathbb{R}_+^2 }{\lambda}^2\tilde{\varphi} dKdNdt\le C_7^2T$.
	\end{lem}
	\begin{proof}
		For $\beta\ge \gamma$,	it follows from Theorem \ref{thm:crucial estimate theorem}, $\varphi =\frac{K^{6}}{K^{10} +1}\frac{N^{2}}{N^{10} +1}$ and ${\eqref{eq:def of tildevarphi}}_1$, i.e.,  $\tilde{\varphi} = (N^{4-4\beta}K^{4\beta-4}  + N^{2+2\beta}K^{-2-2\beta})\varphi$, we have 
		\begin{equation*}
			\begin{aligned}
				{\lambda}^2 \varphi &\le 2\left(C_1^2\frac{N^{2\gamma}}{K^{2\gamma}}+ C_2^2\frac{N^{2\beta}}{K^{2\beta}}\right)\cdot\frac{K^{6}}{K^{10} +1}\frac{N^{2}}{N^{10} +1} \\
				& = 2C_1^2\frac{K^{6-2\gamma}}{K^{10}+1}\frac{N^{2+2\gamma}}{N^{10}+1} + 2C_2^2\frac{K^{6-2\beta}}{K^{10}+1}\frac{N^{2+2\beta}}{N^{10}+1}.
			\end{aligned}
		\end{equation*}
		\begin{equation*}
			\begin{aligned}
				{\lambda}^2 \tilde{\varphi} &\le 2\left(C_1^2\frac{N^{2\gamma}}{K^{2\gamma}}+ C_2^2\frac{N^{2\beta}}{K^{2\beta}}\right)\cdot\left(\frac{N^{4-4\beta}}{K^{4-4\beta}} + \frac{N^{2 + 2\beta}}{K^{2 + 2\beta}} \right) 
				\cdot\frac{K^{6}}{K^{10}+1}\frac{N^{2}}{N^{10}+1} \\
				& = 2C_1^2\frac{K^{2+4\beta-2\gamma}}{K^{10}+1}\frac{N^{6-4\beta+2\gamma}}{N^{10}+1} + 2C_1^2\frac{K^{4-2\gamma - 2\beta}}{K^{10}+1}\frac{N^{4+2\beta+2\gamma}}{N^{10}+1}\\
				&\quad+ 2C_2^2\frac{K^{2+2\beta}}{K^{10}+1}\frac{N^{6-2\beta}}{N^{10}+1}+ 2C_2^2\frac{K^{4 - 4\beta}}{K^{10}+1}\frac{N^{4+4\beta}}{N^{10}+1}.
			\end{aligned}
		\end{equation*}
		For $\beta <\gamma$, 	 we see from Theorem \ref{thm:crucial estimate theorem for case beta<gamma}, $\varphi =\frac{K^{6}}{K^{10} +1}\frac{N^{2}}{N^{10} +1}$ and ${\eqref{eq:def of tildevarphi}}_2$, i.e.,
		$\tilde{\varphi} = (N^{4-4\beta}K^{4\beta-4}  + N^{2+2\gamma}K^{-2-2\gamma})\varphi$ that 
		\begin{equation*}
			\begin{aligned}
				{\lambda}^2 \varphi \le 2C_6^2\frac{K^{6-2\gamma}}{K^{10}+1}\frac{N^{2+2\gamma}}{N^{10}+1}, ~
				{\lambda}^2 \tilde{\varphi} \le C_6^2\frac{K^{2+4\beta-2\gamma}}{K^{10}+1}\frac{N^{6-4\beta+2\gamma}}{N^{10}+1} + C_6^2\frac{K^{4 - 4\gamma}}{K^{10}+1}\frac{N^{4+ 4\gamma}}{N^{10}+1}.
			\end{aligned}
		\end{equation*}
		We can easily verify that all the terms on the right-hand side are integrable on $\mathbb{R}^2_+$, then 
		there exists a constant $C_7=C_7(C_1, C_2, C_6, \beta, \gamma)$ such that the statement holds.
	\end{proof}

	\begin{thm}[$C( 0,T; Y)$ and $L^2(0,T;Y_1)$ estimates for $\lambda$]\label{thm:L^2([0,T];Y_1) estimate}
		Under the same
		assumption of Theorem \ref{thm: A priori estimates},  there exists a constant $\bar{C}$ depending on
		$a_1, a_2$, $a_3, a_4$, $a_5, A$, $C_f,\beta$,$\gamma, \epsilon,{\sigma}$, and $T$, such that 
		\begin{equation}\label{eq:L^2([0,T];Y_1) estimate} 
			\begin{aligned}
				\sup_{0\le t\le T} ||\lambda(t)||_{Y}^2 + \int_{0}^{T}||\lambda||_{Y_1}^2dt \le & \frac{4}{\min{(\epsilon^2,{\sigma}^2)}}\left|\left|u'_2\left(\frac{K}{N}\right)\right|\right|_{L^2_{\varphi}}^2  + \bar{C}(a_1, a_2, a_3, a_4, a_5, A, C_f,\beta,\gamma, \epsilon,{\sigma}, T).
			\end{aligned}
		\end{equation}
	\end{thm}
	For the energy estimates here, we formally  multiply \eqref{eq: differential equation for original $lambda$} by $\lambda \varphi$, then integrate on $\mathbb{R}_+^2$. All the terms can be handled in a classical manner, except for $\int_{\mathbb{R}_+^2 }N\frac{\partial \lambda}{\partial K}j(\psi) \lambda \varphi dKdN$, $\int_{\mathbb{R}_+^2 } A N^{1-\beta}K^{\beta}\frac{\partial \lambda}{\partial K} \lambda \varphi  dKdN$ and $\int_{\mathbb{R}_+^2 } A \beta N^{1-\beta}K^{\beta-1}\lambda^2 \varphi dKdN$ due to the unbounded coefficients. However, they can be controlled with the help of pointwise estimates in Theorems \ref{thm:crucial estimate theorem} and  \ref{thm:crucial estimate theorem for case beta<gamma}  and the decaying weights $\varphi$ and $\tilde{\varphi}$. Taking $\int_{\mathbb{R}_+^2 } A N^{1-\beta}K^{\beta}\frac{\partial \lambda}{\partial K} \lambda \varphi  dKdN$ as an example, we can estimate it as follows:
	\begin{equation*}
		\begin{aligned}
			&\quad	\left|\int_{\mathbb{R}_+^2 } A N^{1-\beta}K^{\beta}\frac{\partial \lambda}{\partial K}\lambda  \varphi dKdN\right| 
			\le A{\tilde{\alpha}}  \left|\left|\frac{\partial \lambda}{\partial K}(t) \right|\right|^2_{L^2_{1,K,\varphi}} + \frac{A}{8{\tilde{\alpha}}} ||\lambda||^2_{L^2_{\tilde{\varphi}}} + \frac{AC(\beta)}{8{\tilde{\alpha}}} ||\lambda ||^2_{L^2_{\varphi}},
		\end{aligned}
	\end{equation*}
	where the second term on the right-hand side is finite due to the pointwise estimates and the choice of our weights $\varphi$ and $\tilde{\varphi}$. We put the complete proof in Appendix \ref{appen:Solvability of X1 and Probabilistic representation of lambda}.

	%
	\begin{thm} [$Y_0$-estimate for $\frac{\partial\lambda}{\partial t}$]\label{thm:Y_0 estimate}
		The weak derivative $\frac{\partial\lambda}{\partial t}$ is  in $L^2( [0,T];Y_0 )$, hence, $\lambda$ is in $\mathcal{H}^1_\varphi$.
	\end{thm}
	\begin{proof}
		We note that for the bilinear form $B[\lambda(t),\phi;t]$ defined in \eqref{eq:bilinear form} we have  $|B[\lambda(t),\phi;t]| \leq C||\lambda(t)||_{Y_1}||\phi||_{Y_1}$. Indeed, using Lemma \ref{lem: { L^2_varphi} estimate}, Theorem \ref{thm:L^2([0,T];Y_1) estimate}, and the Cauchy-Schwarz inequality, we have
		\begin{equation*}
			\begin{aligned}
				&\quad \left|\int_{\mathbb{R}_+^2 } N^{1-\beta} K^{\beta} \frac{\partial \lambda}{\partial K}\phi\varphi dKdN\right| \\
				&\le \left(\int_{\mathbb{R}_+^2 } \left(\frac{\partial \lambda}{\partial K}\right)^2 K^2\varphi dKdN\right)^{\frac{1}{2}}\left(\int_{\mathbb{R}_+^2 } N^{2-2\beta}K^{2\beta-2}\phi^2 \varphi dKdN\right)^{\frac{1}{2}}
		\leq C ||\lambda||_{Y_1}||\phi||_{Y}.
			\end{aligned}
		\end{equation*}
		\begin{equation*}
			\begin{aligned}
				\left|\int_{\mathbb{R}_+^2 } N^{1-\beta} K^{\beta-1} \lambda \phi\varphi dKdN\right| &\le \left(\int_{\mathbb{R}_+^2 }  \lambda^2 \varphi dKdN\right)^{\frac{1}{2}}\left(\int_{\mathbb{R}_+^2 } N^{2-2\beta}K^{2\beta-2}\phi^2 \varphi dKdN\right)^{\frac{1}{2}}
					\leq C ||\lambda||_{Y}||\phi||_{Y}.
			\end{aligned}
		\end{equation*}
		For $\beta \ge \gamma$, using the fact that  $j(\psi)\frac{N}{K} \le \frac{a_1}{c_0}$, we have  
		\begin{equation*}
			\begin{aligned}
				& \left|\int_{\mathbb{R}_+^2} N j(\psi)\frac{\partial \lambda}{\partial K} \phi   \varphi dKdN\right| 
				\le \frac{a_1}{c_0}\left(  \int_{\mathbb{R}_+^2}\left(\frac{\partial \lambda}{\partial K}\right)^2K^2 \varphi dK\right)^{\frac{1}{2}} \left(\int_{\mathbb{R}_+^2}\phi^2  \varphi dKdN\right)^{\frac{1}{2}}
					\leq C ||\lambda||_{Y_1}||\phi||_{Y}.
			\end{aligned}
		\end{equation*}
		While for $\beta<\gamma$, using Equation \eqref{eq:estimate for j(psi)}, we obtain 
		\begin{equation*}
			\begin{aligned}
				&\quad\left| \int_{\mathbb{R}_+^2 } N j(\psi) \frac{\partial \lambda}{\partial K} \phi \varphi dKdN  \right| 
				\le \left| \int_{\mathbb{R}_+^2 } \left(b_7+  \left(\frac{N}{K}\right)^{1-\beta}\right)K\frac{\partial \lambda}{\partial K} \phi \varphi dKdN  \right| \\ 
				& \le  b_7\left( \int_{\mathbb{R}_+^2}\left(\frac{\partial \lambda}{\partial K}\right)^2K^2 \varphi dKdN \right)^{\frac{1}{2}} \left( \int_{\mathbb{R}_+^2}\phi^2  \varphi dKdN \right)^{\frac{1}{2}}\\
				& \qquad + \left(  \int_{\mathbb{R}_+^2}\left(\frac{\partial \lambda}{\partial K}\right)^2K^2 \varphi dKdN \right)^{\frac{1}{2}} \left(\int_{\mathbb{R}_+^2}\phi^2  \left(\frac{N}{K}\right)^{2-2\beta} \varphi dKdN\right)^{\frac{1}{2}}
					\leq C ||\lambda||_{Y_1}||\phi||_{Y}.
			\end{aligned}
		\end{equation*}
		All the other terms are rather standard and can be estimated directly. Hence, we can obtain from \eqref{eq:another writing for weak solution} that 
		$$
		\left|\langle \lambda_t, \phi \rangle \right|=\left| B[\lambda(t),\phi;t] \right|\le C ||\lambda(t)||_{Y_1}||\phi||_{Y_1}$$
		for some  $C$ depending on the given parameters. 
		Then, for a.e. $t$,
		\begin{equation*}
			||\lambda_t(t)||_{Y_0} = \sup_{||\phi||_{Y_1}= 1}\left|\langle \lambda_t, \phi \rangle \right|=\sup_{||\phi||_{Y_1} = 1}\left| B[\lambda(t),\phi;t] \right|\le \sup_{||\phi||_{Y_1} = 1}C ||\lambda||_{Y_1}||\phi||_{Y_1} \le C||\lambda||_{Y_1}.
		\end{equation*}
		Finally, we conclude that
		$
		||\lambda_t||_{L^2([0,T];Y_0)}^2= \int_{0}^{T}||\lambda_t(t)||_{Y_0}^2dt \le C\int_{0}^{T}||\lambda||_{Y_1}^2dt<\infty.
		$
	\end{proof}

	\begin{cor}[Uniqueness of weak solution]\label{cor:Uniqueness of weak solution}
		Let $\psi \in  S$, then the weak solution  $\lambda \in \mathcal{H}^1_\varphi \cap S$ 
		to Equation \eqref{eq: differential equation for $lambda$} in the sense of Definition \ref{def: definiton of the weak solution of the linearized problem} is unique. 
	\end{cor}
	Note carefully that, we do not assume $\psi$ to be smooth here, and we have not shown the existence of a weak solution in this case (this will be done in Section \ref{sec:Weak solution and fixed point for nonlinear equation}, cf. Lemma \ref{lem:extension lemma}). What we try to prove here is that   if a weak solution exists   in $\mathcal{H}^1_\varphi  \cap S$,  it is unique. Moreover, similar to Remark \ref{rem: only need lower bound of psi}, we can weaken the assumption $\psi\in S$ to   that $\psi$ only satisfies the lower bounds under respective cases in the definition of $S$ in \eqref{def:definition for the growth rate set}.
	
	\begin{proof}
		The key to the proof is also the maximum principle, but this time, for the weak solution.
		We just need to prove the weak solution $\lambda$ of $\eqref{eq:weak solution of equation of lambda}_1$ has to be 0 if the terminal data $u'_2$ in $\eqref{eq:weak solution of equation of lambda}_2$ is replaced by 0.  First, we note that $w = \frac{\tilde{\lambda}}{h}$ would be a weak solution to Equation \eqref{eq: the equation of w} (with 0 terminal data) after change-of-variables $K=e^x$, $N=e^y$ and setting $\tilde{\lambda}(t,x,y) = \lambda(t, e^x, e^y)$ if $\lambda$ is a weak solution to Equation \eqref{eq: differential equation for $lambda$}. Since $\lambda \in S$, after choosing the same weight $\rho$  as in the proof of Theorem \ref{thm:Maximum principle}, we know that $\lim_{R \to \infty} \sup_{|x| + |y| \geq R} |w|$$ \to 0$ {uniformly in $t \in [0,T]$}.  We use the maximum principle for the weak solution to conclude that $\sup_{[0,T]\times B_{R}} |w| \le	 \sup_{[0,T]\times \partial B_{R} \cup \{T\}\times B_R} |w|$ (the boundary value of $w$ is interpreted in the trace sense), for a reference, one can see   Theorem 6.25 and Corollary 6.26 in \cite{MR1465184}. As $w(T, x,y) =0$, we can choose $R$ arbitrarily large to see that $|w|$ can be arbitrarily small on the parabolic boundary $[0,T]\times \partial B_{R} \cup \{T\}\times B_R$. Hence, we must have $|w|=0$, this also shows that $\tilde{\lambda}$ and $\lambda$ are 0. 
	\end{proof}
	The  weighted energy estimates in $	\mathcal{H}^2_\varphi$, i.e, Theorem \ref{thm: A priori estimates} can be derived similarly; we leave them to  interested readers, one can also refer to \cite[Section 5.2]{HKUHUB_Hao}.

\section{Solution  for the nonlinear problem and the original HJB equation}\label{sec:Weak solution and fixed point for nonlinear equation}	
Up to now, we have shown that if   $\psi$ is smooth,  $\psi(K,N,t) \ge c_0^{\gamma}\left(\frac{N}{K}\right)^{\gamma}$ when $\beta \ge \gamma$,  and  $\psi(t,K,N) \ge C_3  \left(\frac{N}{K}\right)^{\beta}\left(  C_4 \left(\frac{N}{K}\right)^{\beta-1} + C_5\right)^{ \frac{\gamma-\beta}{\beta-1 } }$ when $\beta < \gamma$, then the linearized problem \eqref{eq: differential equation for $lambda$} has a unique weak solution  in $\mathcal{H}^2_\varphi \cap S$ (cf. Theorem \ref{thm: A priori estimates}) which is also classical (cf. Theorem \ref{thm:existence when psi is smooth}). 
In Subsection \ref{sec:The extension lemma and fixed point argument}, we shall
show that the weak solution still exists 
in $\mathcal{H}^2_\varphi \cap S$ even when the input function $\psi$ is just measurable and satisfies the prescribed lower bound for a.e. $t$, $K$, and $N$ by an approximation argument based on the energy estimates  in Theorem \ref{thm: A priori estimates} (cf. Lemma \ref{lem:extension lemma}).
Next, we consider the solution map 
$\Gamma$ defined on $S$, which is a self-map now, 
and use Schauder’s fixed point theorem to find a solution to the nonlinear Equation \eqref{eq: differential equation for original $lambda$}. The key to applying Schauder’s fixed point theorem is the compactness of $\Gamma(S)\subset \mathcal{H}^2_{\varphi}\cap S$.   It will be done in Subsection \ref{sec: Compact embedding}  (cf. Lemma \ref{lem:compactness of Hvarphi2}).
Subsection \ref{Existence  and uniqueness to the original HJB equation} is devoted to the regularity and uniqueness of weak solutions to the nonlinear problem \eqref{eq: differential equation for original $lambda$} obtained in Subsection \ref{sec:The extension lemma and fixed point argument} and to the original HJB equation \eqref{eq:second HJB for v(K,N)}.
\subsection{The extension lemma for the linearized equation and the fixed point}\label{sec:The extension lemma and fixed point argument}
As already explained above, to show the existence of a fixed point, the key is to prove and apply a certain compactness of  $\Gamma$.
At the first glance, one may try to show the compactness of  $\mathcal{H}^2_\varphi \cap S$ in $L^2([0,T]; Y)$ or $L^2([0,T]; Y_1)$, where the function space $\mathcal{H}^2_\varphi$ is defined in \eqref{functional space mathcal{H}^2_varphi}. However,  $Y_2$ is not compactly embedded in $Y_1$ or $Y$, thus preventing us from using the standard Aubin-Lions lemma to obtain the compactness of  $\mathcal{H}^2_\varphi \cap S$.

To overcome this issue, we introduce two additional spaces: $\bar{Y}:=L^2_{(\varphi+\tilde{\varphi})\bar{\varphi}}$ and $\bar{Y}_1$, defined as Hilbert spaces equipped with the inner product $(\phi,\psi):=\int_{\mathbb{R}_+^2 } \left(\phi \psi (\varphi+\tilde{\varphi}) +\frac{\partial \phi}{\partial K}\frac{\partial \psi}{\partial K}K^2\varphi +\frac{\partial \phi}{\partial N}\frac{\partial \psi}{\partial N}N^2\varphi\right) \bar{\varphi} dKdN$, and $\bar{\varphi}$ is a positive bounded smooth function and $\bar{\varphi} \to 0$ uniformly  as either $K$ or $N$ $\to 0$ or  $\infty$. For example, we can choose  $\bar{\varphi} = \frac{K}{K^2+1}\frac{N}{N^2+1}$. It is obvious that $Y$ is a subspace of $\bar{Y}$. 
The idea is that we use faster decay weights for spaces $\bar{Y}$ and $\bar{Y}_1$ to weaken their corresponding norms (topologies), so that we have more compact sets and make it possible to have a compact embedding.   We shall show that  ${\mathcal{H}}^2_\varphi$ is compactly embedded into $L^2([0,T];\bar{Y}_1)$ in Subsection \ref{sec: Compact embedding}. We now state the main result in this section.
\begin{thm}[Existence of weak solution of  nonlinear equation \eqref{eq: differential equation for original $lambda$}]\label{thm:weak solution of original equation}
	For any given $\psi \in S$, we have a unique weak solution $\lambda \in {\mathcal{H}}^2_\varphi \cap S$ to the linearized problem \eqref{eq: differential equation for $lambda$} in the sense 
	of Definition \ref{def: definiton of the weak solution of the linearized problem}.  Furthermore, the solution map $\Gamma: \psi \mapsto \lambda$ defined by solving the linearized problem \eqref{eq: differential equation for $lambda$} has a fixed point in ${\mathcal{H}}^2_\varphi \cap S$; this gives a weak solution to the nonlinear problem \eqref{eq: differential equation for original $lambda$}. 
\end{thm}

\begin{proof}
	We  shall show  that the solution map $\Gamma$ initially defined on $S\cap{C^{\infty}([0,T]\times\mathbb{R}^2_+)}$ can be extended to the whole $S$ (which is bounded, closed and convex   in $L^2([0,T];\bar{Y})$; see Lemma \ref{lem:S is convex and closed in L^2([0,T];Y)} below) continuously (under the norm of $L^2([0,T]; \bar{Y})$) in Lemma \ref{lem:extension lemma} below. The uniqueness follows from Corollary \ref{cor:Uniqueness of weak solution}. We then consider the extended map $\Gamma:S\to S$, which is well-defined and continuous under the norm of $L^2([0,T];\bar{Y})$.  Due to Theorem \ref{thm: A priori estimates}, $\Gamma(S)\subset  \mathcal{H}^2_{\varphi} \cap S$ is a bounded set in $ \mathcal{H}^2_{\varphi}$, and hence a pre-compact set  in  $L^2([0,T];\bar{Y}_1)$, then obviously a  pre-compact set  in  $L^2([0,T];\bar{Y})$ by Lemma \ref{lem:compactness of Hvarphi2}. It follows from the Schauder fixed point theorem that we have a fixed point in $S$, and this fixed point  belongs to $\mathcal{H}^2_{\varphi}\cap S$. 
\end{proof}

{\color{black}
	
	
	\begin{lem}\label{lem:S is convex and closed in L^2([0,T];Y)}
	The set	$S$ is bounded, closed and convex   in $L^2([0,T];\bar{Y})$.
		\begin{proof}
			
			The boundedness and convexity of $S$  is clear 
			and we only need to show that $S$ is closed in $L^2([0,T]; \bar{Y})$. Indeed, let $\psi_n$ be a sequence in $S$
			and converge to some $\psi \in L^2([0,T]; \bar{Y})$, which implies that
			\begin{equation*}
				\int_{0}^{T}||\psi_n(t)-\psi(t)||_{\bar{Y}}^2dt \to 0.
			\end{equation*}
			It follows that  after passing to a subsequence, which is still denoted by $\{\psi_n\}$, we have 
			\begin{equation*}
				||\psi_n(t)-\psi(t)||_{\bar{Y}}  \to 0 \quad   \textrm{a.e. $t$}.
			\end{equation*}
			Now, we fix a $t$ such that $||\psi_n(t)-\psi(t)||_{\bar{Y}}  \to 0$,  by passing to a subsequence $\{\psi_{k_n}\}$ again, we have $\psi_{k_n}(K,N,t)\to \psi(K,N,t)$ for a.e. $K$ and $N$. Since $c_0^{\gamma}\left(\frac{N}{K}\right)^{\gamma}\le\psi_{k_n}\le C_1\left(\frac{N}{K}\right)^{\gamma} + C_2\left(\frac{N}{K}\right)^{\beta}$ when $\beta\ge \gamma$, we conclude that $c_0^{\gamma}\left(\frac{N}{K}\right)^{\gamma}\le\psi\le C_1\left(\frac{N}{K}\right)^{\gamma} + C_2\left(\frac{N}{K}\right)^{\beta}$ for a.e. $K$ and $N$ and a.e. $t$, and hence $\psi\in S$. The case $\beta<\gamma$ follows in the same manner.
		\end{proof}
	\end{lem}

	\begin{lem}\label{lem:extension lemma}
		The solution map $\Gamma$ initially defined on $S\cap{C^{\infty}([0,T]\times\mathbb{R}^2_+)}$ can be  extended  to the whole set $S$ (under the norm of $L^2([0,T]; \bar{Y})$) continuously. Moreover, the following pointwise estimates of the weak solution $\lambda$ to linearized Equation \eqref{eq: differential equation for $lambda$}:
		\begin{equation*}
			c_0^{\gamma}\left(\frac{N}{K}\right)^{\gamma}\le\lambda \le C_1\left(\frac{N}{K}\right)^{\gamma} + C_2\left(\frac{N}{K}\right)^{\beta} \textrm{for a.e. $K$, $N$, and a.e. $t$,}
		\end{equation*}
		when $\beta \ge \gamma$, and 
		\begin{equation*}
			C_3 \left(\frac{N}{K}\right)^{\beta}\left(  C_4 \left(\frac{N}{K}\right)^{\beta-1} + C_5 \right)^{ \frac{\gamma-\beta}{\beta-1 } } \le \lambda\le C_6^{\gamma}\left(\frac{N}{K}\right)^{\gamma} \textrm{for a.e. $K$, $N$, and a.e. $t$,}
		\end{equation*}
		when  $\beta < \gamma$
		hold for all $\psi\in S$. Hence, the solution map  $\Gamma: S \to \mathcal{H}^2_{\varphi} \cap S \subset S$ is well-defined and continuous (under the norm of $L^2([0,T]; \bar{Y})$).
	\end{lem}
	\begin{proof}
		Let  $\{\psi_n\}$ be a sequence in $S\cap{C^{\infty}([0,T]\times\mathbb{R}^2_+)}$
		that converges to some $\psi \in S$ under the norm of $L^2([0,T]; \bar{Y})$. We shall show that the corresponding weak solution  $\lambda_n=\Gamma(\psi_n)$ will converge to some $\lambda \in L^2([0,T];\bar{Y})$, which is the weak solution corresponding to $\psi$.  By Theorem \ref{thm:existence when psi is smooth}, Theorem \ref{thm:crucial estimate theorem} and Theorem \ref{thm:L^2([0,T];Y_1) estimate}, we know that $\lambda_n \in S\cap{C^{1+\frac{\alpha}{2}, 2+ \alpha}_{loc}([0,T]\times\mathbb{R}^2_+)}\cap \mathcal{H}^2_{\varphi}$ and $\{\lambda_n\}$ is uniformly bounded in $\mathcal{H}^2_{\varphi}$. It follows from Banach-Alaoglu theorem that  there exists a subsequence of $\{\lambda_n\}$ which converges  to some $\lambda \in \mathcal{H}^2_{\varphi}$ weakly. 
		Using the fact that $ \mathcal{H}^2_{\varphi} $ is compactly embedded into $L^2([0,T]; \bar{Y}_1)$, this subsequence of $\{\lambda_n\}$ should converge to $\lambda$ in $L^2([0,T];\bar{Y}_1)$ strongly, and hence, also in $L^2([0,T];\bar{Y})$  strongly. We now fix this subsequence and still use $\{\lambda_n\}$ to denote it. Similar to  the proof of Lemma \ref{lem:S is convex and closed in L^2([0,T];Y)}, by passing to further subsequence, we conclude that $\lambda \in S$,
		and hence, $\lambda\in  \mathcal{H}^2_{\varphi} \cap S $.
	Finally, using the weak convergence, the pointwise bounds on $\lambda$ and the property of  weights, we can show  that the corresponding integrals in the definition of weak solutions will also converge accordingly and hence the following holds:
		\begin{claim}\label{claim: the solution map is continuous}
			$\lambda$ is a weak solution to the linearized problem  \eqref{eq: differential equation for $lambda$} with the input $\psi$ in the sense of Definition \ref{def: definiton of the weak solution of the linearized problem}, i.e., $\lambda = \Gamma(\psi)$. 
		\end{claim}
		The technical details of the proof of Claim \ref{claim: the solution map is continuous} will be presented in  Appendix \ref{appen:Solvability of X1 and Probabilistic representation of lambda}. 
		Now by the uniqueness result  in Corollary \ref{cor:Uniqueness of weak solution}, we know that the whole sequence $\lambda_n$ converges to this $\lambda$, indeed. 
		We then have the following result: for any (not necessarily smooth) $\psi_n \in S$, $\psi_n\to \psi \in S$   under the norm $L^2([0,T]; \bar{Y})$ the corresponding weak solutions $\{\lambda_n\}$  converge to $\lambda$ in $L^2([0,T];\bar{Y}_1)$ and of course in $L^2([0,T];\bar{Y})$, hence, the solution map $\Gamma:$ $S \to S$ is continuous under the norm of  $L^2([0,T];\bar{Y})$.
	\end{proof}

}
{\color{black} 

	\subsection{Compact embedding of ${\mathcal{H}}^2_\varphi  $ into $L^2([0,T];\bar{Y}_1)$}\label{sec: Compact embedding}

	In this subsection, we give several compactness results which have been used in the proof of Theorem \ref{thm:weak solution of original equation} to ensure the existence of a fixed point.
	
	\begin{lem}[Compactness of $Y_1$]\label{lem: the compactness of Y1}

		$Y_2 $ is  compactly embedded into $\bar{Y}_1$ and $\bar{Y}_1$ is continuously embedded into $\bar{Y}$.
	\end{lem}
	The fact that $\bar{Y}_1$ is continuously embedded into $\bar{Y}$ is obvious. For the first part of this lemma, 
	using the decay property of the weights, we can prove it by a diagonal argument together with the classical Rellich's lemma (cf. \cite[ Lemma 6.2.2 ]{HKUHUB_Hao}). Indeed, these kinds of Rellich's lemma and its enhancements in unbounded domains are very subtle. Interested readers can see also \cite{Bensoussan22}.

	\begin{lem}\label{lem:compactness of Hvarphi2}
		The space ${\mathcal{H}}^2_\varphi  $ is compactly embedded into $L^2([0,T];\bar{Y}_1)$.
	\end{lem}
	\begin{proof}
		We note that 
		\begin{equation*}
			\begin{aligned}
				\mathcal{H}^2_\varphi &=  \left\{ \psi\in L^2([0,T];Y_2)\left| \frac{\partial \psi}{\partial t}\in L^2([0,T];Y)\right.\right\}  \subset \left\{ \psi\in L^2([0,T];Y_2)\left| \frac{\partial \psi}{\partial t}\in L^2([0,T];\bar{Y})\right.\right\}.
			\end{aligned}
		\end{equation*}
		By  the Aubin–Lions lemma and Lemma 
		\ref{lem: the compactness of Y1},  ${\mathcal{H}}^2_\varphi $ is compactly embedded into $L^2([0,T]; \bar{Y}_1)$. 
	\end{proof}
	
}

\begin{rem}
	If we use a faster decay weight for the norm in $\bar{Y}$, for example we define $\bar{Y}:=L^2_{(\varphi+\tilde{\varphi})\bar{\varphi}^2}$, then one can prove that $\bar{Y}_1$ is compactly embedded into $\bar{Y}$ by a similar argument for Lemma \ref{lem: the compactness of Y1}.
\end{rem}     

\subsection{Existence, regularity, and uniqueness of solutions to the nonlinear problem and the original HJB equation}\label{Existence  and uniqueness to the original HJB equation}
We first consider the regularity of the weak solution to Equation \eqref{eq: differential equation for original $lambda$} obtained in  Section \ref{sec:The extension lemma and fixed point argument}. Then we shall show the unique existence of classical solutions to the original problem \eqref{eq:second HJB for v(K,N)}.
\begin{thm}[Smoothness  of the solution to Equation \eqref{eq: differential equation for original $lambda$}]\label{thm: classical soul of nonlinear eq of lambda}
	The weak solution  of  Equation \eqref{eq: differential equation for original $lambda$}   given by	Theorem \ref{thm:weak solution of original equation} is indeed classical.
\end{thm}
\begin{proof}
	Right now we have a weak solution for Equation \eqref{eq: differential equation for original $lambda$} with pointwise estimates so that $\lambda \in S$. By a deep regularity result deals with rough coefficients equations originally due to Nash (Chapter 3, Theorem 10.1 in \cite{ladyzhenskaia1988linear}), we know that $\lambda$ is indeed locally H\"older continuous, that is in  $C_{loc}^{\frac{\alpha}{2},\alpha}( [0,T)\times \mathbb{R}_+^2 )$  for some positive $\alpha$.
	By interior Schauder estimates this  solution $\lambda$ to  Equation\eqref{eq: differential equation for original $lambda$} is in $C_{loc}^{1+\frac{\alpha}{2},2+\alpha}( [0,T)\times \mathbb{R}^2_+)$
	(Chapter 4, Theorem 5.1 in \cite{ladyzhenskaia1988linear}).
	If $f$  in \eqref{eq: differential equation for original $lambda$} is smooth, we can repeatedly use this argument and  get a smooth solution. 
\end{proof}
Now we have proven that the solution  $\lambda$ to \eqref{eq: differential equation for original $lambda$}  exists in the classical sense, we try to use this $\lambda$ to obtain $v$. There are several ways to do so. 
One way is replacing  $v_K$ by $\lambda$ in  Equation \eqref{eq:second HJB for v(K,N)}  viewing this $\lambda$ as given, and then solving the resulting equation.  We then obtain 
\begin{equation}\label{eq:third HJB for v(K,N)}
	\left\{
	\begin{aligned}
		&0 = v_t + Nu_1\left((u'_1)^{-1}(\lambda)\right) -N\lambda\cdot(u'_1)^{-1}(\lambda) + AK^{\beta}N^{1-\beta}\lambda
		+f(N)v_N 
		+\frac{1}{2}{\epsilon}^2K^2v_{KK} + \frac{1}{2}{\sigma}^2N^2v_{NN},\\
		&v(T, K, N) = Nu_2\left(\frac{K}{N}\right).\\
	\end{aligned}
	\right.
\end{equation}
It means that we treat $Nu_1\left((u'_1)^{-1}(\lambda)\right) -N\lambda\cdot(u'_1)^{-1}(\lambda) + AK^{\beta}N^{1-\beta}\lambda$ in \eqref{eq:third HJB for v(K,N)} as given source terms, then solve the resulting  parabolic equation uniquely. 
We need to be more precise  for what we mean ```uniquely'', as there is  no uniqueness for the linear equation \eqref{eq:third HJB for v(K,N)} when there is no restriction on the growth rate of the  solution itself (there are some ``physically incorrect'' solutions). For this purpose, we 
let $K=e^x$, $N=e^y$, $\tilde{v}(t, x,y) = v(t, e^x, e^y)$, and $\tilde{\lambda}(t, x,y) = \lambda(t, e^x, e^y)$, Equation \eqref{eq:third HJB for v(K,N)} is transformed to 
\begin{equation}\label{eq: equation for tildev}
	\left\{
	\begin{aligned}
		&\tilde{v}_{t} - \frac{1}{2}\epsilon^2\tilde{v}_{x} + \left(\frac{f(e^y)}{e^y}-\frac{1}{2}\sigma^2\right)\tilde{v}_{y}  +\frac{1}{2}\epsilon^2\tilde{v}_{xx} + \frac{1}{2}\sigma^2\tilde{v}_{yy}\\
		&  = -e^{y} u_1\left(j(\tilde{\lambda})\right) + e^{y}\tilde{\lambda} j(\tilde{\lambda})- Ae^{\beta x +(1-\beta)y}\tilde{\lambda}, ~
		x, y \in \mathbb{R}, ~t \in [0,T),\\
		&\tilde{v}(T,x,y)= e^{y}u_2(e^{x-y}),\\
	\end{aligned}\right.
\end{equation}
where we recall that $j(\cdot) = (u'_1)^{-1}(\cdot)$. The existence of the solution to this linear equation \eqref{eq: equation for tildev} is clear due to 
its uniformly parabolic nature.
Moreover, since the coefficients in front of $\tilde{v}$ and its various derivatives  are bounded, we conclude that there exists at most one solution $\tilde{v}$ to the  above equation satisfying $|\tilde{v}|\le Be^{C(x^2+y^2)}$ for some positive constants $B$ and $C$, the claim follows by referring to, for example  \cite [Chapter 2, Section 4, Theorem 10]{MR0181836}.
It means that Equation \eqref{eq:third HJB for v(K,N)} is uniquely solvable within the  function class with the growth rate no faster than $Be^{C((\ln K)^2+(\ln N)^2)}$ for some positive constants $B$ and $C$.

Next, we need to verify that the solution $v$ (or $\tilde{v}$) we obtained by solving \eqref{eq:third HJB for v(K,N)} (or \eqref{eq: equation for tildev}) indeed satisfies the compatible condition $v_K=\lambda$.
For this purpose, differentiating Equation \eqref{eq:third HJB for v(K,N)} with respect to $K$, we have \begin{equation*}
	\left\{
	\begin{aligned}
		&{\bar{\lambda}}_t - Nj({\lambda})\frac{\partial {\lambda}}{\partial K} + A N^{1-\beta}K^{\beta}\frac{\partial {\lambda}}{\partial K} +{\epsilon}^2K\frac{\partial \bar{\lambda}}{\partial K} + \frac{1}{2}{\epsilon}^2K^2\frac{\partial^2 \bar{\lambda}}{\partial K^2} \\
		&\quad + f(N)\frac{\partial \bar{\lambda}}{\partial N} +  \frac{1}{2}{\sigma}^2N^2 \frac{\partial^2 \bar{\lambda}}{\partial N^2} + A\beta N^{1-\beta}K^{\beta-1}{\lambda}= 0,\\
		&\bar{\lambda}(T,K,N) = u'_2\left(\frac{K}{N}\right),
	\end{aligned}
	\right.
\end{equation*}
where $\bar{\lambda}:=v_K$, and $v$ is the solution to  \eqref{eq:third HJB for v(K,N)}.
Comparing with Equation \eqref{eq: differential equation for original $lambda$} for $\lambda$, $\bar{\lambda}-\lambda$ will satisfy the following equation
\begin{equation*}
	\left\{
	\begin{aligned}
		&\frac{\partial}{\partial t}({\bar{\lambda}} - \lambda)  +{\epsilon}^2K\frac{\partial ({\bar{\lambda}} - \lambda)}{\partial K} + \frac{1}{2}{\epsilon}^2K^2\frac{\partial^2 ({\bar{\lambda}} - \lambda)}{\partial K^2} 
		+ f(N)\frac{\partial ({\bar{\lambda}} - \lambda)}{\partial N} +  \frac{1}{2}{\sigma}^2N^2 \frac{\partial^2 ({\bar{\lambda}} - \lambda)}{\partial N^2} = 0,\\
		&({\bar{\lambda}} - \lambda)(T,K,N) = 0.
	\end{aligned}
	\right.
\end{equation*}
It is then obvious that $\bar{\lambda} (t,K,N)-\lambda(t,K,N) \equiv 0$ (within the ``physically correct'' function class), hence, $\bar{\lambda} \equiv \lambda$. Indeed, if not, then $\bar{\lambda}-\lambda$ and hence  $\bar{\lambda}$ grows faster than $Be^{C((\ln K)^2+(\ln N)^2)}$ for any positive constants $B$ and $C$ since $\lambda$ has pointwise estimates stated in the set $S$ defined in \eqref{def:definition for the growth rate set}; this will contradict with  the fact that $|v_K| \lesssim \frac{1}{K}\rho(\ln K, \ln N)$ which we will prove now, and	 $\rho$ is defined in \eqref{eq:definofrho} below. 

To see this fact, 
we  choose  the weight 
\begin{equation}\label{eq:definofrho}
	\rho(x,y) = e^x +e^y  \textrm{ for }  \beta \ge \gamma, ~ \textrm{ and } \rho(x,y) = e^x +e^y +e^{-x} + e^{-8x} +e^{8y} \textrm{ for } \beta < \gamma,
\end{equation}
and consider $w:= \frac{\tilde{v}}{\rho}$, it satisfies 
\begin{equation}\label{eq:differential eq for tildev/h}
	\left\{
	\begin{aligned}
		&{w}_{t} +\left(  - \frac{1}{2}\epsilon^2 + \epsilon^2\frac{\rho_x}{\rho}\right) {w}_{x} + \left(\frac{f(e^y)}{e^y}-\frac{1}{2}\sigma^2 + \sigma\frac{\rho_y}{\rho}\right) {w}_{y} 
		+ \frac{1}{2}\epsilon^2{w}_{xx} + \frac{1}{2}\sigma^2{w}_{yy}  \\
		&+\left[- \frac{1}{2}\epsilon^2 \frac{\rho_x}{\rho} + \left(\frac{f(e^y)}{e^y}-\frac{1}{2}\sigma^2 \right)\frac{\rho_y}{\rho} + \frac{1}{2}\epsilon^2\frac{\rho_{xx}}{\rho} + \frac{1}{2}\sigma^2\frac{\rho_{yy}}{\rho} \right]w \\
		&\quad = \frac{-e^{y} u_1\left(j(\tilde{\lambda})\right) + e^{y}\tilde{\lambda}\cdot j(\tilde{\lambda}) - Ae^{\beta x +(1-\beta)y}\tilde{\lambda} }{\rho}, ~
		x, y \in \mathbb{R}, ~t \in [0,T),\\
		&{w}(T,x,y)= \frac{e^{y}u_2(e^{x-y})}{\rho} \simeq \frac{e^{(1-\gamma) x+  \gamma y}}{\rho}.\\
	\end{aligned}\right.
\end{equation}
When $\beta\ge \gamma$, using  Assumptions \eqref{Assum: running utility function $u_1$} on $u_1$ and conditions \eqref{eq:condition on j(x)} on $j(x)$, together with $\lambda >rsim \left(\frac{N}{K}\right)^{\gamma}$, we find that $\left|-e^{y} u_1\left(j(\tilde{\lambda})\right)\right.$ $\left. + e^{y}\tilde{\lambda}\cdot j(\tilde{\lambda})\right|$ $ \lesssim e^{(1-\gamma)x+\gamma y} $. Also, by the estimate $|\tilde{\lambda}| \lesssim e^{\beta(y-x)} + e^{\gamma(y-x)}$, we have $\left|Ae^{\beta x +(1-\beta)y}\tilde{\lambda}\right|$ $\lesssim$ $e^y$$+e^{(\beta-\gamma)x+(1-\beta+\gamma)y}$. When $\beta\ge \gamma$, 
$e^{(\beta-\gamma)x+(1-\beta+\gamma)y} \lesssim e^{x} +e^{y}$.
Then we easily verify that the solution $w$ and its derivatives are all bounded as the coefficients, the source term and the terminal data in \eqref{eq:differential eq for tildev/h} are all bounded and Lipschitz continuous. And  we can then obtain that $|\tilde{v}|, |\tilde{v}_x|$, $|\tilde{v}_y|$ $\lesssim \rho$. The case $\beta<\gamma$ can be proved similarly. By a change of variable, we have proven that $|v_K| \lesssim \frac{1}{K}\rho(\ln K, \ln N)$. Hence, we finish the proof of the compatible condition $v_K=\lambda$.

Our main goal  next is  to show the uniqueness of the solution $v$ to the  HJB equation \eqref{eq:second HJB for v(K,N)} within  the class of functions satisfying  a certain growth rate  at infinity and near the boundary $K=0$ and $N=0$. We have:
\begin{thm}[Smoothness  and uniqueness of the solution to HJB \eqref{eq:second HJB for v(K,N)}]\label{thm:uniqueness for HJB}
	Under the same assumption of Theorem \ref{thm:weak solution of original equation}, we have a  classical solution	for the original HJB equation \eqref{eq:second HJB for v(K,N)}. Moreover,  it is a unique one that satisfies the following:
	\begin{equation*}
		\left\{
		\begin{aligned}
		&B_1\left(\frac{N}{K}\right)^{\gamma} \le \frac{\partial v}{\partial K}(t,K,N) \le B_2 \left(\frac{N}{K}\right)^{\gamma} + B_3  \left(\frac{N}{K}\right)^{\beta} ~\textrm{when} ~ \beta \ge \gamma;\\
		&B_4  \left(\frac{N}{K}\right)^{\beta}\left(  B_5 \left(\frac{N}{K}\right)^{\beta-1} + B_6\right)^{ \frac{\gamma-\beta}{\beta-1 } }\le \frac{\partial v}{\partial K} (t,K,N) \le B_7\left(\frac{N}{K}\right)^{\gamma} ~\textrm{when} ~ \beta < \gamma,
		\end{aligned}
		\right.
	\end{equation*}
	and 
	\begin{equation}\label{eq: requirment at far field}
		|v| \le B_8 K^{\beta_1}+ B_8N^{\gamma_1}+B_9,
	\end{equation}
	for some positive  constants  $B_1$, $B_2$, $B_3$, $B_4$, $B_5$, $B_6$, $B_7$, $B_8$, $B_9$, and (not necessarily positive)  constants $\beta_1$ and  $\gamma_1$.
\end{thm}
\begin{proof}
	We have obtained a classical solution from $\lambda = \frac{\partial v}{\partial K}$ for  Equation \eqref{eq:second HJB for v(K,N)}.
	For the uniqueness of the classical solution, we suppose that there exist two classical solutions $v_1$ and $v_2$ to Equation \eqref{eq:second HJB for v(K,N)}, let $v_3 = v_1 -v_2$,  $K=e^x$ and $N=e^y$. We consider $\tilde{v}_3(t, x,y) := v_3(t, e^x, e^y)$, it satisfies the following equations:
	\begin{equation}\label{eq: equation for tilde lamba_3}
		\left\{
		\begin{aligned}
			&\quad \tilde{v}_{3t} +\left( Ae^{(y-x)(1-\beta)} - \frac{1}{2}\epsilon^2\right)\tilde{v}_{3x} + \left(\frac{f(e^y)}{e^y}-\frac{1}{2}\sigma^2\right)\tilde{v}_{3y} +
	 \frac{1}{2}\epsilon^2\tilde{v}_{3xx} + \frac{1}{2}\sigma^2\tilde{v}_{3yy}\\
			&= e^{y}\left(u_1\left(j\left(\frac{\tilde{v}_{1x}}{e^x}\right)\right) - u_1\left(j\left(\frac{\tilde{v}_{2x}}{e^x}\right)\right)\right) 
			-e^{y-x}\left(\tilde{v}_{1x}j\left(\frac{\tilde{v}_{1x}}{e^x}\right) - \tilde{v}_{2x}j\left(\frac{\tilde{v}_{2x}}{e^x}\right) \right)
			~x, y \in \mathbb{R}, ~t \in [0,T),\\
			&\quad\tilde{v}_3(T,x,y)=0,\\
		\end{aligned}\right.
	\end{equation}
	where $j(x)  = (u_1')^{-1}(x)$, $\tilde{v}_1(t, x,y) = v_1(t, e^x, e^y)$ and $\tilde{v}_2(t, x,y) = v_2(t, e^x, e^y)$.		   
	Note that 
	\begin{equation*}
		\begin{aligned}
			&e^{y}u_1\left(j\left(\frac{\tilde{v}_{1x}}{e^x}\right)\right) - e^{y} u_1\left(j\left(\frac{\tilde{v}_{2x}}{e^x}\right)\right) = e^{y}\theta_1 j'(\theta_1)\left(\frac{\tilde{v}_{1x}}{e^x} - \frac{\tilde{v}_{2x}}{e^x} \right) = e^{y-x}\theta_1 j'(\theta_1)\tilde{v}_{3x} \\ 
			&e^{y-x}\tilde{v}_{1x} j\left(\frac{\tilde{v}_{1x}}{e^x}\right) - e^{y-x} \tilde{v}_{2x}j\left(\frac{\tilde{v}_{2x}}{e^x}\right) = e^{y}(\theta_2 j'(\theta_2) + j(\theta_2)) \left(\frac{\tilde{v}_{1x}}{e^x} - \frac{\tilde{v}_{2x}}{e^x} \right) 
			= 
			e^{y-x} (\theta_2 j'(\theta_2) + j(\theta_2))\tilde{v}_{3x},
		\end{aligned}
	\end{equation*}
	for some $\theta_1$ and $\theta_2$ between $\frac{\tilde{v}_{1x}}{e^x}$ and $\frac{\tilde{v}_{2x}}{e^x}$.
We shall apply the maximum principle to conclude that $w:= \frac{\tilde{v}_{3}}{\rho}$ is 0 for a suitable weight $\rho$ to be specified later, similar to the proof of Theorem \ref{thm:Maximum principle}.
And we only need to check the following two facts
	\begin{equation}\label{eq: decay at far field}
		\lim_{R \to \infty} \sup_{|x| + |y| \ge R} |w| \to 0 ~ \textrm{uniformly in $t \in [0,T]$};
	\end{equation}
	\begin{equation}\label{eq:boundess}
		Ae^{(y-x)(1-\beta)}\frac{\rho_x}{\rho} ~\textrm{is bounded from above}.
	\end{equation}
	When $\beta \ge \gamma$, using $\frac{\tilde{v}_{1x}}{e^x}, \frac{\tilde{v}_{2x}}{e^x} \ge c_0^{\gamma}e^{\gamma(y-x)}$ and conditions \eqref{eq:condition on j(x)} on $j(x)$ we know 
	$e^{y-x}\theta_1 j'(\theta_1)$ and $e^{y-x} (\theta_2 j'(\theta_2) + j(\theta_2))$  are bounded. We first consider the case of $0<\beta_1,\gamma_1<1$ and explain the main idea. In this case, we can choose weight $\rho =  e^{-x+5y} + e^{3(x-y)} + e^{x-y} + e^{-x+y} + e^{-(x+y)}$.
	For this purpose, note that $\frac{1}{2}e^{-x+5y} + \frac{1}{2}e^{3(x-y)} \ge e^{x+y}$, and hence $\rho >rsim  e^{x+y} + e^{x-y} + e^{-x+y} + e^{-(x+y)}$
	$\ge e^{|x|+|y|}$, then by assumption \eqref{eq: requirment at far field} and $0<\beta_1, \gamma_1<1$, \eqref{eq: decay at far field} follows easily. Moreover, 
	\begin{equation*}
		Ae^{(y-x)(1-\beta)}\frac{\rho_x}{\rho} = \frac{e^{(2+\beta)x - (2+\beta)y} + e^{\beta x- \beta y} + ~\textrm{negative terms}}{e^{-x+5y} + e^{3(x-y)} + e^{x-y} + e^{-x+y} + e^{-(x+y)}}.
	\end{equation*}
	The fact \eqref{eq:boundess}, 
	$Ae^{(y-x)(1-\beta)}\frac{\rho_x}{\rho} $ is bounded from above  follows from 
	\begin{equation*}
		\begin{aligned}
			e^{(2+\beta)x - (2+\beta)y}  \lesssim  e^{3(x-y)}+ 1,\quad
			e^{\beta x- \beta y}\lesssim e^{x-y}+ 1,\quad
			\rho >rsim e^{3(x-y)} + e^{x-y}+ 1.
		\end{aligned}
	\end{equation*}
	For general cases,
	we  use 
	$\rho =  e^{-x+(2k_1-1) y} + e^{k_1(x-y)} + e^{k_2(x-y)} + e^{-k_2x+k_2y} + e^{-k_2(x+y)}$ 
	to be our weight, where $k_1 =1 + \max\{2 |\gamma_1| + 2\epsilon_0, 2|\beta_1| + 2\epsilon_0\}$  and $k_2 =\max\{ |\gamma_1| + \epsilon_0, |\beta_1| + \epsilon_0\}$ for a fixed $\epsilon_0>0$. Then $\rho>rsim e^{\frac{k_1-1}{2}(x+y)} + e^{k_2(x-y)} + e^{-k_2x+k_2y} + e^{-k_2(x+y)} \ge e^{k_2(|x|+|y|)}$ by our choice of $k_2$. Then by \eqref{eq: requirment at far field}, \eqref{eq: decay at far field}  follows. The fact  \eqref{eq:boundess} can be obtained by a direct calculation similarly as above. 
	
	When $\beta < \gamma$, we need to deal with the additional term 	$e^{y-x}\theta_1 j'(\theta_1) \frac{\rho_x}{\rho}\simeq - e^{y-x} j(\theta_1) \frac{\rho_x}{\rho}$ and $e^{y-x} (\theta_2 j'(\theta_2) + j(\theta_2)) \frac{\rho_x}{\rho}\simeq - e^{y-x} j(\theta_2) \frac{\rho_x}{\rho}$. This could also be done    by a direct calculation and noting that $\theta_1, \theta_2 \ge  B_4 e^{\beta(y-x)}\left( B_5e^{(\beta-1)(y-x)} + B_6\right)^{\frac{\gamma-\beta}{\beta-1}} $ (cf. \cite[Section 6.3]{HKUHUB_Hao}). We complete the proof.
\end{proof}

	\noindent\textbf{Acknowledgements} 
	
	
	Tak Kwong Wong was
	partially supported by the HKU Seed Fund for Basic Research under the project code 201702159009,
	the Start-up Allowance for Croucher Award Recipients, and Hong Kong General Research Fund (GRF)
	grants with project numbers 17306420, 17302521, and 17315322.  
	Phillip Yam acknowledges the financial supports from HKGRF-14301321 with the project title ``General Theory for Infinite Dimensional Stochastic Control: Mean Field and Some Classical Problems'', and HKGRF-14300123 with the project title ``Well-posedness of Some Poisson-driven Mean Field Learning Models and their Applications''. The work described in this article was supported by a grant from the Germany/Hong Kong Joint Research Scheme sponsored by the Research Grants Council of Hong Kong and the German Academic Exchange Service of Germany (Reference No. G-CUHK411/23). He also thanks The University of Texas at Dallas for the kind invitation to be a Visiting Professor in Naveen Jindal School of Management during his sabbatical leave.
	
	\appendix
	\section{Technical Proofs}\label{appen:Solvability of X1 and Probabilistic representation of lambda}
	We give  the technical details of  the respective proof for  Lemma \ref{Solvability of $X_1$}, Lemma \ref{lem:Probabilistic representation for lambda}, Theorem \ref{eq:L^2([0,T];Y_1) estimate} and Claim \ref{claim: the solution map is continuous}.
	\begin{proof}[Proof of Lemma \ref{Solvability of $X_1$}]
		By applying It\^o's lemma, we have that
		\begin{equation}
			\begin{aligned}
				dX_1^{1-\beta} 
				&= (1-\beta)X_1^{-\beta}dX_1-\frac{1}{2}(1-\beta)\beta X_1^{-\beta-1} d \langle X_1\rangle_{t}\\
				&=(1-\beta)X_1^{-\beta}\left( \left[AX_2^{1-\beta}(s)X_1^{\beta}(s) -X_2(s)j\left(\psi\left(s,X_1(s),X_2(s)\right)\right) + {\epsilon}^2X_1(s)\right]ds  \right)\\
				&\quad+(1-\beta){\epsilon}X_1^{1-\beta}(s)dB_1(s)-\frac{1}{2}(1-\beta)\beta X_1^{-\beta-1}(\epsilon^2X_1^2ds)\\
				&= (1-\beta) \left[AX_2^{1-\beta}(s)-X_1^{1-\beta}(s)\frac{X_2(s)}{X_1(s)}j\left(\psi\left(s,X_1(s),X_2(s)\right)\right) + {\epsilon}^2X_1^{1-\beta}(s)\right]ds \\
				&\qquad \qquad + (1-\beta){\epsilon}X_1^{1-\beta}(s)dB_1(s)-\frac{1}{2}(1-\beta)\beta\epsilon^2 X_1^{1-\beta}ds,\\
			\end{aligned}
		\end{equation}
		where $ \langle X_1\rangle_t$ is the quadratic variation process of $X_1$.
		Since $\frac{X_2(s)}{X_1(s)}j\left(\psi\left(s,X_1(s),X_2(s)\right)\right)$ is locally bounded because of \eqref{eq:condition on j(x)} and that $\psi$ is demanded to  satisfy the lower bound  stated in Theorem \ref{thm:crucial estimate theorem} and Theorem \ref{thm:crucial estimate theorem for case beta<gamma}, and it is also continuous, then $X_1^{1-\beta}$ admits a solution by using Theorem 2.3 and 2.4 in Section 3 \cite{ikeda2014stochastic}. Similarly for $X_2$, we can use the comparison principle to show that $X_1^{1-\beta}$ is positive; this implies $X_1$ admits a positive solution. 
	\end{proof}
	
	\begin{proof}[Proof of Lemma \ref{lem:Probabilistic representation for lambda}]
		We consider  $
		Y(s) =e^{\int_{t}^{s}A\beta(X_2(\tau))^{1-\beta}(X_1(\tau))^{\beta-1}d\tau}$$\lambda(s,X_1(s),X_2(s)).
		$
		It follows from  It\^o's lemma and Equation $\eqref{eq: differential equation for $lambda$}_1$ that
		
		\begin{align*}
			dY(s)&= 
			e^{\int_{t}^{s}A{\beta}(X_2(\tau))^{1-\beta}(X_1(\tau))^{\beta-1}d\tau}A\beta(X_2(s))^{1-\beta}(X_1(s))^{\beta-1}\lambda(s,X_1(s),X_2(s))ds\\
			&\quad +e^{\int_{t}^{s}A{\beta}(X_2(\tau))^{1-\beta}(X_1(\tau))^{\beta-1}d\tau}\left[ \frac{\partial \lambda}{\partial s}ds + \frac{\partial \lambda}{\partial K} (AX_1^{\beta}X_2^{1-\beta}  -X_2j(\psi)+ {\epsilon}^2X_1)ds  \right.\\
			&\quad \left.+\frac{\partial \lambda}{\partial N} f(X_2(s))ds + \epsilon \frac{\partial \lambda}{\partial K}X_1dB_1(s) + \sigma \frac{\partial \lambda}{\partial N}X_2dB_2(s) \right]\\
			& \quad+  e^{\int_{t}^{s}A{\beta}(X_2(\tau))^{1-\beta}(X_1(\tau))^{\beta-1}d\tau}\left[ \frac{1}{2}\epsilon^2\frac{\partial ^2\lambda}{\partial K^2}X_1^2ds +\frac{1}{2}{\sigma}^2\frac{\partial ^2 \lambda}{\partial N^2}X_2^2ds \right]\\
			& =  e^{\int_{t}^{s}A{\beta}(X_2(\tau))^{1-\beta}(X_1(\tau))^{\beta-1}d\tau}\left(\epsilon \frac{\partial \lambda}{\partial K}X_1dB_1(s) + {\sigma_N}\frac{\partial \lambda}{\partial N}X_2dB_2(s)\right).
		\end{align*}
		It implies that
		\begin{equation*}
			\begin{aligned}
				Y(T)-Y(t) = \int_{t}^{T}e^{\int_{t}^{s}A{\beta}(X_2(\tau))^{1-\beta}(X_1(\tau))^{\beta-1}d\tau}\left(\epsilon \frac{\partial \lambda}{\partial K}X_1dB_1(s) + {\sigma_N}\frac{\partial \lambda}{\partial N}X_2dB_2(s)\right).
			\end{aligned}
		\end{equation*}
		Taking expectations on both sides, by also stopping time argument, so as to vanish the martingale part,
		we finally obtain
		\begin{equation*}
			\begin{aligned}
				\lambda(t,K,N)	&= \mathbb{E}[Y(t)] = \mathbb{E}[Y(T)] 
				=\mathbb{E}\left[e^{\int_{t}^{T}A\beta(X_2(\tau))^{1-\beta}(X_1(\tau))^{\beta-1}d\tau}\lambda(T,X_1(T),X_2(T))\right] \\
				&= \mathbb{E}\left[e^{\int_{t}^{T}A\beta(X_2(\tau))^{1-\beta}(X_1(\tau))^{\beta-1}d\tau}u'_2\left(\frac{X_1(T)}{X_2(T)}\right)\right].
			\end{aligned}
		\end{equation*} 
		This completes the proof. 
	\end{proof}
	
	\begin{proof}[Proof of Theorem \ref{eq:L^2([0,T];Y_1) estimate}]
		Choose a sequence of cut-off functions $\xi_m(K,N)$ on $\mathbb{R}^2_+$, such that $\xi_m(K,N)$ $\in$ $C^{\infty}_0({\mathbb{R}^2_+})$ and 
		\begin{equation*}
			\begin{aligned}
				&\textrm{(i)}~ \xi_m(K,N) = 1 ~ \textrm{on $\left[\frac{1}{m}, m \right]\times \left[\frac{1}{m},m \right]$}; \quad \textrm{(ii)} ~ 0\le \xi_m(K,N) \le 1 ~ \textrm{on $\mathbb{R}_+^2$};\\
				&  \textrm{(iii)} ~\xi_m(K,N)=0 ~ \textrm{outside $\left[\frac{1}{2m}, 2m\right]\times\left[\frac{1}{2m},2m\right]$},\quad
			\textrm{(iv)}~ 	N \left|\frac{\partial \xi_m}{\partial N} \right|\le 10, \quad ~K\left|\frac{\partial \xi_m}{\partial K}\right|\le 10.\\
			\end{aligned}
		\end{equation*}
		At this stage we only know that $\lambda$ is classical, whether $\lambda(t) \in Y_1$ is not clear. That is the reason why we need to use a cutoff function. 
		Using $\phi : =\lambda(t) \xi_m^2$ as the test function in \eqref{eq:another writing for weak solution}, we have
		\begin{equation}\label{eq:testfunc}
			\left\{
			\begin{aligned}
				&\langle\lambda_t, \lambda \xi_m^2\rangle_{Y_0, Y_1}   + B[\lambda(t),\lambda \xi_m^2;t]  =0 ~\textrm{a.e.}~  t\in [0,T],\\
				&\lambda(T,K,N) = u'_2\left(\frac{K}{N}\right),
			\end{aligned}
			\right.
		\end{equation}
		where $B[\lambda(t),\phi;t] $ is defined in \eqref{eq:bilinear form}.
		We estimate $B[\lambda(t),\lambda(t)\xi_m^2;t] $ first.  For $\beta \ge \gamma$  using Conditions \eqref{eq:condition on j(x)}  on $j(x)$ and the fact that $\psi \ge c_0^{\gamma} \left(\frac{N}{K}\right)^{\gamma}$ we have $0<j(\psi)^2\left(\frac{N}{K}\right)^2\le \frac{a_1^2}{c_0^2}$. Then using  Young's inequality, we have
		\begin{equation*}
			\begin{aligned}
				&\quad\left|\int_{\mathbb{R}_+^2 } N\frac{\partial \lambda}{\partial K}j(\psi)\lambda \xi_m^2 \varphi dKdN\right| 
				\le {\tilde{\alpha}}  \left|\left|\frac{\partial \lambda}{\partial K} \xi_m \right|\right|^2_{L^2_{1,K,\varphi}} + \frac{a_1^2}{4\tilde{\alpha}c_0^2} ||\lambda \xi_m||_{L^2_{\varphi}}^2.
			\end{aligned}
		\end{equation*}
		where ${\tilde{\alpha}}$ is a positive constant that will be determined later.
		For $\beta < \gamma$, similar to Equation \eqref{eq: estimate for term j(psi)}, we have 
		\begin{equation}\label{eq:estimate for j(psi)}
			\begin{aligned}
				j(\psi) \frac{N}{K} &\le  a_1 \frac{N}{K} \left(C_3  \left(\frac{N}{K}\right)^{\beta}\left(  C_4 \left(\frac{N}{K}\right)^{\beta-1} + C_5\right)^{ \frac{\gamma-\beta}{\beta-1 } }\right)^{-\frac{1}{\gamma}}\\
				& \le a_1 a_2^{-\frac{1}{\gamma}}e^{-\frac{1}{\gamma}b_5pT + \frac{\gamma-\beta}{\gamma(1-\beta)} bT } + 
				a_1 a_2^{-\frac{1}{\gamma}} e^{-\frac{1}{\gamma}b_5pT}
				\left(A(1-\beta)\frac{-1+ e^{bT}}{b} \right)^{ \frac{\gamma-\beta}{\gamma(1-\beta) } } \left(\frac{N}{K}\right)^{1- \frac{\beta}{\gamma} }  \\
				&\le b_7(C_f, a_1, a_2, \sigma, \epsilon, \gamma, \beta, A, T) +  \left(\frac{N}{K}\right)^{1-\beta},
			\end{aligned}
		\end{equation}
		for $p =\frac{\gamma +1 -2\beta}{1-\beta}$ and some  constant $b_7>0$. 
		For  the later use, we also assume that $b_7$ is large enough such that $j(\psi) \frac{N}{K} \le b_7 + \left(\frac{N}{K}\right)^{2-2\beta}$.
		We have 
		\begin{equation*}
			\begin{aligned}
				&\quad\left|\int_{\mathbb{R}_+^2 } N\frac{\partial \lambda}{\partial K}j(\psi)\lambda \xi_m^2 \varphi dKdN\right| 
				\le {\tilde{\alpha}}  \left|\left|\frac{\partial \lambda}{\partial K} \xi_m\right|\right|^2_{L^2_{1,K,\varphi}} + \frac{b_7^2}{2\tilde{\alpha}} ||\lambda \xi_m||_{L^2_{\varphi}}^2 + \frac{1}{2\tilde{\alpha}} ||\lambda \xi_m||_{L^2_{\tilde{\varphi}}}^2.
			\end{aligned}
		\end{equation*}
		Using Young's inequality again, we also have 	for any $\tilde{\alpha}>0$, 
		\begin{equation*}
			\begin{aligned}
				&	\left|\int_{\mathbb{R}_+^2 } A N^{1-\beta}K^{\beta}\frac{\partial \lambda}{\partial K}\lambda \xi_m^2 \varphi dKdN\right| 
				\le A{\tilde{\alpha}}  \left|\left|\frac{\partial \lambda}{\partial K}(t) \xi_m\right|\right|^2_{L^2_{1,K,\varphi}} + \frac{A}{8{\tilde{\alpha}}} ||\lambda\xi_m||^2_{L^2_{\tilde{\varphi}}} + \frac{AC(\beta)}{8{\tilde{\alpha}}} ||\lambda \xi_m||^2_{L^2_{\varphi}};\\
				&\epsilon^2 \left|\int_{\mathbb{R}_+^2 }K\frac{\partial \lambda}{\partial K}\lambda \xi_m^2 \varphi dKdN\right|
				\le   \epsilon^2 \tilde{\alpha}  \left|\left|\frac{\partial \lambda}{\partial K}(t)  \xi_m\right|\right|^2_{L^2_{1,K,\varphi}} + \frac{\epsilon^2}{4{\tilde{\alpha}}} ||\lambda \xi_m||^2_{L^2_{\varphi}}.
			\end{aligned}
		\end{equation*}
		Using Young's inequality and Assumption \ref{Assum:drift rate in population model}, we have
		\begin{equation*}
			\begin{aligned}
				&\quad\left|\int_{\mathbb{R}_+^2 } f(N)\frac{\partial \lambda}{\partial N}\lambda \xi_m^2 \varphi dKdN\right| 
				\le {\tilde{\alpha}}  \left|\left|\frac{\partial \lambda}{\partial N} \xi_m \right|\right|^2_{L^2_{1,N,\varphi}} + \frac{C_f^2}{4\tilde{\alpha}} ||\lambda \xi_m||_{L^2_\varphi}^2;\\
				& \left|\int_{\mathbb{R}_+^2} A\beta N^{1-\beta}K^{\beta-1}\lambda^2 \xi_m^2 \varphi dKdN\right| 
				\le A\beta\left(\tilde{\alpha}||\lambda \xi_m ||_{L^2_\varphi}^2+ \frac{1}{8\tilde{\alpha}}||\lambda \xi_m||_{ L^2_{\tilde{\varphi}}}^2 + \frac{C(\beta)}{8\tilde{\alpha}}||\lambda \xi_m||_{ L^2_{\varphi}}^2\right).
			\end{aligned}
		\end{equation*}
		Next, we estimate the remaining terms of $B[\lambda(t),\lambda(t) \xi_m^2;t]$ as follows. We note that $
		\left|K\frac{\partial \varphi}{\partial K}\right| = \left|\frac{4K^{10}-6}{K^{10}+1}\right|\varphi\le 6\varphi$,
		by Young's inequality, we have
		\begin{equation*}
			\begin{aligned}
				&\quad\left|\int_{\mathbb{R}_+^2 }\frac{1}{2}{\epsilon}^2 \frac{\partial \lambda}{\partial K} \lambda \xi_m^2 \frac{\partial }{\partial K}(K^2\varphi)dKdN\right|\le \int_{\mathbb{R}_+^2 }\left|\frac{1}{2}{\epsilon}^2 \frac{\partial \lambda}{\partial K}\cdot 8K\lambda \xi_m^2 \varphi \right| dKdN\\
				& \le  4\epsilon^2 \left({\tilde{\alpha}}  \left|\left|\frac{\partial \lambda}{\partial K} \xi_m \right|\right|^2_{L^2_{1,K,\varphi}}dt + \frac{1}{4\tilde{\alpha}} ||\lambda \xi_m||_{L^2_\varphi}^2\right).
			\end{aligned}
		\end{equation*}
		Similarly, 
		$\left|N\frac{\partial \varphi}{\partial N} \right|= \left|\frac{8N^{10}-2}{N^{10}+1}\right|\varphi< 8\varphi$,
		by Young's inequality again, we have
		\begin{equation*}
			\begin{aligned}
				&\quad\left| \int_{\mathbb{R}_+^2 }\frac{1}{2}{{{\sigma}^2}}\frac{\partial \lambda}{\partial N} \lambda \xi_m^2\frac{\partial}{\partial N}\left(N^2\varphi \right)dKdN\right| 
				\le {5{\sigma}^2}\left({\tilde{\alpha}}  \left|\left|\frac{\partial \lambda}{\partial N} \xi_m\right|\right|^2_{L^2_{1,N,\varphi}} + \frac{1}{4\tilde{\alpha}} ||\lambda \xi_m||_{L^2_\varphi}^2\right).
			\end{aligned}
		\end{equation*}
		Using the properties of the cutoff functions $\xi_m$, we have
		\begin{equation}\label{eq:estimate for second order term1}
			\begin{aligned}
				&-\int_{\mathbb{R}_+^2 }\frac{1}{2}{\epsilon}^2\frac{\partial \lambda }{\partial K} \frac{\partial \lambda \xi_m^2}{\partial K}K^2\varphi dKdN
				=-\int_{\mathbb{R}_+^2 }\frac{1}{2}{\epsilon}^2\frac{\partial \lambda }{\partial K} \left(\frac{\partial \lambda }{\partial K} \xi_m^2 + 2\lambda  \xi_m \frac{\partial \xi_m}{\partial K}\right) K^2\varphi dKdN\\
				\le & -\int_{\mathbb{R}_+^2 }\frac{1}{2}{\epsilon}^2 \left(\frac{\partial \lambda }{\partial K}\right)^2\xi_m^2  K^2\varphi dKdN + 100{\tilde{\alpha}} \epsilon^2  \left|\left|\frac{\partial \lambda}{\partial K} {\xi_m}\right|\right|^2_{L^2_{1,K,\varphi}}dt + \frac{\epsilon^2}{4\tilde{\alpha}} 
				||\lambda  ||_{L^2_\varphi}^2.
			\end{aligned}
		\end{equation}
		\begin{equation}\label{eq:estimate for second order term2}
			\begin{aligned}
				&-\int_{\mathbb{R}_+^2 }\frac{1}{2}{\sigma}^2\frac{\partial \lambda }{\partial N} \frac{\partial \lambda \xi_m^2}{\partial N}N^2\varphi dKdN\\
				\le & -\int_{\mathbb{R}_+^2 }\frac{1}{2}{\sigma}^2 \left(\frac{\partial \lambda }{\partial N}\right)^2\xi_m^2  N^2\varphi dKdN + 100{\tilde{\alpha}} \sigma^2  \left|\left|\frac{\partial \lambda}{\partial N} {\xi_m}\right|\right|^2_{L^2_{1,N,\varphi}}dt + \frac{\sigma^2}{4\tilde{\alpha}} ||\lambda ||_{L^2_\varphi}^2.
			\end{aligned}
		\end{equation}
		Applying all of the above  estimates to \eqref{eq:testfunc}, we finally obtain
		\begin{equation}
			\begin{aligned}
				 B[\lambda(t),\lambda(t) \xi_m^2;t] 
				&\le -\int_{\mathbb{R}_+^2 }\frac{1}{2}{\epsilon}^2 \left(\frac{\partial \lambda}{\partial K}\right)^2  \xi_m^2 K^2 \varphi dKdN  -\int_{\mathbb{R}_+^2 }\frac{1}{2}{{\sigma}}^2 \left(\frac{\partial \lambda}{\partial N}\right)^2 \xi_m^2  N^2 \varphi dKdN \\
				&\quad+C(A,\epsilon,{\sigma})\cdot\tilde{\alpha} \left( \left|\left|\frac{\partial \lambda}{\partial K} \xi_m \right|\right|^2_{L^2_{1,K,\varphi}}+  \left|\left|\frac{\partial \lambda}{\partial N} \xi_m^2  \right|\right|^2_{L^2_{1,N,\varphi}}\right) \\
					&\quad + C(\tilde{\alpha}, a_1, c_0, b_7, A,C_f,\beta,\epsilon,{\sigma})||\lambda||_{L^2_{\varphi}}^2
			 + C(\tilde{\alpha}, A, \beta) ||\lambda||_{L^2_{\tilde{\varphi}}}^2.
			\end{aligned}
		\end{equation}
		Now, we choose $\tilde{\alpha}$ suitably small such that $
		C(A,\epsilon,{\sigma})\cdot\tilde{\alpha} \le \frac{1}{4}\min{(\epsilon^2,{\sigma}^2)},
		$
		then we  get 
		\begin{equation*}
			\begin{aligned}
				 B[\lambda(t),\lambda(t) \xi_m^2 ;t] 
				&\le -\int_{\mathbb{R}_+^2 }\frac{1}{4}{\epsilon}^2 \left(\frac{\partial \lambda}{\partial K}\right)^2 \xi_m^2 K^2 \varphi dKdN  -\int_{\mathbb{R}_+^2 }\frac{1}{4}{{\sigma}}^2 \left(\frac{\partial \lambda}{\partial N}\right)^2 \xi_m^2  N^2 \varphi dKdN \\
				&\quad+C( a_1, c_0, b_7, A,C_f,\beta,\epsilon, {\sigma})||\lambda||_{L^2_{\varphi}}^2 + C(A, \beta, \epsilon, {\sigma}) ||\lambda||_{L^2_{\tilde{\varphi}}}^2\\
				&\le -\frac{1}{4}\min{(\epsilon^2,{\sigma}^2)}\int_{\mathbb{R}_+^2 } \left( \left(\frac{\partial \lambda}{\partial K}\right)^2 \xi_m^2 K^2 + \left(\frac{\partial \lambda}{\partial N}\right)^2 \xi_m^2  N^2 \right)\varphi dKdN \\
				&\quad +C(a_1, c_0, b_7, A,C_f,\beta,\epsilon,{\sigma})||\lambda||_{L^2_{\varphi}}^2 + C(A, \beta, \epsilon, {\sigma}) ||\lambda||_{L^2_{\tilde{\varphi}}}^2.
			\end{aligned}
		\end{equation*}	
		We note that by Lions–Magenes lemma  $
		\frac{1}{2}\frac{d}{dt}||\lambda(t) \xi_m||_{L^2_{\varphi}}^2=\langle\lambda_t,\lambda\xi_m^2\rangle_{Y_0, Y_1},
		$
		and hence,
		\begin{equation*}
			\begin{aligned}
				0&=\frac{d}{dt}||\lambda(t) \xi_m||_{L^2_{\varphi}}^2 +B[\lambda(t),\lambda(t) \xi_m^2 ;t]  \\
				&\le \frac{d}{dt}||\lambda(t) \xi_m||_{L^2_{\varphi}}^2-\frac{1}{4}\min{(\epsilon^2,{\sigma}^2)}\int_{\mathbb{R}_+^2 } \left( \left(\frac{\partial \lambda}{\partial K}\right)^2 \xi_m^2 K^2 + \left(\frac{\partial \lambda}{\partial N}\right)^2 \xi_m^2  N^2 \right)\varphi dKdN \\
				&\quad +C(a_1, c_0, b_7, A,C_f,\beta,\epsilon,{\sigma})||\lambda||_{L^2_{\varphi}}^2 + C(A, \beta, \epsilon, {\sigma}) ||\lambda||_{L^2_{\tilde{\varphi}}}^2.
			\end{aligned}
		\end{equation*}
		Integrating from $t$ to $T$ with respect to time $s$, we obtain  
		\begin{equation}
			\begin{aligned}
				&||\lambda(t) \xi_m||_{L^2_{\varphi}}^2 + \frac{1}{4}\min{(\epsilon^2,{\sigma}^2)} \int_{t}^{T}\int_{\mathbb{R}_+^2 } \left( \left(\frac{\partial \lambda}{\partial K}\right)^2 \xi_m^2 K^2 + \left(\frac{\partial \lambda}{\partial N}\right)^2 \xi_m^2  N^2 \right)\varphi dKdN ds \\
				\le& \left|\left|u'_2\left(\frac{K}{N}\right) \xi_m \right|\right|_{L^2_{\varphi}}^2 + C(a_1, c_0, b_7, A,C_f,\beta,\epsilon,{\sigma}) \int_{t}^{T}||\lambda||_{L^2_{\varphi}}^2 ds 
			+ C(A, \beta, \epsilon, {\sigma}) \int_{t}^{T} ||\lambda||_{L^2_{\tilde{\varphi}}}^2 ds.
			\end{aligned}
		\end{equation}
		Now we can pass to the limit as $m\to \infty$, and then let $t=0$ to obtain 
		\begin{equation*}
			\begin{aligned}
				&\quad\frac{1}{4}\min{(\epsilon^2,{\sigma}^2)} \int_{0}^{T}\int_{\mathbb{R}_+^2 } \left( \left(\frac{\partial \lambda}{\partial K}\right)^2  K^2 + \left(\frac{\partial \lambda}{\partial N}\right)^2   N^2 \right)\varphi dKdN ds\\
				&\le \left|\left|u'_2\left(\frac{K}{N}\right) \right|\right|_{L^2_{\varphi}}^2 + C(a_1, c_0, b_7, A,C_f,\beta,\epsilon,{\sigma}) \int_{0}^{T}||\lambda||_{L^2_{\varphi}}^2 ds  + C(A, \beta, \epsilon, {\sigma}) \int_{0}^{T} ||\lambda||_{L^2_{\tilde{\varphi}}}^2 ds.
			\end{aligned}
		\end{equation*}
		Finally, by combining Lemma \ref{lem: { L^2_varphi} estimate} and noting the dependence of $c_0$, $b_7$ and $C_7$ on the given parameters, we obtain Equation \eqref{eq:L^2([0,T];Y_1) estimate}.
	\end{proof}
	
	\begin{proof}[Proof of Claim \ref{claim: the solution map is continuous}]\label{proof: proof of claim the solution map is continuous}
		We must  show that $\lambda$ is a weak solution to the linearized problem  \eqref{eq: differential equation for $lambda$} with the input $\psi$. For this purpose, the weak convergence of $\{\lambda_n\}$ in  $\mathcal{H}^1_{\varphi}$ will suffice, that is we consider the sequence $\{{\lambda}_n\}$ converges to $\lambda$ in  $L^2([0,T];Y_1)$ weakly and $\left\{\frac{\partial {\lambda}_n}{\partial t}\right\}$ converges to $\frac{\partial \lambda}{\partial t}$ in $L^2([0,T];Y_0)$ in weak*-sense.  
		We have for any $\phi \in Y_1$ and $\phi_1(t)\in C_c^{\infty}(0,T)$ (so that $\phi \phi_1(t)  \in L^2([0,T];Y_1)$) that 
		\begin{equation}
			\left\{
			\begin{aligned}
				&\int_{0}^{T}\left\langle\frac{\partial \lambda_n}{\partial t},\phi\right\rangle_{Y_0, Y_1} \phi_1(t) dt +\int_{0}^{T} \int_{\mathbb{R}_+^2 }\left(- N\frac{\partial \lambda_n}{\partial K}j(\psi_n)  +  A N^{1-\beta}K^{\beta}\frac{\partial \lambda_n}{\partial K} +{\epsilon}^2K\frac{\partial \lambda_n}{\partial K}\right.\\
				&\qquad\left.+ f(N)\frac{\partial \lambda_n}{\partial N} +  A\beta N^{1-\beta}K^{\beta-1}\lambda_n\right) \phi  \phi_1(t) \varphi dKdNdt \\
				& ~~-\int_{0}^{T} \int_{\mathbb{R}_+^2 }\frac{1}{2}{\epsilon}^2\frac{\partial \lambda_n}{\partial K}\left(\frac{\partial \phi}{\partial K}K^2\varphi + \phi\frac{\partial (K^2\varphi)}{\partial K}\right) \phi_1(t) dKdNdt  \\
				&~~- \int_{0}^{T} \frac{1}{2}{{\sigma}}^2\frac{\partial \lambda_n}{\partial N}\left(\frac{\partial \phi}{\partial N}N^2\varphi + \phi\frac{\partial (N^2\varphi)}{\partial N}\right) \phi_1(t) dKdNdt 
				= 0 
				\\
				&\lambda_n(T,K,N)  = u_2'\left(\frac{K}{N}\right).
			\end{aligned}
			\right.
		\end{equation}
		By the  weak convergence of $\{\lambda_n\}$ to $\lambda$ in $L^2([0,T];Y_1)$, one can check quickly   that
		all the terms except for $- N\frac{\partial \lambda_n}{\partial K}j(\psi_n) \phi  \phi_1(t) \varphi$ will converge accordingly.
		For example, 
		as 
		\begin{equation*}
			\begin{aligned}
				&\quad \int_{0}^{T} \int_{\mathbb{R}_+^2 } A N^{1-\beta}K^{\beta}\frac{\partial \lambda}{\partial K} \phi  \phi_1(t) \varphi dKdNdt 
				\\
				&\le   A\left( \int_{0}^{T} \int_{\mathbb{R}_+^2 } \left(\frac{\partial \lambda}{\partial K} \right)^2 K^2 \varphi dKdNdt \right)^{\frac{1}{2}} \left( \int_{0}^{T} \int_{\mathbb{R}_+^2 } N^{2-2\beta} K^{2\beta -2} \phi^2  \phi_1^2(t)  \varphi dKdNdt \right)^{\frac{1}{2}} \\
				&\le    A||\phi \phi_1||_{L^2([0,T];Y)} ||\lambda||_{L^2([0,T];Y_1)},
			\end{aligned}
		\end{equation*}
the left-hand side above can be viewed as a linear functional on $L^2([0,T];Y_1)$,  		then by the weak convergence of $\{\lambda_n\}$ in $L^2([0,T];Y_1)$,  we have, as $n \to \infty$,
		\begin{equation*}
			\begin{aligned}
				\int_{0}^{T} \int_{\mathbb{R}_+^2 } A N^{1-\beta}K^{\beta}\frac{\partial \lambda_n}{\partial K} \phi  \phi_1(t) \varphi dKdNdt \to \int_{0}^{T} \int_{\mathbb{R}_+^2 } A N^{1-\beta}K^{\beta}\frac{\partial \lambda}{\partial K} \phi  \phi_1(t) \varphi dKdNdt.
			\end{aligned}
		\end{equation*}
		We are left to show, as $n \to \infty$,
		\begin{equation}\label{eq:convergence of nonlinear term}
			\int_{0}^{T} \int_{\mathbb{R}_+^2 }- N\frac{\partial \lambda_n}{\partial K}j(\psi_n)\phi\varphi \phi_1(t) dKdNdt \to 		\int_{0}^{T} \int_{\mathbb{R}_+^2 }- N\frac{\partial \lambda}{\partial K}j(\psi)\phi\varphi \phi_1(t) dKdN dt.  
		\end{equation}
		If so, then
		$\int_{0}^{T} \left(\langle\lambda_t, \phi\rangle_{Y_0, Y_1}  +  B[\lambda(t),\phi;t] \right)\phi_1(t) dt  =0$, 
		for all $\phi_1(t) \in C_c^{\infty}(0,T)$, and hence, 
		$\langle\lambda_t, \phi\rangle_{Y_0, Y_1}   + B[\lambda(t),\phi;t]   =0, ~ \textrm{for a.e. $t \in [0,T]$}.$
		Without loss of generality, we may assume that $\psi_n \to \psi$ a.e.. Then 
		\begin{equation}\label{eq:mainerror}
			\begin{aligned}
				&\left| \int_{0}^{T}  \int_{\mathbb{R}_+^2 }- N\frac{\partial \lambda_n}{\partial K}j(\psi_n) 
				\phi \phi_1(t) \varphi dKdNdt -	\int_{0}^{T}  \int_{\mathbb{R}_+^2 }- N\frac{\partial \lambda}{\partial K}j(\psi) \phi \phi_1(t) \varphi dKdN dt \right|\\
				= &\left| \int_{0}^{T}  \int_{\mathbb{R}_+^2} Nj(\psi)\left(\frac{\partial \lambda}{\partial K}-\frac{\partial \lambda_n}{\partial K}\right) \phi \phi_1(t) \varphi + N\left(j(\psi)- j(\psi_n)\right)\frac{\partial\lambda_n}{\partial K}\phi \phi_1(t) \varphi dKdN dt\right|\\
				\le& \left| \int_{0}^{T}  \int_{\mathbb{R}_+^2} Nj(\psi)\left(\frac{\partial \lambda}{\partial K}-\frac{\partial \lambda_n}{\partial K}\right) \phi \phi_1(t) \varphi dKdN dt\right| 
				+ \left|\int_{0}^{T}  \int_{\mathbb{R}_+^2}N\left(j(\psi)- j(\psi_n)\right)\frac{\partial\lambda_n}{\partial K}\phi \phi_1(t)  \varphi dKdN dt\right|.
			\end{aligned}
		\end{equation}
		For the first term, when $\beta \ge \gamma$,  using the fact that  $j(\psi)\frac{N}{K} \le \frac{a_1}{c_0}$, we have  
		\begin{equation*}
			\begin{aligned}
				&\quad \left|\int_{0}^{T}  \int_{\mathbb{R}_+^2} Nj(\psi)\frac{\partial \lambda}{\partial K} \phi \phi_1(t)  \varphi dKdNdt\right| \\
				&\le \frac{a_1}{c_0}\left(\int_{0}^{T}  \int_{\mathbb{R}_+^2}\left(\frac{\partial \lambda}{\partial K}\right)^2K^2 \varphi dKdNdt\right)^{\frac{1}{2}} \left(\int_{0}^{T}  \int_{\mathbb{R}_+^2}\phi^2 \phi^2_1(t) \varphi dKdNdt\right)^{\frac{1}{2}},
			\end{aligned}
		\end{equation*}
		then due to the weak convergence of $\{\lambda_n\}$ to $\lambda$, $\int_{0}^{T}  \int_{\mathbb{R}_+^2} Nj(\psi)\left(\frac{\partial \lambda}{\partial K}-\frac{\partial \lambda_n}{\partial K}\right) \phi \phi_1(t) \varphi dKdN dt$ tends to 0. 
		When $\beta<\gamma$, the same conclusion holds by noting that using \eqref{eq:estimate for j(psi)}, we have 
		\begin{equation*}
			\begin{aligned}
				&\quad\left|\int_{0}^{T} \int_{\mathbb{R}_+^2 } N j(\psi) \frac{\partial \lambda}{\partial K} \phi \phi_1(t) \varphi dKdN dt \right| 
				\le \int_{0}^{T} \int_{\mathbb{R}_+^2 } \left(b_7+  \left(\frac{N}{K}\right)^{1-\beta}\right) \left|K\frac{\partial \lambda}{\partial K} \phi \phi_1(t) \right| \varphi dKdN dt \\ 
				& \le  b_7\left(\int_{0}^{T}  \int_{\mathbb{R}_+^2}\left(\frac{\partial \lambda}{\partial K}\right)^2K^2 \varphi dKdNdt\right)^{\frac{1}{2}} \left(\int_{0}^{T}  \int_{\mathbb{R}_+^2}\phi^2 \phi^2_1(t) \varphi dKdNdt\right)^{\frac{1}{2}}\\
				& \quad + \left(\int_{0}^{T}  \int_{\mathbb{R}_+^2}\left(\frac{\partial \lambda}{\partial K}\right)^2K^2 \varphi dKdNdt\right)^{\frac{1}{2}} \left(\int_{0}^{T}  \int_{\mathbb{R}_+^2}\phi^2 \phi^2_1(t) \left(\frac{N}{K}\right)^{2-2\beta} \varphi dKdNdt\right)^{\frac{1}{2}}\\
				&\leq C ||\lambda||_{L^2([0,T]; Y_1)}||\phi \phi_1(t)||_{L^2([0,T]; Y)}.
			\end{aligned}
		\end{equation*}
		
		For the second term in the right-hand side of \eqref{eq:mainerror}, we first note that by the mean value theorem,
		there exists a real number $\theta$ between $\psi_n$ and $\psi$ such that $\left|j(\psi)- j(\psi_n)\right|=|j'(\theta)| |\psi-\psi_n|$. 
		When $\beta\ge \gamma$, $\theta \ge c_0^{\gamma} \left(\frac{N}{K}\right)^{\gamma}$ as $\psi_n$, $\psi$$\ge c_0^{\gamma} \left(\frac{N}{K}\right)^{\gamma}$. Using conditions \eqref{eq:another condition on j(x)} on $j'(x)$, we have  $\left|j(\psi)- j(\psi_n)\right| \le  (\tilde{a}_1 \theta^{-1-\frac{1}{\gamma}} + \tilde{\tilde{a}}_1) |\psi-\psi_n| \le\tilde{a}_1 {c_0}^{-1-\gamma}N^{-1-\gamma}K^{1+\gamma}|\psi-\psi_n| + {\tilde{\tilde{a}}}_1|\psi-\psi_n|$. Therefore, 
		\begin{equation*}
			\begin{aligned}
				&\quad\left| \int_{0}^{T} \int_{\mathbb{R}_+^2}N\left(j(\psi)- j(\psi_n)\right)\frac{\partial\lambda_n}{\partial K}\phi \phi_1(t) \varphi dKdNdt\right|\\
				&\le \tilde{a}_1 {c_0}^{-1-\gamma} \left(\int_{0}^{T} \int_{\mathbb{R}_+^2}\left(\frac{\partial {\lambda}_n}{\partial K}\right)^2K^2\varphi dKdNdt\right)^{\frac{1}{2}}
				\cdot \left(\int_{0}^{T} \int_{\mathbb{R}_+^2}{N^{-2\gamma}}{K^{2\gamma}}|\psi-\psi_n|^2\phi^2 \phi^2_1(t)  \varphi dKdNdt\right)^{\frac{1}{2}}\\
				& \quad+  {\tilde{\tilde{a}}}_1 \left(\int_{0}^{T} \int_{\mathbb{R}_+^2}\left(\frac{\partial {\lambda}_n}{\partial K}\right)^2K^2\varphi dKdNdt\right)^{\frac{1}{2}}
				\cdot \left(\int_{0}^{T} \int_{\mathbb{R}_+^2}{N^{2}}{K^{-2}}|\psi-\psi_n|^2\phi^2 \phi^2_1(t) \varphi dKdNdt\right)^{\frac{1}{2}}.
			\end{aligned}
		\end{equation*}
		By weak convergence (or by Theorem \ref{thm:L^2([0,T];Y_1) estimate}), $\int_{0}^{T} \int_{\mathbb{R}_+^2}\left(\frac{\partial {\lambda}_n}{\partial K}\right)^2K^2\varphi dK$ $dNdt$ is  uniformly bounded and 
		\begin{equation*}
			\begin{aligned}
				&\quad\int_{0}^{T} \int_{\mathbb{R}_+^2}{N^{-2\gamma}}{K^{2\gamma}}|\psi-\psi_n|^2\phi^2 \phi^2_1(t) \varphi dKdNdt
				\le 2\int_{\mathbb{R}_+^2}{N^{-2\gamma}}{K^{2\gamma}}(\psi^2+\psi_n^2)\phi^2 \phi^2_1(t) \varphi dKdNdt\\
				&\le 8\int_{0}^{T} \int_{\mathbb{R}_+^2}{N^{-2\gamma}}{K^{2\gamma}}\left(C_1^2\left(\frac{N}{K}\right)^{2\gamma} + C_2^2\left(\frac{N}{K}\right)^{2\beta}\right)\phi^2 \phi^2_1(t) \varphi dKdNdt\\
				&= 8\int_{0}^{T} \int_{\mathbb{R}_+^2}\left(C_1^2+ C_2^2\left(\frac{N}{K}\right)^{2\beta-2\gamma}\right)\phi^2 \phi^2_1(t)\varphi dKdNdt
				< (8C_1^2 + 8C(\beta,\gamma)C_2^2)||\phi\phi_1(t) ||_{L^2([0,T];Y)}^2.
			\end{aligned}
		\end{equation*}
		In the above, the last inequality follows from the fact that $0 \le 2\beta-2\gamma < 2+2\beta$
		and hence, there exists a constant $C(\beta, \gamma)$ such that $\left(\frac{N}{K}\right)^{2\beta-2\gamma}\varphi$ $<\left(C(\beta, \gamma) + \left(\frac{N}{K}\right)^{2+2\beta}\right)$$\varphi$$ \le C(\beta, \gamma) (\varphi + \tilde{\varphi})$.
		Similarly 
		\begin{equation*}
			\begin{aligned}
				& \quad\int_{0}^{T} \int_{\mathbb{R}_+^2}{N^{2}}{K^{-2}}|\psi-\psi_n|^2\phi^2 \phi^2_1(t)\varphi dKdNdt
				\le 2\int_{\mathbb{R}_+^2}{N^{2}}{K^{-2}}(\psi^2+\psi_n^2)\phi^2 \phi^2_1(t)\varphi dKdNdt\\
				& \le  8\int_{0}^{T} \int_{\mathbb{R}_+^2}{N^{2}}{K^{-2}}\left(C_1^2\left(\frac{N}{K}\right)^{2\gamma} + C_2^2\left(\frac{N}{K}\right)^{2\beta}\right)\phi^2 \phi^2_1(t)\varphi dKdNdt\\
				&= 8\int_{0}^{T} \int_{\mathbb{R}_+^2}\left(C_1^2 \left(\frac{N}{K}\right)^{2+2\gamma}+ C_2^2\left(\frac{N}{K}\right)^{2+2\beta}\right)\phi^2 \phi^2_1(t)\varphi dKdNdt
				< 8(C(\beta, \gamma)C_1^2+C_2^2)||\phi\phi_1(t)||_{L^2([0,T];Y)}^2.
			\end{aligned}
		\end{equation*}		
		Now we can  apply Lebesgue's dominated  convergence theorem to conclude that $\int_{0}^{T} \int_{\mathbb{R}_+^2}$  $N\left(j(\psi)- j(\psi_n)\right)$ $\frac{\partial\lambda_n}{\partial K}$$\phi \phi_1(t) $$\varphi $$dKdNdt$$\to 0$.
		
		 If $\beta < \gamma$, $\left|j(\psi)- j(\psi_n)\right|\le \tilde{a}_1 \left(C_3 \left(\frac{N}{K}\right)^{\beta}\left(  C_4 \left(\frac{N}{K}\right)^{\beta-1} + C_5 \right)^{ \frac{\gamma-\beta}{\beta-1 } } \right)^{-\frac{\gamma_+1}{\gamma}} |\psi-\psi_n|$ $ + {\tilde{\tilde{a}}}_1|\psi-\psi_n|$. We only need to consider the first term; we have
		\begin{equation*}
			\begin{aligned}
				&~\left(C_3 \left(\frac{N}{K}\right)^{\beta}\left(  C_4 \left(\frac{N}{K}\right)^{\beta-1} + C_5 \right)^{ \frac{\gamma-\beta}{\beta-1 } }  \right)^{-\frac{\gamma_+1}{\gamma}} \\
				&= C_3^{-\frac{\gamma_+1}{\gamma}} \left(\frac{N}{K}\right)^{-\frac{\beta(\gamma +1)}{\gamma}}
				\left(	\left(  C_4 \left(\frac{N}{K}\right)^{\beta-1} + C_5 \right)^{ \frac{\gamma-\beta}{\gamma (1-\beta) } }  \right)^{\gamma+1}\\
				 &\leq C_3^{-\frac{\gamma_+1}{\gamma}} \left(\frac{N}{K}\right)^{-\frac{\beta(\gamma +1)}{\gamma}}
				\left(	C_4^{\frac{\gamma -\beta}{\gamma (1-\beta)}}\left(\frac{N}{K}\right)^{-\frac{\gamma -\beta}{\gamma}}  + C_5^{\frac{\gamma -\beta}{\gamma (1-\beta)}} \right)^{\gamma+1}\\
				&\leq C \left(\left(\frac{N}{K}\right)^{-\gamma-1} + \left(\frac{N}{K}\right)^{-\frac{\beta(\gamma +1)}{\gamma}}\right)
				\le  C'+ C'\left(\frac{N}{K}\right)^{-\gamma-1},
			\end{aligned}
		\end{equation*}
		for some constants $C$  and $C'$ depending on $C_3$, $C_4$, $C_5$, $\beta$ and $\gamma$. Then, we can use 
		the same argument  above for the case $\beta\geq \gamma$ to show that
		$\int_{0}^{T} \int_{\mathbb{R}_+^2}N\left(j(\psi)- j(\psi_n)\right)\frac{\partial\lambda_n}{\partial K}$ $\phi \phi_1(t) \varphi dK$$dNdt$ $\to$ 0
		by noting that $\psi_n$, $\psi$ $\le C_6 \left(\frac{N}{K}\right)^{\gamma}$ and $\tilde{\varphi} = (N^{4-4\beta}K^{4\beta-4}  + N^{2+2\gamma}K^{-2-2\gamma})\varphi$  when $\beta < \gamma$. We have proven \eqref{eq:convergence of nonlinear term}, hence  $\lambda$ is a weak solution to the linearized problem \eqref{eq: differential equation for $lambda$} with the input $\psi$.
	\end{proof}
	
	\bibliographystyle{plain} 
	\bibliography{ref} 
	
	
	

\end{document}